\documentclass[11pt]{amsart}
\usepackage{mathtools, amssymb,amsfonts,amsthm,amscd,amsmath,amsgen,upref,enumerate,nicefrac,esint}

\usepackage[colorinlistoftodos,prependcaption,color=yellow,textsize=tiny]{todonotes}

\usepackage[font=small,skip=0pt]{caption}

\usepackage[margin=1in]{geometry}

\theoremstyle{plain}
\newtheorem{cor}{Corollary}
\newtheorem{lem}[cor]{Lemma}
\newtheorem{prop}[cor]{Proposition}

\newtheorem{thm}[cor]{Theorem}
\newtheorem*{thm*}{Theorem}
\theoremstyle{definition}
\newtheorem{definition}[cor]{Definition}
\newtheorem{remark}[cor]{Remark}

\newtheorem{assumption}[cor]{Assumption}

\numberwithin{cor}{section}
\numberwithin{equation}{section}

\DeclareMathOperator*{\essinf}{ess\,inf}
\DeclareMathOperator{\C}{C}

\newcommand{\E}{\mathbb{E}}

\newcommand{\esssup}{\text{ess sup}}

\renewcommand{\d}{\delta}

\renewcommand{\and}{\quad\textrm{ and }\quad}

\renewcommand{\P}{\mathbb{P}}
\renewcommand{\a}{\alpha}
\renewcommand{\b}{\beta}
\renewcommand{\k}{\kappa}

\renewcommand{\O}{\Omega}

\newcommand{\sgn}{\text{sgn}}

\newcommand{\F}{\mathcal F}

\newcommand{\R}{\mathbb{R}}
\newcommand{\TT}{\mathbb{T}}

\newcommand{\N}{\mathbb{N}}
\newcommand{\Z}{\mathbb{Z}}

\newcommand{\norm}[1]{\left\| #1 \right\|}

\newcommand{\ve}{\varepsilon}

\newcommand{\Sum}{\sum\limits}

\newcommand{\abs}[1]{| #1 |}
\providecommand{\ud}[1]{\, \mathrm{d} #1}
\providecommand{\dx}{\ud{x}}
\providecommand{\dy}{\ud{y}}
\providecommand{\dxi}{\ud \xi}
\providecommand{\deta}{\ud{\eta}}
\providecommand{\dr}{\ud{r}}

\providecommand{\dxip}{\ud{\xi'}}

\providecommand{\ds}{\ud{s}}
\providecommand{\dt}{\ud{t}}

\providecommand{\dd}{\ud}

\newcommand{\eps}{\varepsilon}

\def\XXint#1#2#3{{\setbox0=\hbox{$#1{#2#3}{\int}$ }
\vcenter{\hbox{$#2#3$ }}\kern-.6\wd0}}

\usepackage{graphicx}
\usepackage{color}

\begin{document}

\title[SPDE and fluctuations of the SSEP]{Conservative stochastic PDE and fluctuations of the symmetric simple exclusion process}


\date{\today}

\author{Nicolas Dirr}
\address{School of Mathematics, University of Cardiff, CF24 4AG Cardiff, United Kingdom}
\email{DirrNP@cardiff.ac.uk}

\author{Benjamin Fehrman}
\address{Department of Mathematics, Louisiana State University, Louisiana 70802, United States}
\email{fehrman@math.lsu.edu}

\author{Benjamin Gess}
\address{Max Planck Institute for Mathematics in the Sciences, 04103 Leipzig, Germany \newline \indent Fakult\"at f\"ur Mathematik, Universit\"at Bielefeld, 33615 Bielefeld, Germany}
\email{Benjamin.Gess@mis.mpg.de}

\begin{abstract}  In this paper, we provide a continuum model for the fluctuations of the symmetric simple exclusion process about its hydrodynamic limit.  The model is based on an approximating sequence of stochastic PDEs with nonlinear, conservative noise.  In the small-noise limit, we show that the fluctuations of the solutions are to first-order the same as the fluctuations of the particle system.  Furthermore, the SPDEs correctly simulate the rare events in the particle process.  We prove that the solutions satisfy a zero-noise large deviations principle with rate function equal to that describing the deviations of the symmetric simple exclusion process.  \end{abstract}

\maketitle

\section{Introduction}

In this work we propose a nonlinear, conservative stochastic PDE as a continuum model incorporating fluctuation corrections for the symmetric simple exclusion process (SSEP).   More precisely, for every $N\in\N$ let $\eta^N_{t}(x)$ be the SSEP on the discrete torus $\Z^{d}/N\Z^d$ with slowly varying initial state $\overline{\rho}_0(\nicefrac{x}{N})$ for some $\overline{\rho}_0\in\C^4(\TT^d;[0,1])$ (for details see Kipnis and Landim \cite[Chapter~2]{KipLan1999}). Then, the parabolically rescaled empirical measures
\[\pi_{N}:=\frac{1}{N^d}\sum_{x\in(\Z^{d}/N\Z^d)}\delta_{\frac{x}{n}}\eta_{N^{2}t}(x),\]
defined on the unit torus $\TT^d$, converge in probability as $N\rightarrow\infty$ to a deterministic measure that is absolutely continuous with respect to the Lebesgue measure, with density $\bar{\rho}$ solving
\begin{equation}\label{intro_heat}\partial_{t}\bar{\rho}	=\Delta\bar{\rho}\;\;\textrm{in}\;\;\TT^d\times(0,T)\;\;\textrm{with}\;\;
\bar{\rho}(\cdot,0)=\bar{\rho}_{0}.\end{equation}
In this way, solutions to the heat equation \eqref{intro_heat} describe the SSEP dynamics up to zeroth order.  To reach higher-order continuum approximations, it is necessary to incorporate fluctuations present in $\pi_{N}$. The non-equilibrium central limit fluctuations of $\pi_{N}$ have been analyzed in $d=1$ by Calves, Kipnis, and Spohn \cite{CalKipSpo} and in $d\geq 2$ by Ravishankar \cite{Rav1992}, where it is shown that the measures
\[m_N:=N^{\frac{1}{2}}(\pi_{N}-\E\pi_{N}),\]
converge as $N\rightarrow\infty$ to the generalized Ornstein--Uhlenbeck process $m$ solving
\begin{equation}\label{intro_gen_ou}\partial_t m = \Delta m - \nabla\cdot (\sqrt{\overline{\rho}(1-\overline{\rho})}\xi)\;\;\textrm{in}\;\;\TT^d\times(0,T)\;\;\textrm{with}\;\;m=0\;\;\textrm{on}\;\;\TT^d\times\{0\},\end{equation}
for $\xi$ a $d$-dimensional space-time white noise and for $\overline{\rho}$ solving \eqref{intro_heat}.  Since solutions of \eqref{intro_gen_ou} are not function-valued, the convergence is in distribution on the space $D([0,\infty);\mathcal{D}'(\TT^d))$, see \cite{Rav1992} and Jara and Landim \cite[Theorem~2.2]{JarLan2006}.

Introducing the fluctuation corrected continuum model $\bar{\rho}_{N}:=\overline{\rho}\dx+\frac{1}{\sqrt{N}}m$ yields
\begin{equation}\label{intro_first_order_1}\pi_{N}=\bar{\rho}_{N}+(\E\pi_N-\overline{\rho}\dx)+o(\nicefrac{1}{\sqrt{N}}).\end{equation}
It is shown in \cite{JarLan2006} that $\E\pi_N$ is a solution to the discrete heat equation, and the estimates of \cite[Theorem~A.1]{JarLan2006} that compare $\E\pi_N$ to the solution of \eqref{intro_heat} prove with \eqref{intro_first_order_1} that we have the first order expansion
\begin{equation}\label{intro_first_order}\pi_{N}=\bar{\rho}_{N}+o(\nicefrac{1}{\sqrt{N}}).\end{equation}

However, while it follows from \eqref{intro_first_order} that the continuum model $\overline{\rho}_N$ correctly describes to first-order the central limit fluctuations of the $\pi_N$, the rare events of $\pi_{N}$ are not correctly captured by the affine linear expansion $\bar{\rho}_{N}$.  Indeed, in Quastel, Rezakhanlou, and Varadhan \cite{QuaRezVar1999} and \cite[Chapter~10]{KipLan1999} it has been shown that the dynamical (equilibrium) large deviations for $\pi_{N}$ are described in terms of the rate function
\begin{equation}\label{intro_rate_function}I_{\rho_0}(\mu)=\frac{1}{2}\inf\big\{\norm{g}^2_{L^2}\colon\,\mu=\rho\,dx,\,\partial_{t}\rho=\Delta\rho-\nabla\cdot(\sqrt{\rho(1-\rho)}g)\;\;\textrm{with}\;\;\rho(\cdot,0)=\rho_0\big\}.\end{equation}
In contrast $\overline{\rho}_N$ solves the linear equation, for the solution $\overline{\rho}$ of \eqref{intro_heat},
\begin{equation}\label{intro_rho_N}\partial_t \overline{\rho}_N = \Delta \overline{\rho}_N - \frac{1}{\sqrt{N}}\nabla\cdot (\sqrt{\overline{\rho}(1-\overline{\rho})}\xi)\;\;\textrm{in}\;\;\TT^d\times(0,T)\;\;\textrm{with}\;\;\overline{\rho}_N=\rho_0\;\;\textrm{on}\;\;\TT^d\times\{0\},\end{equation}
and, as follows formally from \eqref{intro_rho_N}, the rate function associated to rare events of $\bar{\rho}_{N}$ is given by
\[J_{\rho_0}(\mu)=\frac{1}{2}\inf\big\{\norm{g}^2_{L^2}\colon\,\mu=\rho\,dx,\,\partial_{t}\rho=\Delta\rho-\nabla\cdot(\sqrt{\bar{\rho}(1-\bar{\rho})}g)\;\;\textrm{with}\;\; \rho(\cdot,0)=\rho_0\big\}.\]

Motivated by work of the second two authors in the context of the zero range process \cite{FehGes19}, and by connections to macroscopic fluctuation theory (see, for example, Bertini, De Sole, Gabrielli, Jona Lasinio, and Landim \cite{BDSGJLL15} and Derrida \cite{D07}) and fluctuating hydrodynamics (see, for example, Hohenberg and Halperin \cite{HH77}, Landau and Lifshitz \cite{LL87}, Spohn \cite{S91}, Bouchet, Gaw\c edzki, and Nardini \cite{BGN16}), in this work, we introduce the nonlinear, stochastic PDE
\begin{equation}\label{intro_approx_eq}\partial_{t}\rho_{N}=\Delta\rho_{N}-\frac{1}{\sqrt{N}}\nabla\cdot\Big(\sqrt{\rho_N(1-\rho_N)}\circ\xi^{K(N)}\Big),\end{equation}
for spectral approximations $\{\xi^K\}_{K\in\N}$  of $\R^d$-valued space-time white noise.  We prove that, under an appropriate $N\rightarrow\infty$, $K(N)\rightarrow\infty$ scaling, the solutions $\rho^{N}$ not only yield a first order expansion
\[\pi_{N}=\rho_{N}+o(\nicefrac{1}{\sqrt{N}}),\]
analogous to \eqref{intro_first_order}, but also display the correct rare event behavior, in the sense that the large deviations for the $\rho_{N}$ are described by the rate function \eqref{intro_rate_function}.  That is, as $N\rightarrow\infty$,
\[\P[\rho_{N}\in A]\approx e^{-N\inf_{\mu\in A}I_{\rho_0}(\mu)}.\]
The introduction of spatially correlated noise in \eqref{intro_approx_eq} is done in order to obtain function-valued solutions.  This fact is already present in the simpler equation \eqref{intro_gen_ou}, for which solutions exist only in a space of distributions and stand in contrast to the large deviations of the exclusion process \cite{KipLan1999,QuaRezVar1999}, which are supported on a space of functions.

We will henceforth let $\ve\in(0,1)$ play the role of $\nicefrac{1}{N}$ to emphasize the fact that the results apply to the full continuous limit $\ve\rightarrow 0$, and we will consider the solutions of the equation
\begin{equation}\label{intro_fd_eq} \partial_t\rho^\ve = \Delta\rho^\ve - \sqrt{\ve}\nabla\cdot\Big(\sqrt{\rho^\ve(1-\rho^\ve)}\circ\xi^{K(\ve)}\Big)\;\;\textrm{in}\;\;\TT^d\times(0,T)\;\;\textrm{with}\;\;\rho^\ve(\cdot,0)=\rho_0,\end{equation}
where $\xi^{K(\ve)}$ is a spectral approximation of space-time white noise defined in Section~\ref{sec_noise}, and where $\circ$ denotes Stratonovich integration.  We will write \eqref{intro_fd_eq} in its It\^o formulation
\begin{equation}\label{intro_Ito} \partial_t\rho^\ve = \Delta\rho^\ve - \sqrt{\ve}\nabla\cdot\Big(\sqrt{\rho^\ve(1-\rho^\ve)}\xi^{K(\ve)}\Big)+\frac{\ve \langle \xi^{K(\ve)}\rangle}{8}\Delta\Theta(\rho^\ve),\end{equation}
for the spatially constant quadratic variation $\langle \xi^{K(\ve)}\rangle$ of the probabilistically stationary noise $\xi^{K(\ve)}$, and for $\Theta\in \C^1(0,1)$ satisfying $\Theta'(\xi) = \frac{(1-2\xi)^2}{\xi(1-\xi)}$.  Equation \eqref{intro_Ito} demonstrates two fundamental difficulties in treating \eqref{intro_fd_eq}.  The first is the irregularity appearing in the noise coefficient due to the singularities of the square-root and its derivative near $\rho^\ve\simeq 0$ and $\rho^\ve\simeq 1$.  The second is the logarithmic divergence of $\Theta$ near $\xi\simeq 0$ and $\xi\simeq 1$ and the fact that $\Theta(\rho^\ve)$ is not known to be integrable.  For these reasons, it is not even clear how to define a notion of weak solution to \eqref{intro_fd_eq}.

These issues were originally addressed by the second two authors in \cite{FG21}, where a general class of SPDEs were treated including, for example, the Dean--Kawasaki equation with correlated noise
\begin{equation}\label{intro_DK}\partial_t\rho = \Delta\rho-\sqrt{\ve}\nabla\cdot(\sqrt{\rho}\circ\xi^{K(\ve)}).\end{equation}
In this case, the Stratonovich-to-It\^o correction presents similar difficulties and takes the form
\[\partial_t\rho = \Delta\rho-\sqrt{\ve}\nabla\cdot\Big(\sqrt{\rho}\xi^{K(\ve)}\Big)+\frac{\ve\langle\xi^{K(\ve)}\rangle}{8}\Delta\log(\rho^\ve).\]
In \cite{FG21} the authors proved the well-posedness of \emph{stochastic kinetic solutions} based on the equation's kinetic formulation by renormalizing the equation away from the zero set of the solution.  We refer to Section~\ref{sec_rks} of this paper and \cite[Section~3]{FG21} for a derivation of the equation's kinetic form and a more complete explanation of the difficulties in treating equations like \eqref{intro_Ito} and \eqref{intro_DK}.  In terms of \eqref{intro_fd_eq}, the equation must be localized away from both of the sets $\{\rho^\ve\simeq 0\}$ and $\{\rho^\ve \simeq 1\}$.  The first main result of this work extends the methods of \cite{FG21} to equations with noise coefficients like $\sqrt{\rho^\ve(1-\rho^\ve)}$ that contain multiple singularities.

\begin{thm*}[cf.\ Theorems~\ref{thm_rks_unique}, \ref{thm_rks_exist}] Let $T\in(0,\infty)$, let $\ve\in(0,1)$, let the noise $\{\xi^K\}_{K\in\N}$ satisfy Assumption~\ref{def_noise_trig} (see also Assumption~\ref{def_noise}) with respect to a filtration $(\F_t)_{t\in[0,\infty)}$ on some probability space $(\O,\F,\P)$, and let $\rho_0\in L^2(\O;L^2(\TT^d;[0,1]))$ be $\F_0$-measurable.  Then there exists a unique \emph{stochastic kinetic solution} of \eqref{intro_fd_eq} in the sense of Definition~\ref{def_sol}.  Furthermore, $\P$-a.s.\ stochastic kinetic solutions $\rho^\ve_1, \rho^\ve_2$ corresponding to initial data $\rho_{0,1}, \rho_{0,2}$ satisfy
\[\sup_{t\in[0,T]}\norm{\rho^\ve_1(\cdot,t)-\rho^\ve_2(\cdot,t)}_{L^1(\TT^d)}\leq \norm{\rho_{0,1}-\rho_{0,2}}_{L^1(\TT^d)}.\]
\end{thm*}

In the second main result of this paper, we identify a scaling regime such that the solutions of \eqref{intro_fd_eq} satisfy a central limit theorem (CLT) equal to that of the SSEP.  The fluctuations are characterized by a generalized Ornstein--Uhlenbeck process $v$ which is a continuous $H^{-s}(\TT^d)$-valued process, for every $s>\frac{d}{2}$, that solves the equation
\begin{equation}\label{intro_clt}\partial_t v = \Delta v-\nabla\cdot(\sqrt{\overline{\rho}(1-\overline{\rho})}\xi)\;\;\textrm{in}\;\;\TT^d\times(0,T)\;\;\textrm{with}\;\;v=0\;\;\textrm{on}\;\;\TT^d\times\{0\},\end{equation}
for $\xi$ an $\R^d$-valued space-time white noise and for $\overline{\rho}$ the solution of \eqref{intro_heat}.  

One essential difficulty in proving the CLT is that the solutions of \eqref{intro_fd_eq} are defined in the renormalized sense of Definition~\ref{def_sol}, and the fundamentally nonlinear nature of this definition is incompatible with convergence in a space of distributions.  For this reason, we first prove in Theorem~\ref{thm_clt} a quantitative CLT for solutions of \eqref{intro_fd_eq} with the square root $\sqrt{\rho^\ve(1-\rho^\ve)}$ replaced by a $\C^1$-smooth, approximating nonlinearity $\sigma(\rho^\ve)$.  We then transfer the CLT for the approximating equation to the solutions of \eqref{intro_fd_eq}, which essentially relies on the following novel $L^\infty$-estimate, the proof of which is based on a Moser iteration.

\begin{thm*}[cf.\ Theorem~\ref{prop_est_infty}]  Let $T\in(0,\infty)$, let $\ve\in(0,1)$, let $\{\xi^K\}_{K\in\N}$ satisfy Assumption~\ref{def_noise_trig} (see also Definition~\ref{def_noise}), let $\rho_0\in L^2(\O;L^2(\TT^d;[0,1]))$ be $\F_0$-measurable, let $M=\esssup_{x\in\TT^d}\rho_0(x)$, and let $M'=\essinf_{x\in\TT^d}\rho_0(x)$.  Then, if $\rho^\ve$ is the unique solution of \eqref{intro_fd_eq} in the sense of Definition~\ref{def_sol} below, there exist $c,\gamma\in(0,\infty)$ independent of $\ve$ and $K$ but depending on $T$ such that
\[\E\Big[\norm{(\rho^\ve-M)_+}_{L^\infty(\TT^d\times[0,T])}\Big]+\E\Big[\norm{(\rho^\ve-M')_-}_{L^\infty(\TT^d\times[0,T])}\Big]\leq c\big(\ve K^{d+2}\big)^\gamma. \]
\end{thm*}

For initial data taking values in $[\d,1-\d]$, for some $\d\in(0,\nicefrac{1}{2})$, this $L^\infty$-estimate proves that, along appropriate scaling limits, the solutions of \eqref{intro_fd_eq} take values in $[\nicefrac{\d}{2},1-\nicefrac{\d}{2}]$ with high probability.  On this event, the pathwise uniqueness proof of Theorem~\ref{thm_rks_unique} below shows that the solutions of \eqref{intro_fd_eq} and the solutions of the approximating SPDE agree for an appropriately chosen $\sigma$, which together with Theorem~\ref{thm_clt} proves the following quantitative CLT in probability.  In the statement, we have fixed a specific choice of scaling limit $K(\ve)\rightarrow 0$ as $\ve\rightarrow 0$ for simplicity.  See Theorem~\ref{thm_clt_prob} for a general estimate that holds for arbitrary $K\in\N$ and $\ve\in(0,1)$.

\begin{thm*}[cf.\ Theorem~\ref{thm_clt_prob}, Corollary~\ref{cor_clt_prob}]  Let $T\in(0,\infty)$, let $\{\xi^K\}_{K\in\N}$ be the noise defined in Definition~\ref{def_noise_trig}, let $\alpha_d=(\frac{1}{2(d+2)}\wedge\frac{1}{2d})$ and let $K(\ve) = \lfloor\ve^{-\a_d}\rfloor$ for every $\ve\in(0,1)$, and let $\rho_0\in L^2(\TT^d)$ satisfy $\d\leq\rho_0\leq 1-\d$ for some $\d\in(0,\nicefrac{1}{2})$.  Let $\rho^\ve$ be the solution of \eqref{intro_fd_eq} corresponding to $(\ve,K(\ve))$ in the sense of Definition~\ref{def_sol} below, let $\overline{\rho}$ be the solution of \eqref{intro_heat}, and let $v^\ve = \ve^{-\nicefrac{1}{2}}(\rho^\ve-\overline{\rho})$.  Then,  for every $s>\frac{d}{2}$ there exists $c,\gamma\in(0,\infty)$ such that, for every $\eta\in(0,1)$,
\[\P\Big[\norm{v^\ve - v}_{L^2([0,T];H^{-s}(\TT^d))}\geq\eta\Big] \leq c\eta^{-2}\delta^{-2}\Big(\ve^{\alpha_d}+\ve^{\alpha_d(2s-(d+2))}\Big)+c\delta^{-1}\ve^\frac{\gamma}{2},\]
for $v$ the solution of \eqref{intro_clt} in the sense of Definition~\ref{def_ou} below.
\end{thm*}

The third main result is the identification of an $\ve\rightarrow 0$, $K(\ve)\rightarrow\infty$ scaling regime such that the solutions $\rho^\ve$ of \eqref{intro_fd_eq} satisfy a large deviations principle with rate function \eqref{intro_rate_function}.  The proof of the large deviations principle is based on the weak convergence approach to large deviations of Budhiraja, Dupuis, and Maroulas \cite{BudDupMar2008}, Dupuis and Ellis \cite{DupEll1997}, and Budhiraja and Dupuis \cite{BudDup2019}.

\begin{thm*}[cf.\ Theorem~\ref{thm_new}] Let $T\in(0,\infty)$, let $\{\xi^K\}_{K\in\N}$ be the noise defined in Definition~\ref{def_noise_trig}, let $\{K(\ve)\}_{\ve\in(0,1)}$ be a sequence that satisfies, as $\ve\rightarrow 0$,
\[\ve K(\ve)^{d+2}\rightarrow 0\;\;\textrm{and}\;\;K(\ve)\rightarrow\infty,\]
and for every $\rho_0\in L^2(\TT^d;[0,1])$ let $\rho^\ve(\rho_0)$ be the unique stochastic kinetic solution of \eqref{intro_fd_eq} in the sense of Definition~\ref{def_sol}.  Then, for every $\rho_0\in L^2(\TT^d;[0,1])$ the solutions $\{\rho^\ve(\rho_0)\}_{\ve\in(0,1)}$ satisfy a large deviations principle on $L^2(\TT^d\times[0,T])$ with rate function $I_{\rho_0}$ defined in \eqref{intro_rate_function}.  Furthermore, the solutions $\{\rho^\ve(\rho_0)\}_{\rho_0\in L^2(\TT^d)}$ satisfy a uniform large deviations principle on $L^2(\TT^d\times[0,T])$ with respect to weakly $L^2(\TT^d;[0,1])$-compact subsets of $L^2(\TT^d;[0,1])$. \end{thm*}

\subsection{Comments on Applications and Numerics} 
The accurate computation of rare events for high-dimensional systems is a very important topic in applied sciences, for example in  chemistry and material science.  Interacting particle models like the SSEP provide simplified examples that nevertheless retain many of the features of more realistic models. Often, the number of particles is quite large but the features one is interested in appear on a much coarser scale than the microscopic scale of individual particles. In such situations it is inefficient to use the particle model itself for Monte--Carlo simulations. An SPDE solved on a coarser time-space grid is more effective, but in order to investigate large deviations from the deterministic limit, this SPDE must model the fluctuations correctly on the chosen scale. This is the case for equation \eqref{intro_fd_eq} as far as the large deviations rate functional, i.e. the highest exponential order, is considered, as argued in this paper.
 
It is worth noting that rare events can numerically be computed  to higher accuracy than just the leading exponential order given by the large deviation principle, see e.g.~\cite{VEW12}. This requires frequent updating of the external field which creates the large deviation. Therefore, fast methods to do so are required.

If the deterministic limit path $\bar \rho$ for $\eps\to 0$ is known, then  both \eqref{intro_fd_eq} and \eqref{intro_clt} are candidates for simulating the fluctuations, and, by the results of Section~\ref{sec_clt} below, they have identical fluctuations on the CLT scale. However, as argued above, only \eqref{intro_fd_eq} recovers the correct large deviations behavior.

Before presenting numerical results, let us emphasize a further qualitative difference between the two SPDEs: in the case of additive noise \eqref{intro_clt} the fluctuations are symmetric around $\overline \rho$. Their strength does not depend on the actual excursion away from $\overline\rho$, and the expectation of the fluctuation field integrated against any smooth test function is zero. The situation is different for the nonlinear SPDE \eqref{intro_fd_eq}: if $\overline \rho$ is close to $1,$ say, then excursions above, i.e. such that $\rho^\eps>\overline \rho$ are driven by weaker fluctuations than excursions below, so that the fluctuations are asymmetric. The same effect can be expected for the full particle system because of the exclusion rule: the closer the local density of particles is to $1$, the more jumps are excluded because positions are already filled by particles, and, thus, fluctuations are damped.

In order to numerically demonstrate this effect, we aim to make the local density explore a large range of values, including those close to $0$ and $1$, where fluctuations are expected to be damped, but also those close to $1/2,$ where fluctuations are expected to be maximal. The scenario giving rise to Figure \ref{Figure1} is as follows: the evolution starts from uniformly distributed particles on the unit interval with periodic boundary conditions and the field $H=\sin(2\pi x)$ creates a rare event which consists in many particles gathering near $x=1/4$ and few particles near $x=3/4$.
Following Kipnis and Landim \cite[p~ 258]{KL99}: we modify the jump rates by weighting them with $H$ to obtain the shifted generator
\[(L^H_{N,t}f)(\eta)=N^2\sum_{|x-y|=1}\eta(x)(1-\eta(y))e^{H(y/N)-H(x/N)}(f(\eta^{x,y})-f(\eta)).\]
Due to the difference in the exponential, this leads to the {\em derivative} of $H$ appearing in the exponential change of measure and the deterministic limit equation, see \cite[p.258 and (5.1)]{KL99}.

The SPDEs are solved by a spectral method (Fast Fourier Transform), where the nonlinearity is resolved explicitly. The noise on each of the $N_K=2K-1$ basis functions $\sin(2\pi kx)$ and $\cos(2\pi kx)$ is independent and identically distributed. The parameter $\epsilon$ is as above. The numerical accuracy is determined by $M\gg K,$ the number of nodes computed, and $h,$ the time step.
The following computations are done at time $t=1.$ The solutions are then smoothed by a projection on the span of the first 20 complex Fourier modes. The black curve is the deterministic solution $\overline \rho.$ The second plot shows the fluctuation fields  smoothed by projecting on the span of the first 40 Fourier modes. This is all done with 100 independent realizations for $N=1280$ particles. The last plot (not to scale) is the approximated pointwise expectation of the fluctuation fields, smoothed by taking the sliding average in space over 50 points.

\begin{figure}[h!]
\begin{tabular}{ccc}
\includegraphics[scale=0.33]{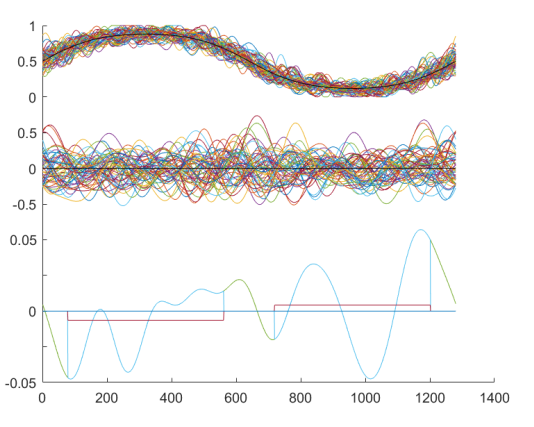}
&\includegraphics[scale=0.33]{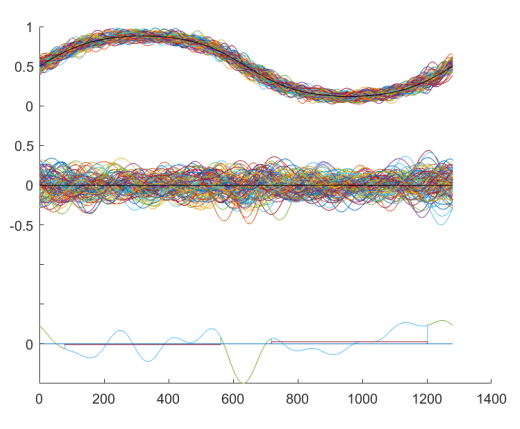}
&\includegraphics[scale=0.33]{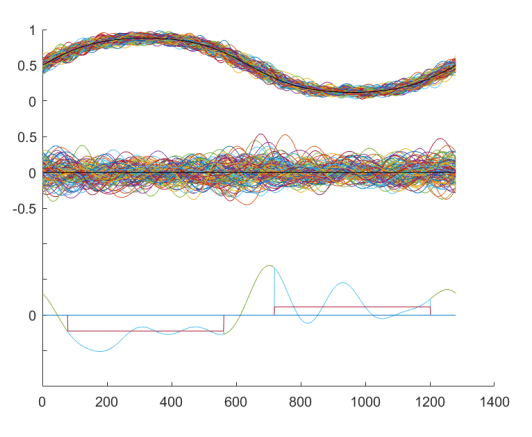}\end{tabular}
\caption{{\small Particles (left), linear SPDE (centre), SPDE  \eqref{intro_fd_eq} (right). First row smoothed output, second row smoothed fluctuations, third row averaged fluctuations.}\label{Figure1}}
\end{figure}

The red line is the spatial average of the fluctuations in the region between the vertical bars, that is, where the density is away from $1/2$, either below or above. We indeed observe an asymmetry of the fluctuations: deviations from the deterministic limit are observed to be more likely pointing towards $1/2$ than away from it.

We next numerically investigate the hypothesis that the qualitative advantage of the nonlinear SPDE \eqref{intro_fd_eq} and the linear SPDE \eqref{intro_clt} also leads to a quantitative advantage. The corresponding findings are reported in Figure \ref{Figure2}. 

The number of nodes with noise is chosen in accordance with
Corollary \ref{cor_clt_prob}, that is, of the order $\eps^{-1/2}$.  The number of gridpoints for the SPDEs is $1280$, the time step is chosen inversely quadratic to the number of gridpoints in order to keep numerical errors low. The time $t=1$ is chosen such that the fluctuations of the initial condition (independent Bernoulli, $p=1/2$) for the particle system do not play a role anymore except for the conserved total mass, which the plot is adjusted for.
The expectation is  approximated by averaging over $100$ independent realizations.

The plot in Figure \ref{Figure2} shows the fluctuation fields applied to a function $\overline \Psi(x),$ which is a smooth approximation of $1_{(0,0.5)}-1_{(0.5,1)},$ i.e. a moment of the fluctuation field
\begin{equation}\label{eq:testfct}
  \eps^{-1/2}{\mathbb E} \int_0^1\overline \Psi(\rho-\overline \rho)dx,
\end{equation}
where $\overline \rho$ is the deterministic solution and $\rho$ is either the particles or the solution of the SPDE. This shows the aforementioned asymmetry.

\begin{figure}[h!]
\includegraphics[scale=0.5]{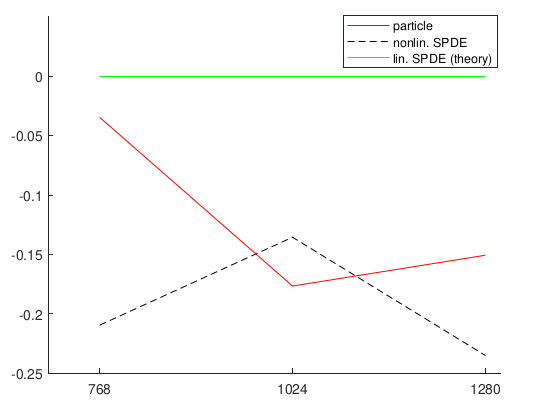}
\caption{\small $x-$axis: $N=1/\epsilon.$ $y$-axis: fluctuations integrated against the specific test function \eqref{eq:testfct}. Red particles, black SPDE \eqref{intro_fd_eq} (multiplicative noise), green SPDE \eqref{intro_clt} (additive noise, theoretical value). \label{Figure2}}
\end{figure}
  
We observe that for $N=1280$ the result for the particle system (red line) is with high probability negative, and this is correctly captured by the nonlinear SPDE \eqref{intro_fd_eq} (black dotted line). In contrast, the theoretical value for the linear SPDE \eqref{intro_clt} is zero, so it cannot capture the law of the particle model at orders beyond the CLT scale.

\subsection{Comments on the literature}

Stochastic PDEs with nonlinear, conservative noise have been studied in the context of stochastic scalar conservation laws by Lions, Perthame, and Souganidis \cite{LPS13-2,LPS13,LPS14}, Friz and Gess \cite{FG16}, and Gess and Souganidis \cite{GS14,GS17}. These methods were later extended to parabolic-hyperbolic equations with conservative noise by Gess and Souganidis \cite{GS16-2}, Fehrman and Gess \cite{FG17,FG21}, and Dareiotis and Gess \cite{DG18}.

Large deviation principles for singular SPDEs have been obtained by Cerrai and Freidlin \cite{CF11}, Faris and Jona-Lasinio \cite{FJL82}, Jona-Lasinio and Mitter \cite{JLM90}, Hairer and Weber \cite{HaiWeb2015}, and Fehrman and Gess \cite{FehGes19}.  In particular, in the context of stochastic Allen-Cahn equations, it was shown in \cite{HaiWeb2015} that renormalization constants may enter into the large deviations rate functional.  The constants are determined by the relative scaling of the noise intensity $\ve$ and ultraviolet cutoff $K$.  In analogy with \cite{FehGes19}, the above results therefore identify a scaling regime in which the large deviations are unaffected by renormalization and for which the solutions of \eqref{intro_fd_eq} correctly simulate the fluctuations of the particle process.  Further applications of the weak convergence approach to proving large deviations principles for SPDE include, for example, Cerrai and Debussche \cite{CD19}, Brze\'zniak, Goldys, and Jegaraj \cite{BGJ17}, and Dong, Wu, Zhang, and Zhang \cite{DWZZ20}.

The inference of the fluctuation correction from data, that is, from observations of the underlying particle system has been studied by Li, Dirr, Embacher, Zimmer, and Reina in \cite{LDEZR2019}.

A central limit theorem for the stochastic heat equation has been obtained by Huang, Nualart, and Viitasaari \cite{HuaNuaVii2020}, Huang, Nualart, Viitasaari, and Zheng \cite{HuaNuaViiZhe2020}, and for parabolic equations with multiplicative noise by Chen, Khoshnevisan, Nualart, and Pu \cite{CheKhoNuaPu2019}.  A central limit theorem and moderate deviations principle has been obtained by Hu, Li, and Wang \cite{HuLiWan2020} and Guo, Zhang, and Zhuo \cite{ZhaZhoGuo2019} for a class of semilinear SPDEs.  The authors are not aware of any prior works proving a central limit theorem for an SPDE with conservative noise. 

After publication of the first version of this article, several groups have investigated the use of conservative SPDEs in the numerical approximation of particle systems: in an intriguing contribution \cite{CF23}, Cornalba and Fischer have shown that a system of independent Brownian motions can be approximated to arbitrary order by a discretization of the Dean-Kawasaki SPDE. A related result has been obtained by Djurdjevac, Kremp, and Perkowski in \cite{DKP22}. Subsequently, Cornalba, Fischer, Ingmanns, and Raithel in \cite{CFIR23} have extended their theory to weakly interacting particles.

The fluctuations of the symmetric simple exclusion process about its hydrodynamic limit have been studied by Calves, Kipnis, and Spohn \cite{CalKipSpo}, Ravishankar \cite{Rav1992}, Kipnis and Varadhan \cite{KipVar1986}, Ferrari, Presutti, Scacciatelli, and Vares \cite{FerPresScacVar1991,FerPresScaVar19912}, and Rezakhanlou \cite{Rez1994}.  Tagged particles in symmetric simple exclusion processes have been studied by Arratia \cite{Arr1983}, Sethuraman, Varadhan, and Yau \cite{SetVarYau2000}, Varadhan \cite{Var1995}, and Jara and Landim \cite{JarLan2006}.  Large deviations principles for a general class of symmetric simple exclusion processes were obtained by Quastel, Rezakhanlou, and Varadhan \cite{QuaRezVar1999}.

\subsection{Organization of the paper}   In Section~\ref{sec_spde_wp}, we prove the well-posedness of stochastic kinetic solutions to \eqref{intro_fd_eq}.  Section~\ref{sec_noise} defines the noise, Section~\ref{sec_rks} defines stochastic kinetic solutions and proves that they are unique, and Section~\ref{sec_rks_exist} proves the existence of solutions.  We prove the central limit theorem in Section~\ref{sec_clt}, and we prove the large deviations principle in Section~\ref{sec_old}.

\section{The well-posedness of the SPDE}\label{sec_spde_wp}

We will now establish the existence and uniqueness of suitably defined renormalized kinetic solutions to equation \eqref{intro_fd_eq}.  The section is split into three subsections.  The first defines the randomness in the equation, the second defines stochastic kinetic solutions in Definition~\ref{def_sol} and proves that they are unique, and the third proves that they exist.

\subsection{The definition of the noise}\label{sec_noise}  We will first introduce the randomness in the equation
\begin{equation}\label{spde_1}\partial_t\rho = \Delta\rho - \nabla\cdot(\sqrt{\rho(1-\rho)}\circ \dd\xi^F)\;\;\textrm{in}\;\;\TT^d\times(0,\infty)\;\;\textrm{with}\;\;\rho(\cdot,0) = \rho_0.\end{equation}
We will assume that the noise $\xi^F$ is colored in space and white in time, and that it is adapted to some filtration $\F_t$, and that the initial data is $\F_0$-measurable.

\begin{assumption}\label{def_noise}  Let $(\O,\F,\P)$ be a probability space, let $(\F_t)_{t\in[0,\infty)}$ be a filtration on $\O$, let $\{B^k\}_{k\in\N}$ be independent $\F_t$-adapted $d$-dimensional Brownian motions on $\O$, and let $\{f_k\colon\TT^d\rightarrow\R\}_{k\in\N}$ be twice continuously differentiable functions on $\TT^d$.  We define $\xi^F$ to be the noise $\xi^F(x,t) = \sum_{k=1}^\infty f_k(x) B^k_t$ and assume that the sums $F_1 = \sum_{k=1}^\infty f_k^2$ and $F_3 = \sum_{k=1}^\infty \abs{\nabla f_k}^2$ are finite and continuous on $\TT^d$ and that the noise is probabilistically stationary in the sense that $F_1$ is constant on $\TT^d$.  We also assume that $\rho_0\in L^2(\O;L^2(\TT^d;[0,1]))$ is $\F_0$-measurable.  \end{assumption}

\subsection{Uniqueness of stochastic kinetic solutions}\label{sec_rks}  We will now explain how the methods of the second two authors \cite{FG21} can be adapted to prove the well-posedness \eqref{spde_1}, which we now write in the It\^o formulation
\begin{equation}\label{spde_2}\partial_t\rho  = \Delta\rho - \nabla\cdot(\sqrt{\rho(1-\rho)}\dd\xi^F)+\frac{F_1}{8}\nabla\cdot\Big(\frac{(1-2\rho)^2}{\rho(1-\rho)}\nabla\rho\Big).\end{equation}
There are two essential difficulties in treating \eqref{spde_2}.  The first is the singularity of the noise coefficient $\sqrt{\rho(1-\rho)}$.  It is only $\nicefrac{1}{2}$-H\"older continuous on $[0,1]$ and its derivative diverges algebraically on the sets $\{\rho\simeq 0\}$ and $\{\rho\simeq 1\}$.  The second is the singularity in the Stratonovich-to-It\^o correction, which is defined by
\[\nabla\cdot\Big(\frac{(1-2\rho)^2}{\rho(1-\rho)}\nabla\rho\Big) = \Delta\Theta(\rho)\;\textrm{for $\Theta$ satisfying $\Theta(\nicefrac{1}{2})=0$ and}\;\Theta'(\xi) = \frac{(1-2\xi)^2}{\xi(1-\xi)},\]
and which diverges logarithmically near $\xi=0$ and $\xi=1$.  In particular, since for a solution $\rho$ of \eqref{spde_2} the composition $\Theta(\rho)$ is not known to be integrable, there is not even an obvious concept of weak solution for \eqref{spde_2}.

It is for these reasons that it is necessary to develop a renormalized solution theory.  The theory is based on localizing the solution away from the sets $\{\rho\simeq 0\}$ and $\{\rho \simeq 1\}$, and will require only local regularity of the solution away from these sets.  The methods are based on the equation's kinetic formulation.  Precisely, if $S\colon\R\rightarrow\R$ is a convex function, then it follows from It\^o's formula and the probabilistic stationarity of the noise that the composition $S(\rho)$ formally satisfies the equation
\begin{align}\label{spde_3}
\partial_tS(\rho) & = \Delta S(\rho) -S'(\rho)\nabla\cdot(\sqrt{\rho(1-\rho)}\dd\xi^F)+\frac{F_1}{8}\nabla\cdot\Big(\frac{(1-2\rho)^2}{\rho(1-\rho)}\nabla S(\rho)\Big)
\\ \nonumber & \quad -S''(\rho)\abs{\nabla\rho}^2  +\frac{1}{2}S''(\rho) F_3\rho(1-\rho),
\end{align}
where, in general, $S(\rho)$ will only satisfy the \emph{entropy inequality} that the lefthand side is less than or equal to the righthand side.  The kinetic formulation replaces this ensemble of equations by a single equation with respect to an additional velocity variable $\xi\in\R$.

The kinetic function $\chi\colon\TT^d\times\R\times[0,T]\rightarrow\{-1,0,1\}$ of $\rho$ is defined by
\[\chi(x,\xi,t) = \mathbf{1}_{\{0<\xi<\rho(x,t)\}}-\mathbf{1}_{\{\rho(x,t)<\xi<0\}},\]
and using the equality that $S(\rho) = \int_{\R}S'(\xi)\chi \dxi$, the ``test function'' $S'(\xi)$ can be factored out \eqref{spde_3} and the kinetic function is formally a distributional solution of the equation
\begin{align*}
\partial_t\chi & = \Delta_x\chi - \delta_0(\xi-\rho)\nabla\cdot(\sqrt{\rho(1-\rho)}\dd\xi^F)+\frac{F_1(1-2\xi)^2}{8\xi(1-\xi)}\Delta_x\chi
\\ & \quad + \partial_\xi\Big(\delta_0(\xi-\rho)\abs{\nabla\rho}^2\Big)-\frac{F_3}{2}\partial_\xi\Big(\delta_0(\xi-\rho)\xi(1-\xi)\Big),
\end{align*}
for the one-dimensional Dirac delta distribution $\delta_0$ at zero.  However, as in the case of the entropy inequality above, the kinetic function will not in general satisfy this equality exactly.  Rather, there will $\P$-a.s.\ exist a nonnegative parabolic defect measure $q$ on $\TT^d\times\R\times[0,T]$ satisfying
\[\delta_0(\xi-\rho)\abs{\nabla\rho}^2\leq q\;\textrm{on}\;\TT^d\times\R\times[0,T],\]
such that the kinetic function is a distributional solution of the equation
\begin{align*}
\partial_t\chi & = \Delta_x\chi - \delta_0(\xi-\rho)\nabla\cdot(\sqrt{\rho(1-\rho)}\dd\xi^F)+\frac{F_1(1-2\xi)^2}{8\xi(1-\xi)}\Delta_x\chi
\\ & \quad +\partial_\xi q-\frac{F_3}{2}\partial_\xi\Big(\delta_0(\xi-\rho)\xi(1-\xi)\Big).
\end{align*}
The parabolic defect measure $q$ quantifies the entropy inequality exactly, and allows to consider test functions that are not convex in the velocity variable.  This point is essential for the solution theory developed in this work, since here it is necessary to localize the solution compactly in the interval $(0,1)$.  We define a stochastic kinetic solution in Definition~\ref{def_sol} below and prove that such solutions are unique and satisfy a pathwise $L^1$-contraction in Theorem~\ref{thm_rks_unique}.

\begin{definition} Let $(\O,\F,\P)$ be a probability space and let $(\F_t)_{t\in[0,\infty)}$ be a filtration on $(\O,\F)$.  A \emph{kinetic measure} is a measurable map $q$ from $\O$ to the space of nonnegative, finite Radon measures on $\TT^d\times[0,1]\times[0,T]$ such that, for every $\psi\in \C^\infty(\TT^d\times [0,1])$, the process
\[(\omega,t)\in \O\times[0,T]\rightarrow \int_0^t\int_{\TT^d}\int_\R \psi(x,\xi) \dd q(\omega)\;\;\textrm{is $\F_t$-predictable.}\]
\end{definition}

\begin{remark}  In what follows, we will often encounter derivatives of functions $\psi\in\C^\infty_c(\TT^d\times \R)$ evaluated at $\xi = \rho(x,t)$.  We will write
\[(\nabla\psi)(x,\rho(x,t)) = \nabla\psi(x,\xi)|_{\xi = \rho(x,t)},\]
for the gradient of $\psi$ evaluated at $(x,\rho(x,t))$ as opposed to the gradient of the full composition. \end{remark}

\begin{remark} If $G\in L^2([0,T]\times\O;L^2(\TT^d))^d$ is an $\F_t$-predictable process, we will write $\int_0^T\int_{\TT^d}G\cdot\xi^F$ for the stochastic integral understood in the It\^o-sense.  Due to the $L^2$-integrability and $\F_t$-predictability of $G$, this is a well-defined random variables.  We will use the notation $\int_0^TG\circ\xi^F$ to denote the corresponding Stratonovich integral.  \end{remark}

\begin{definition}\label{def_sol}  Let $T\in(0,\infty)$ and let $\xi^F$ and $\rho_0\in L^2(\O;L^2(\TT^d;[0,1]))$ satisfy Assumption~\ref{def_noise}.  A \emph{stochastic kinetic solution} of \eqref{spde_1} is a continuous $L^2(\TT^d;[0,1])$-valued, $\F_t$-adapted process $\rho\in L^2(\O\times[0,T];L^2(\TT^d))$ that satisfies the following two properties.

\begin{enumerate}
\item \emph{Preservation of mass}:  $\P$-a.s.\ for every $t\in[0,T]$,
\[\norm{\rho(\cdot,t)}_{L^1(\TT^d)}=\norm{\rho_0}_{L^1(\TT^d)}.\]
\item \emph{Local regularity}:  for every $\d\in(0,\nicefrac{1}{2})$,
\[(\delta \vee \rho )\wedge (1-\delta)\in L^2(\O\times[0,T];H^1(\TT^d)).\]
\end{enumerate}

Furthermore, there exists a nonnegative \emph{parabolic defect measure} $q$ that $\P$-a.s.\ satisfies the following three properties.

\begin{enumerate}\setcounter{enumi}{2}
\item \emph{The entropy condition}:  $\P$-a.s.\ as measures on $\TT^d\times\R\times[0,T]$,
\[\delta_0(\xi-\rho)\abs{\nabla\rho}^2\leq q.\]
\item \emph{Optimal regularity}:  the measure $\mu$ defined by
\[\dd\mu = (\xi(1-\xi))^{-1} \dd q\;\textrm{is finite on $\TT^d\times(0,1)\times[0,T]$.}\]
\item \emph{The equation}:  for the kinetic function $\chi$ of $\rho$, for every $\psi\in\C^\infty_c(\TT^d\times(0,1))$ and $t\in[0,T]$,
\begin{align*}
& \int_{\TT^d}\int_\R \chi(x,\xi,t)\psi(x,\xi) = \int_{\TT^d}\int_\R \overline{\chi}(\rho_0)\psi(x,\xi)-\int_0^t\int_{\TT^d} \nabla\rho \cdot (\nabla\psi)(x,\rho)
\\ & -\int_0^t\int_{\TT^d} \nabla\cdot\Big(\sqrt{\rho(1-\rho)}\dd\xi^F\Big)\psi(x,\rho)-\frac{F_1}{8}\int_0^t\int_{\TT^d} \frac{(1-2\rho)^2}{\rho(1-\rho)}\nabla\rho\cdot(\nabla\psi)(x,\rho)
\\ & -\int_0^t\int_{\TT^d}\int_\R \partial_\xi \psi(x,\xi) \dd q+\frac{1}{2}\int_0^t\int_{\TT^d} F_3 (\partial_\xi\psi)(x,\rho)\rho(1-\rho).
\end{align*}
\end{enumerate}
\end{definition}

\begin{thm}\label{thm_rks_unique}  Let $T\in(0,\infty)$ and let $\xi^F$ and $\rho_{0,1},\rho_{0,2}\in L^2(\O;L^2(\TT^d;[0,1]))$ satisfy Assumption~\ref{def_noise}.  If $\rho_1,\rho_2$ are pathwise kinetic solutions of \eqref{spde_1} in the sense of Definition~\ref{def_sol} then, $\P$-a.s.,
\[\max_{t\in[0,T]}\norm{\rho_1(\cdot,t)-\rho_2(\cdot,t)}_{L^1(\TT^d)}\leq \norm{\rho_{0,1}-\rho_{0,2}}_{L^1(\TT^d)}.\]
\end{thm}

\begin{proof}  For every $i\in\{1,2\}$ let $\chi_i$ denote the kinetic function of $\rho_i$ and let $q_i$ denote the corresponding kinetic measure defined in Definition~\ref{def_sol}.  For every $\ve\in(0,1)$ let $\kappa^\ve$ denote a standard convolution kernel of scale $\ve\in(0,1)$ on $\TT^d$, for every $\d\in(0,1)$ let $\kappa^\d$ denote a standard convolution kernel of scale $\d\in(0,1)$ on $\R$, and let $\kappa^{\ve,\d}(x,y,\xi,\eta)=\kappa^\ve(x-y)\kappa^\d(\xi-\eta)$.  It is furthermore necessary to introduce a cutoff in the velocity variable.   For every $\beta\in(0,\nicefrac{1}{4})$ let $\zeta_\beta\colon\R\rightarrow[0,1]$ be a smooth function satisfying $\zeta_\beta(\xi) = 0$ if $\xi\leq \beta$ or $\xi\geq 1-\beta$, $\zeta_\beta(\xi) = 1$ if $\xi\in[2\beta, 1-2\beta]$, and $\abs{\zeta_\beta'(\xi)}\leq\nicefrac{c}{\beta(1-\beta)}$ for some $c\in(0,\infty)$ independent of $\beta$.

We write $\chi_{i,t}(x,\xi)=\chi_i(x,\xi,t)$ and make a similar convention for all other time-dependent quantities and we define the convolution $\chi^{\ve,\d}_{i,t} = (\chi_{i,t}*\kappa^{\ve,\d})$.  The proof of uniqueness is based on the equality
\[\int_{\TT^d}\abs{\rho_1(x,t)-\rho_2(x,t)}\dx = \int_{\TT^d}\int_\R(\chi_1+\chi_2-2\chi_1\chi_2)\dx\dxi,\]
which we approximate, for every $\beta\in(0,\nicefrac{1}{4})$, $\ve\in(0,1)$, and $\delta\in(0,\nicefrac{\beta}{2})$ by the quantity
\begin{equation}\label{u_1} \int_{\TT^d}\int_\R (\chi^{\ve,\d}_{1,t}(y,\eta)+\chi^{\ve,\d}_{2,t}(y,\eta)-2\chi^{\ve,\d}_{1,t}(y,\eta)\chi^{\ve,\d}_{2,t}(y,\eta))\zeta_\beta(\eta) \dy\deta.\end{equation}
In view of definition~\ref{def_sol}, the cutoff $\zeta_\beta$ and the restriction $\delta\in(0,\nicefrac{\beta}{2})$ guarantees that on the support of $\zeta_\beta$ the function $\kappa^{\ve,\d}(\cdot-y,\cdot-\eta)$ is an admissible test function.  We will differentiate the three terms on the righthand side of \eqref{u_1} individually.

\textbf{The singletons}.  For the first two terms on the righthand side of \eqref{u_1}, it follows from the equation and the symmetry of the convolution kernel that, for every $\eta$ in the support of $\zeta_\beta$, $\delta\in(0,\nicefrac{\beta}{2})$, and $y\in\TT^d$, for $\overline{\kappa}^{\ve,\d}_{i,t}(x,y,\eta) = \kappa^{\ve,\d}(x,y,\rho_i(x,t),\eta)$,
\begin{align}\label{u_2}
\dd \chi^{\ve,\d}_{i,t}(y,\eta) & = \nabla_y\cdot (\nabla\rho_i*\overline{\kappa}^{\ve,\d}_{i,t})-(\overline{\kappa}^{\ve,\d}_{i,t}*\nabla\cdot(\sqrt{\rho_i(1-\rho_i)}\dd\xi^F))
\\ \nonumber & \quad +\frac{F_1}{8}\nabla_y\cdot \Big(\frac{(1-2\rho_i)^2}{\rho_i(1-\rho_i)}\nabla\rho_i*\overline{\kappa}^{\ve,\d}_{i,t}\Big)+\partial_\eta (\kappa^{\ve,\d}*q_{i,t})-\partial_\eta(F_3\rho_i(1-\rho_i)*\overline{\kappa}^{\ve,\d}_{i,t}).
\end{align}
Therefore, using \eqref{u_2} and integrating by parts, for every $i\in\{1,2\}$,
\begin{equation}\label{u_3}  \dd \int_{\TT^d}\int_\R \chi^{\ve,\d}_{i,t}(y,\eta)\zeta_\beta(\eta)\dy\deta  = \dd I^{\textrm{mart}}_{i,t}+\dd I^{\textrm{cut}}_{i,t}, \end{equation}
for the martingale term
\[ \dd I^{\textrm{mart}}_{i,t} = -\int_{\TT^d}\int_{\R}\big(\overline{\kappa}^{\ve,\d}_{i,t}*\nabla\cdot(\sqrt{\rho_i(x,t)(1-\rho_i(x,t))}\dd\xi^F)\big)\zeta_\beta,\]
and for the cutoff term
\begin{align*}
\dd I^{\textrm{cut}}_{i,t} & = -\int_{\TT^d}\int_{\R}(\kappa^{\ve,\d}*q_{i,t})\zeta_\beta'+\frac{1}{2}\int_{(\TT^d)^2}\int_\R \big(F_3\rho_i)(1-\rho_i)*\overline{\kappa}^{\ve,\d}_{i,t}\big)\zeta_\beta',
\end{align*}
which completes the initial analysis of the first two terms on the righthand side of \eqref{u_1}.

\textbf{The mixed term}.  For the mixed term on the righthand side of \eqref{u_1}, we observe using the distributional equalities
\[\nabla_x\chi_i(x,\xi,t) = \delta_0(\xi-\rho_i(x,t))\nabla\rho_i\;\textrm{and}\;\partial_\xi \chi_i(x,\xi,t) = \delta_0(\xi)-\delta_0(\xi-\rho_i(x,t)),\]
the stochastic product rule, and $\delta\in(0,\nicefrac{\beta}{2})$ that
\begin{equation}\label{u_4}\textrm{d}\Big(\int_{\TT^d}\int_\R\chi^{\ve,\d}_{1,t}(y,\eta)\chi^{\ve,\d}_{2,t}(y,\eta)\zeta_\beta(\eta)\Big)= \dd I^{\textrm{meas}}_{t}+\dd I^{\textrm{mart}}_{\textrm{mix},t}+\dd I^{\textrm{cor}}_{t}+\dd I^{\textrm{cut}}_{\textrm{mix},t},\end{equation}
for the measure term
\begin{align*}
\dd I^{\textrm{meas}}_{t} & = -2 \int_{\TT^d}\int_\R (\nabla\rho_1*\kappa^{\ve,\d}_{1,t})(\nabla\rho^2*\kappa^{\ve,\d}_{2,t})\zeta_\beta
\\ & \quad +\int_{(\TT^d)^2}\int_{\R}(\kappa^{\ve,\d}*q_{1,t})(y,\eta)\overline{\kappa}^{\ve,\d}_{2,t}(x,y,\eta)\zeta_\beta(\eta)
\\ & \quad +\int_{(\TT^d)^2}\int_{\R}(\kappa^{\ve,\d}*q_{2,t})(y,\eta)\overline{\kappa}^{\ve,\d}_{1,t}(x,y,\eta)\zeta_\beta(\eta),
\end{align*}
for the martingale term
\begin{align*}
\dd I^{\textrm{mart}}_{\textrm{mix},t} & = -\int_{\TT^d}\int_\R \chi^{\ve,\d}_{2,t}\big(\overline{\kappa}^{\ve,\d}_{1,t}*\nabla\cdot\big(\sqrt{\rho_1(1-\rho_1)}\dd \xi^F\big)\big)\zeta_\beta
\\ & \quad -\int_{\TT^d}\int_\R \chi^{\ve,\d}_{1,t} \big(\overline{\kappa}^{\ve,\d}_{2,t}*\nabla_{x'}\cdot\big(\sqrt{\rho_2(1-\rho_2)}\dd\xi^F\big)\big)\zeta_\beta,
\end{align*}
for the correction term, writing $(x,\xi)$ for the convolution variables of $\rho_1$ and all related quantities and writing $(x',\xi')$ for the convolution variables of $\rho_2$,
\begin{align*}
& \dd I^{\textrm{cor}}_{t} = -\frac{F_1}{8}\int_{(\TT^d)^3}\int_\R \zeta_\beta\overline{\kappa}^{\ve,\d}_{1,t}(x,y,\eta)\overline{\kappa}^{\ve,\d}_{2,t}(x',y,\eta)\Big(\frac{(1-2\rho_1)^2}{\rho_1(1-\rho_1)}+\frac{(1-2\rho_2)^2}{\rho_2(1-\rho_2)}\Big)\nabla\rho_1\cdot\nabla\rho_2
\\ & + \frac{1}{4}\int_{(\TT^d)^3}\int_\R \zeta_\beta\overline{\kappa}^{\ve,\d}_{1,t}(x',y,\eta)\overline{\kappa}^{\ve,\d}_{2,t}(x,y,\eta)\Big(\sum_{k=1}^\infty f_k(x)f_k(x')\Big) \frac{(1-2\rho_1)(1-2\rho_2)}{\sqrt{\rho_1(1-\rho_1)}\sqrt{\rho_2(1-\rho_2)} }\nabla\rho_1\cdot\nabla\rho_2
\\ & -\frac{1}{2}\int_{(\TT^d)^3}\int_\R \zeta_\beta\overline{\kappa}^{\ve,\d}_{1,t}(x,y,\eta)\overline{\kappa}^{\ve,\d}_{2,t}(x',y,\eta)F_3(x)\rho_1(1-\rho_1)
\\ & -\frac{1}{2}\int_{(\TT^d)^3}\int_\R\zeta_\beta\overline{\kappa}^{\ve,\d}_{2,t}(x',y,\eta)\overline{\kappa}^{\ve,\d}_{1,t}(x,y,\eta)F_3(x')\rho_2(1-\rho_2)
\\ & +\int_{(\TT^d)^3}{\int_\R}\zeta_\beta\overline{\kappa}^{\ve,\d}_{1,t}(x',y,\eta)\overline{\kappa}^{\ve,\d}_{2,t}(x,y,\eta)\Big(\sum_{k=1}^\infty \nabla f_k(x)\cdot\nabla f_k(x')\Big)\sqrt{\rho_1(1-\rho_1)}\sqrt{\rho_2(1-\rho_2)},
\end{align*}
and for the cutoff term
\begin{align*}
\dd I^{\textrm{cut}}_{\textrm{mix},t} &  = -\int_{\TT^d}\int_{\R} (\kappa^{\ve,\d}*q_{1,t})\chi^{\ve,\d}_{2,t} \zeta_\beta' - \int_{\TT^d}\int_{\R} (\kappa^{\ve,\d}*q_{2,t})\chi^{\ve,\d}_{1,t} \zeta_\beta'
\\ & \quad +\frac{1}{2}\int_{\TT^d}\int_\R (\overline{\kappa}^{\ve,\d}_{1,t}*F_3\rho_1(1-\rho_1))\chi^{\ve,\d}_{2,t}\zeta_\beta'+\frac{1}{2}\int_{\TT^d}\int_\R (\overline{\kappa}^{\ve,\d}_{2,t}*F_3\rho_2(1-\rho_2))\chi^{\ve,\d}_{1,t}\zeta_\beta'.
\end{align*}
This completes the initial analysis of the mixed term.  Returning to \eqref{u_1}, we have from \eqref{u_3} and \eqref{u_4} that
\begin{equation}\label{u_5} \dd\Big( \int_{\TT^d}\int_\R (\chi^{\ve,\d}_{1,t}+\chi^{\ve,\d}_{2,t}-2\chi^{\ve,\d}_{1,t}\chi^{\ve,\d}_{2,t})\zeta_\beta\Big) = -2\dd I^{\textrm{meas}}_t+\dd I^{\textrm{mart}}_t-2\dd I^{\textrm{cor}}_t+\dd I^{\textrm{cut}}_t,\end{equation}
for the cutoff and martingale terms
\[\dd I^{\textrm{cut}}_t = \dd I^{\textrm{cut}}_{1,t}+\dd I^{\textrm{cut}}_{2,t} - 2\dd I^{\textrm{cut}}_{\textrm{mix},t}\;\textrm{and}\;\dd I^{\textrm{mart}}_t = \dd I^{\textrm{mart}}_{1,t}+\dd I^{\textrm{mart}}_{2,t} - 2\dd I^{\textrm{mart}}_{\textrm{mix},t}.\]
We will handle each of the four terms on the righthand side of \eqref{u_5} separately.

\textbf{The measure term}.  It is an immediate consequence of H\"older's inequality and the entropy condition of Definition~\ref{def_sol} that, $\P$-a.s.,
\begin{equation}\label{u_6} I^{\textrm{meas}}_t= \int_0^t\dd I^{\textrm{meas}}_{s}\geq 0.\end{equation}
\textbf{The martingale term}.  The analysis of the martingale term will use the following fact:  if $F^\ve_t,F\in L^2(\O\times[0,T])$ are $\F_t$-progressively measurable processes that satisfy that, as $\ve\rightarrow 0$, $F^\ve_t\rightarrow F$ in $L^2(\O\times[0,T])$, then along a subsequence $\ve\rightarrow 0$ we have that, $\P$-a.s.\ for every $t\in[0,T]$,
\begin{equation}\label{u_7}\lim_{\ve\rightarrow 0}\int_0^tF^\ve_s\dd B_s = \int_0^t F_s\dd B_s,\end{equation}
which follows from the Burkh\"older--Davis--Gundy inequality (see, for example, \cite[Chapter~4, Theorem~4.1]{RevYor1999}).  It follows from the local $H^1$-regularity of the solutions, the definition of the convolutions, and \eqref{u_7} that, $\P$-a.s.\ along a subsequence $\ve\rightarrow 0$, for $I^{\textrm{mart}}_t = \int_0^t\dd I^{\textrm{mart}}_s$,
\begin{align*}
\lim_{\ve\rightarrow 0}  I^{\textrm{mart}}_t & = \int_0^t\int_\R\int_{\TT^d}\overline{\kappa}^\d_{s,1}(2\chi^\d_{s,2}-1)\zeta_\beta(\eta)\nabla\cdot(\sqrt{\rho_1(1-\rho_1)}\dd\xi^F)\dy\deta
\\  & \quad + \int_0^t\int_\R\int_{\TT^d}\overline{\kappa}^\d_{s,2}(2\chi^\d_{s,1}-1)\zeta_\beta(\eta)\nabla\cdot(\sqrt{\rho_2(1-\rho_2)}\dd\xi^F)\dy\deta,
\end{align*}
for $\overline{\kappa}^\d_{s,i}(y,\eta)=\kappa^\d_1(\eta-\rho_i(y,t))$ and for $\chi^\d_{s,i}(y,\eta)=(\chi^i_s(y,\cdot)*\kappa^\d_1)(\eta)$.  It follows from the local $H^1$-regularity of the $\rho_i$ and the dominated convergence theorem that, after passing to a subsequence $\d\rightarrow 0$, $\P$-a.s.\ for every $t\in[0,T]$,
\begin{equation}\label{u_09}\lim_{\d\rightarrow 0}\abs{\int_0^t\int_\R\int_{\TT^d}\overline{\kappa}^\d_{s,1}(2\chi^\d_{s,2}-1)(\zeta_\beta(\eta)-\zeta_\beta(\rho_1))\nabla\cdot(\sqrt{\rho_1(1-\rho_1)}\dd\xi^F)}=0,\end{equation}
and similarly for the symmetric term obtained by swapping the roles of $\rho_1$ and $\rho_2$.  It then follows from an explicit calculation that, whenever $2\d<\rho_2(y,s)$,
\[ \int_{\R^2}\kappa^\d(\xi-\eta)\kappa^\d(\eta-\xi')\chi^2_s(y,\xi')\deta\dxip =\left\{\begin{aligned} & 0 && \textrm{if}\;\;\xi\leq -2\d\;\;\textrm{or}\;\;\xi\geq \rho_2(y,s)+2\d, \\ & \nicefrac{1}{2} && \textrm{if}\;\;\xi = 0\;\;\textrm{or}\;\;\xi = \rho_2(y,s), \\ & 1 && \textrm{if}\;\;2\d<\xi<\rho_2(y,s)-2\d.\end{aligned}\right.\]
Then, using $\zeta_\beta(0)=0$, we have for every $(y,\eta)\in\TT^d\times\R$ that
\begin{equation}\label{u_9}\lim_{\d\rightarrow 0}\Big(\int_\R\overline{\kappa}^\d_{s,1}(2\chi^\d_{s,2}-1)\deta\Big)\zeta_\beta(\rho_1) =  \Big(\mathbf{1}_{\{\rho_1=\rho_2\}}+2\mathbf{1}_{\{0\leq \rho_1<\rho_2\}}-1\Big)\zeta_\beta(\rho_1),\end{equation}
and similarly for the symmetric term obtained by swapping the roles of $\rho_1$ and $\rho_2$.  In combination \eqref{u_09}, \eqref{u_9}, and the equality $\sgn(\rho_2-\rho_1)=\mathbf{1}_{\{\rho_1=\rho_2\}}+2\mathbf{1}_{\{\rho_1<\rho_2\}}-1$ prove that, $\P$-a.s.\ there exist random subsequences $\ve,\d\rightarrow 0$ such that
\begin{align}\label{u_10}
\lim_{\d,\ve\rightarrow 0}\Big(I^\textrm{mart}_t\Big)  & =  \int_0^t\int_{\TT^d}\sgn(\rho_2-\rho_1)\zeta_\beta(\rho_1)\nabla\cdot(\sqrt{\rho_1(1-\rho_1)}\dd\xi^F)
\\ \nonumber &\quad - \int_0^t\int_{\TT^d}\sgn(\rho_2-\rho_1)\zeta_\beta(\rho_2)\nabla\cdot(\sqrt{\rho_2(1-\rho_2)}\dd\xi^F),
\end{align}
where it is not necessary to define $\sgn(0)$ since by Stampacchia's lemma (see Evans \cite[Chapter~5, Exercises~17,18]{Eva2010}), almost everywhere on the set $\{\rho_1=\rho_2\}$,
\[\Big(\zeta_\beta(\rho_1)\nabla\cdot(\sqrt{\rho_1(1-\rho_1)}\dd\xi^F)-\zeta_\beta(\rho_2)\nabla\cdot(\sqrt{\rho_2(1-\rho_2)}\dd\xi^F)\Big)=0.\]
For every $\beta\in(0,\nicefrac{1}{4})$ let $\Theta_\beta\colon[0,\infty)\rightarrow[0,\infty)$ satisfy $\Theta_\beta(0)=0$ and $\Theta_\beta'(\xi)=\zeta_\beta(\xi)\partial_\xi(\sqrt{\xi(1-\xi)})$.  Returning to \eqref{u_10}, we have that
\begin{align}\label{u_11}
& \lim_{\d,\ve\rightarrow 0}\Big(I^\textrm{mart}_t\Big)  =  \int_0^t\int_{\TT^d}\sgn(\rho_2-\rho_1)\nabla\cdot\Big(\Big(\Theta_\beta(\rho_1)-\Theta_\beta(\rho_2)\Big)\dd\xi^F\Big)
\\ \nonumber &  \quad +\int_0^t\int_{\TT^d}\sgn(\rho_2-\rho_1)\Big(\zeta_\beta(\rho_1)\sqrt{\rho_1(1-\rho_1)}-\Theta_\beta(\rho_1)\Big)\nabla\cdot\dd\xi^F
\\ \nonumber & \quad - \int_0^t\int_{\TT^d}\sgn(\rho_2-\rho_1)\Big(\zeta_\beta(\rho_2)\sqrt{\rho_2(1-\rho_2)}-\Theta_\beta(\rho_2)\Big)\nabla\cdot\dd\xi^F.
\end{align}
For the first term on the righthand side of \eqref{u_11}, for the convolution $\sgn^\d=(\sgn*\kappa^\d_1)$ for every $\d\in(0,1)$, we have using \eqref{u_7} that, $\P$-a.s.\ along a random subsequence $\d\rightarrow 0$, for every $t\in[0,T]$,
\begin{align}\label{u_12}
& \int_0^t\int_{\TT^d}\sgn(\rho_2-\rho_1)\nabla\cdot\Big(\Big(\Theta_\beta(\rho_1)-\Theta_\beta(\rho_2)\Big)\dd\xi^F\Big)
 \\ \nonumber   = & \lim_{\d\rightarrow 0}  \int_0^t\int_{\TT^d}\sgn^\d(\rho_1-\rho_2)\nabla\cdot\Big(\Big(\Theta_\beta(\rho_1)-\Theta_\beta(\rho_2)\Big)\dd\xi^F\Big)
\\ \nonumber  = & -\lim_{\d\rightarrow 0} \int_0^T\int_{\TT^d}(\sgn^\d)'(\rho_1-\rho_2)\Big(\Theta_\beta(\rho_1)-\Theta_\beta(\rho_2)\Big)(\nabla\rho_1-\nabla\rho_2)\cdot \dd\xi^F.
\end{align}
It follows from the Lipschitz continuity of $\Theta_\beta$ that there exists $c\in(0,\infty)$ independent of $\d\in(0,1)$ but depending on $\beta$ such that
\begin{equation}\label{u_14}\abs{(\sgn^\d)'(\rho_1-\rho_2)\big(\Theta_\beta(\rho_1)-\Theta_\beta(\rho_2)\big)} \leq c\mathbf{1}_{\{0<\abs{\rho_1-\rho_2}<c\d\}}.\end{equation}
We then have using the local $H^1$-regularity of the solutions, the dominated convergence theorem, \eqref{u_12}, and \eqref{u_14} that, $\P$-a.s.\ for every $t\in[0,T]$,
\begin{equation}\label{u_16}  \int_0^t\int_{\TT^d}\sgn(\rho_2-\rho_1)\nabla\cdot\Big(\Big(\Theta_{\beta,M}(\rho_1)-\Theta_{\beta,M}(\rho_2)\Big)\dd\xi^F\Big)=0.\end{equation}
For the the second two terms on the righthand side of \eqref{u_11}, it follows from the definition of $\Theta_\beta$ and the dominated convergence theorem that, for every $i\in\{1,2\}$,
\[\lim_{\beta\rightarrow 0}\Theta_\beta(\rho_i) = \sqrt{\rho_i(1-\rho_i)}\;\textrm{strongly in}\;L^2(\TT^d\times[0,T]).\]
Therefore, using \eqref{u_7}, there exists a random sequence $\beta\rightarrow 0$ such that, for every $i\in\{1,2\}$ and $t\in[0,T]$,
\begin{equation}\label{u_15}\lim_{\beta\rightarrow 0}\abs{\int_0^t\int_{\TT^d}\sgn(\rho_2-\rho_1)\Big(\zeta_\beta(\rho_1)\sqrt{\rho_1(1-\rho_1)}-\Theta_\beta(\rho_1)\Big)\nabla\cdot\dd\xi^F} =0.\end{equation}
In combination, \eqref{u_10}, \eqref{u_11}, \eqref{u_16}, and \eqref{u_15} prove that, $\P$-a.s.\ along random subsequences, for every $t\in[0,T]$,
\begin{equation}\label{u_17} \lim_{\beta\rightarrow 0}\Big(\lim_{\d\rightarrow 0}\Big(\lim_{\ve\rightarrow 0} I^\textrm{mart}_t\Big)\Big) =0.\end{equation}
\textbf{The correction term}.  The local $H^1$-regularity of the solutions and the support of $\zeta_\beta$ proves that, for every $t\in[0,T]$, for $I^{\textrm{cor}}_t = \int_0^t \dd I^{\textrm{cor}}_s$,
\begin{align*}
\lim_{\ve\rightarrow 0} I^{\textrm{cor}}_{t} & = -\frac{F_1}{8}\int_0^t\int_{\TT^d}\int_\R \zeta_\beta\overline{\kappa}^{\d}_{1,s}\overline{\kappa}^{\d}_{2,s}\Big(\frac{(1-2\rho_1)}{\sqrt{\rho_1(1-\rho_1)}}-\frac{(1-2\rho_2)}{\sqrt{\rho_2(1-\rho_2)}}\Big)^2\nabla\rho_1\cdot\nabla\rho_2
\\ & \quad -\frac{1}{2}\int_0^t\int_{\TT^d}\int_\R \zeta_\beta\overline{\kappa}^{\d}_{1,s}\overline{\kappa}^{\d}_{2,s}F_3\Big(\sqrt{\rho_1(1-\rho_1)}-\sqrt{\rho_2(1-\rho_2)}\Big)^2.
\end{align*}
Since on the support of $\zeta_\beta\overline{\kappa}^{\d}_{1,s}\overline{\kappa}^{\d}_{2,s}$ we have for every $\d\in(0,\nicefrac{\beta}{2})$ and $i\in\{1,2\}$ that $\nicefrac{\beta}{2}\leq \rho_i\leq 1-\nicefrac{\beta}{2}$, we have using the local Lipschitz continuity of the nonlinearities on $(0,1)$ that, for $c\in(0,\infty)$ independent of $\d\in(0,1)$ but depending on $\beta\in(0,1)$, for every $t\in[0,T]$,
\[\abs{\lim_{\ve\rightarrow 0} I^{\textrm{cor}}_t} \leq c\delta \int_0^T\int_{\TT^d}\int_\R \zeta_\beta (\delta \overline{\kappa}^\d_{1,s})\overline{\kappa}^\delta_{2,s} \nabla\rho_1\cdot\nabla\rho_2\leq c\delta \int_0^T\int_{\TT^d}\mathbf{1}_{\{\nicefrac{\beta}{4}\leq \rho_i\leq 1-\nicefrac{\beta}{4}\;\forall\;i\in\{1,2\}\}}\nabla\rho_1\cdot\nabla\rho_2. \]
It then follows from the local $H^1$-regularity of the solutions and the dominated convergence theorem that, for every $t\in[0,T]$,
\begin{equation}\label{u_18}\lim_{\d\rightarrow 0}\Big(\lim_{\ve\rightarrow 0} I^\textrm{cor}_t\Big) = 0.\end{equation}
\textbf{The cutoff term}.  The cutoff term takes the form, for every $t\in[0,T]$, for $I^{\textrm{cut}}_t = \int_0^t\dd I^{\textrm{cut}}_{s}$,
\begin{align*}
I^{\textrm{cut}}_t & = \int_0^t\int_{\TT^d}\int_{\R} (2\chi^{\ve,\d}_{2,s}-1)(\kappa^{\ve,\d}*q_{1,s})\zeta'_\beta+\int_0^t\int_{\TT^d}\int_{\R} (2\chi^{\ve,\d}_{1,s}-1)(\kappa^{\ve,\d}*q_{2,s})\zeta'_\beta
\\ & \quad -\frac{1}{2}\int_0^t\int_{\TT^d}\int_{\R} \big(\overline{\kappa}^{\ve,\d}_{1,s}*F_3(\rho_1(1-\rho_1)\big)(2\chi^{\ve,\d}_{2,s}-1)\zeta'_\beta
\\ & \quad -\frac{1}{2}\int_0^t\int_{\TT^d}\int_{\R} \big(\overline{\kappa}^{\ve,\d}_{2,s}*F_3\rho_2(1-\rho_2)\big)(2\chi^{\ve,\d}_{1,s}-1)\zeta'_\beta,
\end{align*}
and therefore, using the boundedness of the kinetic function and $F_3$ and the support of the convolution kernel, for some $c\in(0,\infty)$ independent of $\ve$, $\d$, and $\beta$, for every $t\in[0,T]$,
\begin{align*}
\limsup_{\ve,\d\rightarrow 0}\abs{I^{\textrm{cut}}_t } &  \leq c\Big(\int_0^T\int_{\TT^d}\int_\R \abs{\zeta'_\beta(\xi)}\dd q_1+\int_0^T\int_{\TT^d}\int_\R \abs{\zeta_\beta'(\xi')}\dd q_2\Big)
\\ & \quad  + c\Big(\int_0^T\int_{\TT^d}\rho_1(1-\rho_1)\abs{\zeta'_\beta(\rho_1)}+\int_0^T\int_{\TT^d}\rho_2(1-\rho_2)\abs{\zeta'_\beta(\rho_2)}\Big).
\end{align*}
Finally, using the bounds on $\zeta_\beta'$, for some $c\in(0,\infty)$ independent of $\ve$, $\d$, and $\beta$, for every $t\in[0,T]$,
\begin{align}\label{u_20}
\limsup_{\ve,\d\rightarrow 0}\abs{I^{\textrm{cut}}_t }& \leq c\sum_{i=1}^2\Big(\int_0^T\int_\beta^{2\beta}\int_{\TT^d} (\xi(1-\xi))^{-1}\dd q_i+\int_0^T\int_{1-2\beta}^{1-\beta}\int_{\TT^d}(\xi(1-\xi))^{-1}\dd q_i\Big)
\\ \nonumber & \quad +c\sum_{i=1}^2\Big( \int_0^T\int_{\TT^d}\mathbf{1}_{\{\beta<\rho_i<2\beta\}}+\mathbf{1}_{\{1-2\beta<\rho_i<1-\beta\}} \Big).
\end{align}
It is a consequence of the finiteness of the measures $\dd \mu_i = (\xi(1-\xi))^{-1}\dd q_i$ and the dominated convergence theorem that both terms on the righthand side of \eqref{u_20} vanish as $\beta\rightarrow 0$, and therefore that
\begin{equation}\label{u_21}\lim_{\beta\rightarrow 0}\Big(\limsup_{\ve,\d\rightarrow 0}\abs{I^{\textrm{cut}}_t }\Big) = 0.\end{equation}
Since we have using the fact that the kinetic functions are $\{0,1\}$-valued, for every $t\in[0,T]$,
\[\lim_{\beta,\d,\ve\rightarrow 0}\int_{\TT^d}\int_\R (\chi^{\ve,\d}_{1,t}+\chi^{\ve,\d}_{2,t}-2\chi^{\ve,\d}_{1,t}\chi^{\ve,\d}_{2,t})\zeta_\beta = \int_{\TT^d}\int_\R\abs{\chi_1-\chi_2}^2 = \int_{\TT^d}\abs{\rho_1-\rho_2},\]
and after returning to \eqref{u_1} and \eqref{u_5}, we have from \eqref{u_6}, \eqref{u_17}, \eqref{u_18}, and \eqref{u_21} that, $\P$-a.s.\ for every $t\in[0,T]$,
\[ \norm{\rho_1(\cdot,t)-\rho_2(\cdot,t)}_{L^1(\TT^d)}\leq \norm{\rho_{0,1}-\rho_{0,2}}_{L^1(\TT^d)},\]
which completes the proof.  \end{proof}

\subsection{Existence of renormalized kinetic solutions}\label{sec_rks_exist}  We will first show the existence of solutions to approximations of \eqref{intro_fd_eq} defined by smoothed versions of the square root, and including an $L^2$-valued control term that will be important for the proof the large deviations principle below.  The reason for the specific form of the regularizations in \eqref{lem_approx} appears in the proof of the entropy estimate in Proposition~\ref{prop_ap} below.

\begin{lem}\label{lem_approx}  Let $s\colon\R\rightarrow[0,1]$ be defined by
\[s(x)=\sqrt{x(1-x)}\;\textrm{if $x\in[0,1]$ and}\;s(x)=0\;\textrm{if $x\notin[0,1]$.}\]
Then there exists a sequence of smooth, compactly supported approximations $\{s^\eta\}_{\eta\in(0,\nicefrac{1}{4})}$ satisfying the properties that $s^\eta(x) = 0$ if $x\notin[\eta,1-\eta]$, that, as $\eta\rightarrow 0$,
\[s^\eta\rightarrow s\;\;\textrm{uniformly on $\R$ and}\;(s^\eta)'\rightarrow s'\;\textrm{locally uniformly on $(0,1)$,}\]
and that, for some $c\in(0,\infty)$ independent of $\eta\in(0,1)$,
\[s^\eta(x)\leq cs(x)\;\;\textrm{and}\;\;\abs{(s^\eta)'(x)}\leq c \abs{s'(x)}\;\;\textrm{for every}\;\;x\in(0,1).\]
\end{lem}

\begin{proof}  For every $\eta\in(0,\nicefrac{1}{4})$ let $\tilde{s}^\eta\in\C(\R)$ be defined by
\[\tilde{s}^\eta(x) = \sqrt{x(1-x)}-\sqrt{2\eta(1-2\eta)}\;\;\textrm{for every $x\in[2\eta,1-2\eta]$ and}\;\;\tilde{s}^\eta(x)=0\;\textrm{otherwise,}\]
and for every $\ve\in(0,1)$ let $\kappa^\ve\in\C^\infty_c(\R)$ denote a standard convolution kernel of scale $\ve\in(0,1)$.  Then, since $\tilde{s}^\eta(x)=s(x)$ and $(\tilde{s}^\eta)'(x)=s'(x)$ for every $x\in(2\eta,1-2\eta)$ and $\tilde{s}^\eta(x)=(\tilde{s}^\eta)'(x)=0$ if $x\notin[2\eta,1-2\eta]$, for every $\eta\in(0,\nicefrac{1}{4})$ there exists $\ve_\eta\in(0,\eta)$ such that the functions defined by $s^\eta=(\tilde{s}^\eta*\kappa^{\ve_\eta})$ satisfy the hypothesis of the lemma with $c=2$.  This completes the proof.  \end{proof}

We will now prove the existence of solutions to the controlled Stratonovich equation
\begin{equation}\label{ex_1}\partial_t\rho = \Delta\rho-\sqrt{\ve}\nabla\cdot\Big(s^\eta(\rho)\circ\dd\xi^F\Big)-\nabla\cdot\Big(s^\eta(\rho)g\Big)\;\textrm{in}\;\TT^d\times(0,T)\;\textrm{with}\;\;\rho(\cdot,0)=\rho_0,\end{equation}
for noise $\xi^F$ and initial data satisfying Assumption~\ref{def_noise} and for an $\F_t$-predictable, $L^2$-integrable control $g$.  We will write \eqref{ex_1} in its It\^o-form
\begin{equation}\label{ex_2}\partial_t\rho = \Delta\rho-\sqrt{\ve}\nabla\cdot\Big(s^\eta(\rho)\dd\xi^F\Big)-\nabla\cdot\Big(s^\eta(\rho)g\Big)+\frac{\ve F_1}{2}\nabla\cdot \Big( (s^\eta)'(\rho)\nabla s^\eta(\rho))\Big),\end{equation}
and understand solutions in the equation's weak-formulation.  We will first define the notion of a weak solution in Definition~\ref{def_weak_solution}, and establish stable with respect to $\eta\in(0,1)$ energy estimates for the solutions in Proposition~\ref{prop_ap}.  We prove the existence of solutions to \eqref{ex_2} in Proposition~\ref{prop_weak_exist}.

\begin{definition}\label{def_weak_solution}  Let $T\in(0,\infty)$, let $\xi^F$ and $\rho_0$ satisfy Assumption~\ref{def_noise}, let $\{s^\eta\}_{\eta\in(0,\nicefrac{1}{4})}$ satisfy the conditions of Lemma~\ref{lem_approx}, let $\ve,\eta\in(0,\nicefrac{1}{4})$, and let $g\in L^2(\O\times[0,T];L^2(\TT^d))^d$ be an $\F_t$-predictable process.   A \emph{weak solution} of \eqref{ex_2} is a $\P$-a.s.\ $L^2(\TT^d;[0,1])$-continuous, $\F_t$-adapted process $\rho\in L^2([0,T]\times\O;H^1(\TT^d))$ that satisfies, for every $\psi\in\C^\infty(\TT^d)$ and $t\in[0,T]$,
\begin{align}\label{dwp_1}
& \int_{\TT^d}\rho(x,t)\psi(x)\dx  = \int_{\TT^d}\rho_0(x)\psi(x)\dx -\int_0^t\int_{\TT^d}\nabla\rho(x,s)\cdot\nabla\psi(x)\dx\ds
\\ \nonumber & \quad +\sqrt{\ve}\int_0^t\int_{\TT^d}s^\eta(\rho(x,s))\nabla \psi(x)\cdot \dd\xi^K  + \int_0^t\int_{\TT^d}s^\eta(\rho(x,s))\nabla\psi(x)\cdot g(x,s)\dx\ds
\\ \nonumber & \quad -\frac{\ve F_1}{2}\int_0^t\int_{\TT^d}((s^\eta)'(\rho(x,s)))^2\nabla \rho(x,s)\cdot \nabla\psi(x)\dx\ds.
\end{align}
\end{definition}

\begin{prop}\label{prop_ap}  Let $T\in(0,\infty)$, let $\xi^F$ and $\rho_0$ satisfy Assumption~\ref{def_noise}, let $\{s^\eta\}_{\eta\in(0,\nicefrac{1}{4})}$ satisfy Lemma~\ref{lem_approx}, let $\ve,\eta\in(0,\nicefrac{1}{4})$, and let $g\in L^2(\O\times[0,T];L^2(\TT^d))^d$ be an $\F_t$-predictable process.  Then, if $\rho$ is a weak solution of \eqref{ex_2} in the sense of Definition~\ref{def_weak_solution}, we have the following four estimates.
\begin{enumerate}
\item  \emph{Preservation of mass}:  $\P$-a.s.\ for every $t\in[0,T]$,
\[\norm{\rho(\cdot,t)}_{L^1(\TT^d)}=\norm{\rho_0}_{L^1(\TT^d)}.\]
\item \emph{The energy estimate}:  for every $\a\in[1,\infty)$ there exists $c\in(0,\infty)$ independent of $\ve$, $\eta$, and $\a$ such that, for $\norm{F_3}=\max_{x\in\TT^d}\abs{F_3(x)}$,
\begin{align*}
& \E\Big[\max_{t\in[0,T]}\norm{\rho(\cdot,t)}^{\a+1}_{L^{\a+1}(\TT^d)}+\int_0^T\int_{\TT^d}\abs{\nabla\rho^\frac{\a+1}{2}}^2\Big]
\\ & \leq c\E\Big[\int_{\TT^d}\rho_0^{\a+1}\Big]+c\a^2\Big(\E\Big[\int_0^T\int_{\TT^d}\abs{g}^2\Big]+\ve\norm{F_3}\E\Big[\int_0^T\int_{\TT^d}\rho^{\a-1}\Big]\Big).
\end{align*}

\item \emph{The entropy estimate}: for $\Psi\colon(0,1)\rightarrow \R$ defined by
\[\Psi(\xi) = \Big(\xi\log(\xi)-\xi\Big)-\Big((1-\xi)\log(1-\xi)-(1-\xi)\Big),\]
for $c\in(0,\infty)$ independent of $\ve$ and $\eta$,
\begin{align*}
& \E\Big[\max_{t\in[0,T]}\Big(\int_{\TT^d}\Psi(\rho)\Big)+\int_0^T\int_{\TT^d}\frac{1}{\rho(1-\rho)}\abs{\nabla\rho}^2\dt\Big]
\\ & \leq \E\Big[\int_{\TT^d}\Psi(\rho_0)\Big]+c\Big(\E\Big[\int_0^T\int_{\TT^d}\abs{g}^2\Big]+ \ve (F_1+\norm{F_3}T)\Big).
\end{align*}
\item \emph{The time regularity}:  for every $\beta\in(0,\nicefrac{1}{2})$, for $c\in(0,\infty)$ independent of $\ve$ and $\eta$,
\begin{align*}
& \E\Big[\norm{\rho}^2_{W^{\beta,2}([0,T];H^{-1}(\TT^d))}\Big]
\\ & \leq c\Big(1+\ve^2 F_1^2\norm{(s^\eta)'}^4_{L^\infty(\R)}\Big)\Big(\E\Big[\int_{\TT^d}\rho_0^2\Big]+\E\Big[\int_0^T\int_{\TT^d}\abs{g}^2\Big]+\ve(F_1+\norm{F_3}T\Big).
\end{align*}
\end{enumerate}
\end{prop}

\begin{proof}  The $L^1$-estimate follows from taking $\psi=1$ in Definition~\ref{def_weak_solution} and the fact that $\rho$ is $\P$-a.s.\ $[0,1]$-valued.  The $L^{\a+1}$-estimate for $\a\in[1,\infty)$ is a consequence of It\^o's formula (see, for example, Krylov~\cite[Theorem~3.1]{Kry2013}) applied to the function $\rho^{\a+1}$.  We have $\P$-a.s.\ that
\begin{align}\label{ap_0}
& \dd\Big(\int_{\TT^d}\rho^{\a+1}\Big)+\a(\a+1)\int_{\TT^d}\rho^{\a-1}\abs{\nabla\rho}^2\dt  = \sqrt{\ve}\a(\a+1)\int_{\TT^d}s^\eta(\rho)\rho^{\a-1}\nabla\rho\cdot\dd\xi^F
\\ \nonumber & \quad +\a(\a+1)\int_{\TT^d}s^\eta(\rho)\rho^{\a-1}\nabla\rho\cdot g \dt + \frac{\ve \a(\a+1)}{2}\int_{\TT^d}F_3(s^\eta(\rho))^2\rho^{\a-1}\dt,
\end{align}
where here we have already observed the cancellation between the final term on the righthand side of \eqref{dwp_1} and part of the It\^o-correction.  Since $0\leq\rho\leq 1$ we have that $\rho^{\a-1}\leq\rho^\frac{\a-1}{2}$ and therefore, using the boundedness of $s^\eta$, H\"older's inequality, and Young's inequality, for every $\d\in(0,1)$,
\[\a(\a+1)\int_{\TT^d}s^\eta(\rho)\rho^{\a-1}\nabla\rho\cdot g \leq \frac{\a(\a+1)\d}{2}\int_{\TT^d}\rho^{\a-1}\abs{\nabla\rho}^2+\frac{\a(\a+1)}{2\d}\int_{\TT^d}\abs{g}^2.\]
Returning to \eqref{ap_0}, after choosing $\d$ sufficiently small, we have $\P$-a.s.\ for some $c\in(0,\infty)$ that
\begin{align}\label{ap_20}
& \max_{t\in[0,T]}\norm{\rho(\cdot,t)}^{\a+1}_{L^{\a+1}(\TT^d)}+\a(\a+1)\int_0^T\int_{\TT^d}\rho^{\a-1}\abs{\nabla\rho}^2
\\ \nonumber & \leq \norm{\rho_0}^{\a+1}_{L^{\a+1}(\TT^d)} + c\a(\a+1)\Big(\int_0^T\int_{\TT^d}\abs{g}^2+\ve\norm{F_3}\int_0^T\int_{\TT^d} \rho^{\a-1}\Big)
\\ \nonumber & \quad + c\sqrt{\ve}\a(\a+1)\Big(\max_{t\in[0,T]}\abs{\int_0^t\int_{\TT^d}s^\eta(\rho)\rho^{\a-1}\nabla\rho\cdot\dd\xi^F}\Big),
\end{align}
for $\norm{F}_3 = \max_{x\in\TT^d}\abs{F_3(x)}$.  For the stochastic integral, let $\Theta^\eta$ be defined for every $\xi\in[0,1]$ by
\[\Theta^\eta(\xi) = \int_0^\xi s^\eta(\xi')(\xi')^{\a-1}\dxip\leq \int_0^\xi (\xi')^{\a-1}\dxip\leq \frac{1}{\a}\xi^\a,\]
and observe that, after integrating by parts, for every $t\in[0,T]$,
\[\int_0^t\int_{\TT^d}s^\eta(\rho)\rho^{\a-1}\nabla\rho\cdot\dd \xi^F = -\int_0^t\int_{\TT^d}\Theta^\eta(\rho)\nabla\cdot\dd\xi^F.\]
The Burkholder--Davis--Gundy inequality, H\"older's inequality, and Young's inequality prove that, for every $\d\in(0,1)$, for some $c\in(0,\infty)$,
\begin{align}\label{fgd_2}
& \E_{\F_0}\Big[\sqrt{\ve}\a(\a+1)\max_{t\in[0,T]}\abs{\int_0^t\int_{\TT^d}s^\eta(\rho)\nabla\rho\cdot\dd \xi^F}\Big]
\\ \nonumber & \leq c\sqrt{\ve}(\a+1)\E_{\F_0}\Big[\Big(\int_0^T\Sum_{i=1}^d\sum_{k=1}^\infty\Big(\int_{\TT^d}\rho^{\a}\partial_if_k\Big)^2\Big)^\frac{1}{2}\Big]
\\ \nonumber & \leq c\sqrt{\ve}(\a+1)\E_{\F_0}\Big[\Big(\max_{t\in[0,T]}\int_{\TT^d}\rho^{\a+1}\Big)^\frac{1}{2}\Big(\int_0^T\int_{\TT^d}\rho^{\a-1}F_3\Big)^\frac{1}{2}\Big]
\\ \nonumber & \leq \frac{\d}{2}\E_{\F_0}\Big[\max_{t\in[0,T]}\int_{\TT^d}\rho^{\a+1}\Big]+\frac{\ve(\a+1)^2\norm{F_3}}{2\d}\E_{\F_0}\Big[\int_0^T\int_{\TT^d}\rho^{\a-1}\Big].
\end{align}
Returning to \eqref{ap_20}, after choosing $\d$ sufficienty small in \eqref{fgd_2} and using that $(\a+1)\leq 2\a$ for every $\a\in[1,\infty)$, for some $c\in(0,\infty)$ independent of $\ve$, $\eta$, and $\a$,
\begin{align*}
& \E\Big[\max_{t\in[0,T]}\norm{\rho(\cdot,t)}^{\a+1}_{L^{\a+1}(\TT^d)}+\int_0^T\int_{\TT^d}\abs{\nabla\rho^\frac{\a+1}{2}}^2\Big]
\\ & \leq c\E\Big[\int_{\TT^d}\rho_0^{\a+1}\Big]+c\a^2\Big(\E\Big[\int_0^T\int_{\TT^d}\abs{g}^2\Big]+\ve\norm{F_3}\E\Big[\int_0^T\int_{\TT^d}\rho^{\a-1}\Big]\Big),
\end{align*}
where in the final step we have used the equality $(\a+1)^2\rho^{\a-1}\abs{\nabla\rho}^2 = 4\abs{\nabla\rho^\frac{\a+1}{2}}^2$.  This completes the proof of the second energy estimate.

For the entropy estimate, let $\Psi\colon[0,1]\rightarrow\R$ denote the bounded, continuous function
\[\Psi(\xi) = \Big(\xi\log(\xi)-\xi\Big)- \Big((1-\xi)\log(1-\xi)-(1-\xi)\Big),\]
for which we have that $\Psi'(\xi) = \log(\xi)+\log(1-\xi)$ and $\Psi''(\xi) = (\xi(1-\xi))^{-1}$ on $(0,1)$.  For every $\d\in(0,1)$ let $\Psi_\d\colon [0,1]\rightarrow\R$ denote the unique smooth function satisfying $\Psi_\d(\nicefrac{1}{2})=\Psi(\nicefrac{1}{2})$ and $\Psi_\d''(\xi) = (\xi(1-\xi)+\d)^{-1}$.  It is then a consequence of It\^o's formula (see, for example, \cite[Theorem~3.1]{Kry2013}) applied to the function $\Psi(\rho)$ that, $\P$-a.s.\ for every $t\in[0,T]$,
\begin{align*}
& \dd\Big(\frac{1}{2}\int_{\TT^d}\Psi_\d(\rho)\Big)+\int_{\TT^d}\frac{1}{\rho(1-\rho)+\d}\abs{\nabla\rho}^2\dt
\\ & = \sqrt{\ve}\int_{\TT^d}\frac{s^\eta(\rho)}{\rho(1-\rho)+\d}\nabla\rho\cdot\dd\xi^F+\int_{\TT^d}\frac{s^\eta(\rho)}{\rho(1-\rho)+\d}\nabla\rho\cdot g \dt + \frac{\ve}{2}\int_{\TT^d}\frac{F_3(s^\eta(\rho))^2}{\rho(1-\rho)+\d}\dt,
\end{align*}
where here, as in the above, we have already observed the cancellation between the final term of \eqref{dwp_1} and part of the It\^o-correction.  It then follows from the fact that there exists $c\in(0,\infty)$ independent of $\eta$ such that $s^\eta(x)\leq c s(x)$ and from H\"older's inequality and Young's inequality that, for some $c\in(0,\infty)$ independent of $\ve$ and $\eta$,
\begin{align*}
& \max_{t\in[0,T]}\int_{\TT^d}\Psi_\d(\rho)+\int_0^T\int_{\TT^d}\frac{1}{\rho(1-\rho)+\d}\abs{\nabla\rho}^2\dt
\\ & \leq  \int_{\TT^d}\Psi_\d(\rho_0)+c\Big(\sup_{t\in[0,T]}\Big(\sqrt{\ve}\int_0^t\int_{\TT^d}\frac{s^\eta(\rho)}{\rho(1-\rho)+\d}\nabla\rho\cdot\dd\xi^F\Big)+\int_0^T\int_{\TT^d}\abs{g}^2+ \ve\norm{F_3}T\Big).
\end{align*}
A repetition of the argument above relying on the Burkh\"older--Davis--Gundy inquality, H\"older's inequality, Young's inequality, Jensen's inequality, and $s^\eta(x)\leq c s(x)$ then proves that, for some $c\in(0,\infty)$ independent of $\ve$, $\eta$, and $\d$,
\[\E\Big[\max_{t\in[0,T]}\abs{c\sqrt{\ve}\int_0^t\int_{\TT^d}\frac{s^\eta(\rho)}{\rho(1-\rho)+\d}\nabla\rho\cdot\dd\xi^F}\Big] \leq c\ve F_1+\frac{1}{2}\int_0^T\int_{\TT^d}\frac{1}{(\rho(1-\rho))+\d}\abs{\nabla\rho}^2.\]
Therefore, for some $c\in(0,\infty)$ independent of $\ve$, $\eta$, and $\delta$,
\begin{align*}
& \E\Big[\max_{t\in[0,T]}\int_{\TT^d}\Psi_\d(\rho)+\int_0^T\int_{\TT^d}\frac{1}{\rho(1-\rho)+\d}\abs{\nabla\rho}^2\dt\Big]
\\ & \leq \E\Big[\int_{\TT^d}\Psi_\d(\rho_0)\Big]+c\Big(\E\Big[\int_0^T\int_{\TT^d}\abs{g}^2\Big]+ \ve (F_1+T\norm{F_3})\Big).
\end{align*}
The claim then follows after passing to the limit $\d\rightarrow 0$ using the definition of $\Psi_\d$ and the monotone convergence theorem.

It remains to estimate the time derivative.  We observe distributionally that, $\P$-a.s.\ for every $t\in[0,T]$,
\[\rho(\cdot,t)=\rho_0(\cdot)+I^{\textrm{f.v.}}_t(\cdot)+I^{\textrm{mart}}_t(\cdot),\]
for the finite variation part
\[I^{\textrm{f.v.}}_t(\cdot) = \nabla\cdot\Big(\int_0^t\nabla\rho(\cdot,s)\ds-\int_0^ts^\eta(\rho(\cdot,s))g\ds+\frac{\ve F_1}{2}\int_0^t((s^\eta)'(\rho(\cdot,s)))^2\nabla\rho(\cdot,s)\ds \Big),\]
and for the martingale part
\[I^{\textrm{mart}}_t(\cdot) = -\nabla\cdot\Big(\sqrt{\ve}\int_0^ts^\eta(\rho(\cdot,s))\dd\xi^F\Big).\]
We then have that using the boundedness of $s^\eta$ that, for $c\in(0,\infty)$ independent of $\ve$ and $\eta$,
\begin{align*}
& \norm{I^{\textrm{f.v.}}}^2_{W^{1,2}([0,T];H^{-1}(\TT^d))}
\\ & \leq c\Big(\Big(1+\ve^2 F_1^2\norm{(s^\eta)'}^4_{L^\infty(\R)}\Big)\norm{\nabla\rho}^2_{L^2(\TT^d\times[0,T])^d} +\norm{g}^2_{L^2([0,T];L^2(\TT^d;\R^d)} \Big).
\end{align*}
For the martingale part, we have by definition of the fractional Sobolev norm and the definition of the noise that, for every $\beta\in(0,\nicefrac{1}{2})$,
\[\E\norm{I^{\textrm{mart}}}^2_{W^{\beta,2}([0,T];H^{-1}(\TT^d))} = \ve \int_0^T\int_0^T\frac{\E\Big[\abs{\sum_{k=1}^\infty \int_s^t\Big(\int_{\TT^d}f_k^2(s^\eta(\rho))^2\Big)^\frac{1}{2}\dd B^k_t}^2\Big]}{\abs{s-t}^{1+2\beta}}\ds\dt.\]
It follows from H\"older's inequality, the Burkh\"older--Davis--Gundy inequality, the definition of the noise, and the boundedness of $s^\eta$ that, for some $c\in(0,\infty)$ independent of $\ve$ and $\eta$, for every $s\leq t\in [0,T]$,
\[\E\Big[\abs{\sum_{k=1}^\infty \int_s^t\Big(\int_{\TT^d}f_k^2(s^\eta(\rho))^2\Big)^\frac{1}{2}\dd B^k_t}^2\Big]\leq c F_1\abs{s-t}.\]
Therefore, since $\beta\in(0,\nicefrac{1}{2})$, there exists $c\in(0,\infty)$ independent of $\ve$ and $\eta$ but depending on $T$ and $\beta$ such that
\[\E\norm{I^{\textrm{mart}}}^2_{W^{\beta,2}([0,T];H^{-1}(\TT^d))}\leq c\ve F_1.\]
It then follows from the embedding of $W^{2,1}$ into $W^{\beta,2}$ and the energy estimate that
\begin{align*}
& \E\Big[\norm{\rho}^2_{W^{\beta,2}([0,T];H^{-1}(\TT^d))}\Big]
\\ & \leq c\Big(1+\ve^2 F_1^2\norm{(s^\eta)'}^4_{L^\infty(\R)}\Big)\Big(\E\Big[\int_{\TT^d}\rho_0^2\Big]+\E\Big[\int_0^T\int_{\TT^d}\abs{g}^2\Big]+\ve(F_1+\norm{F_3}T)\Big),
\end{align*}
which completes the proof. \end{proof}

\begin{prop}\label{prop_weak_exist}   Let $T\in(0,\infty)$, let $\xi^F$ and $\rho_0$ satisfy Assumption~\ref{def_noise}, let $\{s^\eta\}_{\eta\in(0,\nicefrac{1}{4})}$ satisfy the conditions of Lemma~\ref{lem_approx}, let $\ve,\eta\in(0,1)$, and let $g\in L^2(\O\times[0,T];L^2(\TT^d))^d$ be an $\F_t$-predictable process.  Then there exists a weak solution of \eqref{ex_2} in the sense of Definition~\ref{def_weak_solution}.  \end{prop}

\begin{proof}  The proof of existence is a consequence of the smoothness and definition of the coefficients and noise, a standard Galerkin approximation, the estimates of Proposition~\ref{prop_ap}, and the Aubin--Lions--Simon lemma, and specifically Simon \cite[Corollary~5]{Simon}, relying on the compact embedding of $H^1(\TT^d)$ into $L^2(\TT^d)$ and the continuous embedding of $L^2(\TT^d)$ into $H^{-1}(\TT^d)$.  The fact that the solution is $[0,1]$-valued follows by applying It\^o's formula to the functions $(\rho)_-$ and $(\rho-1)_+$, which completes the proof. \end{proof}

We will now prove the existence of stochastic kinetic solutions solutions to the equation
\begin{equation}\label{ex_3}\partial_t\rho = \Delta\rho-\sqrt{\ve}\nabla\cdot(\sqrt{\rho(1-\rho)}\circ\dd\xi^F)+\nabla\cdot(\sqrt{\rho(1-\rho)}g)\;\textrm{in $\TT^d\times(0,T)$ with}\;\rho(\cdot,0)=\rho_0,\end{equation}
in the sense of Definition~\ref{def_sol}.  The essential difficulty is that the singularity appearing due to the Stratonovich-to-It\^o correction makes it intractable to obtain stable estimates on the time-derivative of the solution.  This issue also appears in the estimate on the time regularity in Proposition~\ref{prop_ap}, which diverges in the limit $\eta\rightarrow 0$.  It is for this reason that we instead consider renormalizations of the solution that localize it away from the sets $\{\rho\simeq 0\}$ and $\{\rho\simeq 1\}$.  These renormalizations are defined in Definition~\ref{def_Theta} and are used in Proposition~\ref{def_metric} to define an equivalent metric for the strong topology on $L^2(\TT^d\times[0,T])$ that is used in Proposition~\ref{prop_eta_tight} to prove the tightness of the approximating probability laws defined by the $\rho^\eta$ and to prove existence in Theorem~\ref{thm_new}.

\begin{definition}\label{def_Theta}  For every $\d\in(0,\nicefrac{1}{4})$ let $\theta_\d\in\C^\infty_c(\R;[0,1])$ be a function satisfying that $\theta_\d(\xi) = 0$ if $\xi\leq \d$ or $\xi\geq 1-\d$ and that $\theta_\d(\xi)=1$ if $2\d\leq \xi\leq 1-2\d$.  For every $\d\in(0,\nicefrac{1}{4})$ let $\Theta_\d\colon[0,\infty)\rightarrow[0,1]$ be defined by $\Theta_\d(\xi) = \xi\theta_\d(\xi)$.  \end{definition}

\begin{prop}\label{prop_Theta_ap}  Let $T\in(0,\infty)$, let $\xi^F$ and $\rho_0$ satisfy Assumption~\ref{def_noise}, let $\ve\in(0,1)$, let $g\in L^2(\O\times[0,T];L^2(\TT^d))^d$ be an $\F_t$-predictable process, and for every $\eta\in(0,\nicefrac{1}{4})$ let $\rho^\eta$ be the solution of \eqref{ex_2} constructed in Proposition~\ref{prop_weak_exist}.  Then,  for every $\d\in(0,\nicefrac{1}{4})$ and $s > \frac{d+2}{2}$, there exists $c\in(0,\infty)$ independent of $\eta$ and $\ve$ such that
\begin{align*}
\E \norm{\Theta_\d(\rho^\eta)}_{W^{\beta,1}([0,T];H^{-s}(\TT^d))}& \leq c\E\Big[\int_0^T\int_{\TT^d}\abs{\nabla\rho^\eta}+s^\eta(\rho^\eta)\abs{g}+(1+\ve F_1)\abs{\nabla\rho^\eta}^2 \Big]
\\ & \quad +c\sqrt{\ve F_1}\E\Big[\int_0^T\int_{\TT^d}s^\eta(\rho^\eta)^2+\abs{\nabla\rho^\eta}^2\Big]^\frac{1}{2}.
\end{align*}
\end{prop}

\begin{proof}  For the functions $\{s^\eta\}_{\eta\in(0,\nicefrac{1}{4})}$ satisfying the conditions of Lemma~\ref{lem_approx}, it follows from It\^o's formula that distributionally, $\P$-a.s.\ for every $t\in[0,T]$,
\[\Theta_\d(\rho^\eta) = \Theta_\d(\rho_0)+I^{\textrm{f.v.}}_t+I^{\textrm{mart.}}_t,\]
for the finite variation part
\begin{align*}
I^{\textrm{f.v.}}_t & = \nabla\cdot\Big(\int_0^t\nabla\Theta_\d(\rho^\eta)-\int_0^ts^\eta(\rho^\eta)g\ds +\frac{\ve F_1}{2}\int_0^t((s^\eta)'(\rho^\eta))^2\nabla\Theta_\d(\rho^\eta)\ds \Big)
\\ & \quad -\int_0^t\Theta_\d''(\rho^\eta)\abs{\nabla\rho^\eta}^2-\frac{\ve F_1}{2}\int_0^t((s^\eta)'(\rho^\eta))^2\Theta_\d''(\rho^\eta)\abs{\nabla\rho^\eta}^2,
\end{align*}
and for the martingale part
\[I^{\textrm{mart.}}_t = \sqrt{\ve}\nabla\cdot\Big(\int_0^t\Theta'_\d(\rho^\eta)s^\eta(\rho^\eta)\dd\xi^F\Big)-\sqrt{\ve}\int_0^t\Theta''_\d(\rho^\eta)s^\eta(\rho^\eta)\nabla\rho^\eta\cdot\dd\xi^F.\]
There are two essential differences in this case, when compared to the proof of time-regularity in Proposition~\ref{prop_ap}.  The first is that $\Theta'_\d$ and $\Theta''_\d$ are bounded by a constant depending on $\d$ and are supported on the set $\d\leq \xi\leq 1-\d$ on which $\abs{(s^\eta)'(\xi)}$ remains bounded independently of $\eta$.   The second are the terms arising from the It\^o-correction, which are only $L^1$-integrable.  It is for this reason that here we obtain estimates in $H^{-s}$ for $s>\frac{d+2}{2}$, since by the Sobolev embedding theorem we have that $H^s$ embeds into $W^{1,\infty}$.

Repeating the methods of Proposition~\ref{prop_ap}, it follows from the definition of $\Theta_\d$ that there exists $c\in(0,\infty)$ depending on $\d$ and $\sup_{\xi\in[\d,1-\d]}\abs{(s^\eta)'(\xi)}$ but independent of $\eta$ and $\ve$ such that
\[\norm{I^{\textrm{f.v.}}_t}_{W^{1,1}([0,T];H^{-s}(\TT^d))}\leq c\Big(\int_0^T\int_{\TT^d}\abs{\nabla\rho^\eta}+s^\eta(\rho^\eta)\abs{g}+(1+\ve F_1)\abs{\nabla\rho^\eta}^2 \Big),\]
and for the martingale term, it follows from the Burkh\"older--Davis--Gundy inequality and H\"older's inequality that, for every $\beta\in(0,\nicefrac{1}{2})$, for $c\in(0,\infty)$ depending on $\d$ but independent of $\ve$ and $\eta$,
\begin{align*}
& \E\norm{I^{\textrm{mart}}}^2_{W^{\beta,1}([0,T];H^{-s}(\TT^d))} = \sqrt{\ve} \int_0^T\int_{\TT^d}\int_0^T\frac{\E\Big[\abs{\sum_{k=1}^\infty \int_s^t\Big(\int_{\TT^d}\abs{f_k}(s^\eta(\rho^\eta)+\abs{\nabla\rho^\eta})\Big)\dd B^k_r}\Big]}{\abs{s-t}^{1+2\beta}}\ds\dt
\\ & \leq c\sqrt{\ve F_1}\int_0^T\int_0^T\frac{\E\Big[\int_s^t\int_{\TT^d}s^\eta(\rho^\eta)^2+\abs{\nabla\rho^\eta}^2\dr\Big]}{\abs{s-t}^{1+2\beta}}\ds\dt \leq c\sqrt{\ve F_1}\E\Big[\int_0^T\int_{\TT^d}s^\eta(\rho^\eta)^2+\abs{\nabla\rho^\eta}^2\Big]^\frac{1}{2},
\end{align*}
which, together with the embedding of $W^{1,1}$ into $W^{\beta,1}$ and the $L^2$-energy estimate for $\nabla\rho^\eta$, completes the proof.  \end{proof}

\begin{prop}\label{def_metric}  Let $D\colon L^2(\TT^d\times[0,T];[0,1])\times L^2(\TT^d\times[0,T];[0,1])\rightarrow[0,1]$ be defined by
\[D(f,g) = \abs{\langle f\rangle-\langle g\rangle}+ \sum_{k=1}^\infty 2^{-k}\frac{\norm{\Theta_{\nicefrac{1}{5k}}(f)-\Theta_{\nicefrac{1}{5k}}(g)}_{L^2(\TT^d\times[0,T])}}{1+\norm{\Theta_{\nicefrac{1}{5k}}(f)-\Theta_{\nicefrac{1}{5k}}(g)}_{L^2(\TT^d\times[0,T])}},\]
for $\langle f\rangle=\int_0^T\int_{\TT^d} f$.  Then $D$ defines a metric on $L^2(\TT^d\times[0,T];[0,1])$ that is equivalent to the strong norm-induced metric on $L^2(\TT^d\times[0,T];[0,1])$.
\end{prop}

\begin{proof}  It follows from Definition~\ref{def_Theta} that $f=g$ in $L^2(\TT^d\times[0,T];[0,1])$ if and only if $\Theta_{\nicefrac{1}{5k}}(f) = \Theta_{\nicefrac{1}{5k}}(g)$ for every $k\in\N$ and $\langle f\rangle = \langle g\rangle$, where the averages are used to differentiate the functions $f=0$ and $g=1$.  This completes the proof that $D(f,g)=0$ if and only if $f=g$.  The symmetry is a consequence of the definition of $D$ and the symmetry of the norm-induced metric on $L^2$ and the triangle inequality is a consequence of the triangle inequality for the norm-induced metric on $L^2$ and the concavity of the function $\xi\mapsto \xi(1+\xi)^{-1}$ for $\xi\in[0,\infty)$.

In order to prove that $D$ is equivalent to the strong norm-induced metric it suffices to prove that they determine the same convergent sequences.  It follows from the definition of the $\Theta_\d$ that
\[\norm{\Theta_{\nicefrac{1}{5k}}(f)-\Theta_{\nicefrac{1}{5k}}(g)}_{L^2(\TT^d\times[0,T])}\leq \norm{f-g}_{L^2(\TT^d\times[0,T])},\]
and similarly for the averages.  Therefore, if a sequence is convergent with respect to the norm-induced metric it also converges with respect to $D$.  Conversely, for every $k\in\N$, it follows from Definition~\ref{def_Theta} that, for some $c\in(0,\infty)$ independent of $k\in\N$,
\[\norm{f-g}_{L^2(\TT^d\times[0,T])}\leq \norm{\Theta_{\nicefrac{1}{5k}}(f)-\Theta_{\nicefrac{1}{5k}}(g)}_{L^2(\TT^d\times[0,T])}+\frac{c}{k},\]
from which it follows that a convergent sequence with respect to $D$ must also converge with respect to the norm-induced metric.  This completes the proof.  \end{proof}

\begin{prop}\label{prop_eta_tight}  Let $T\in(0,\infty)$, let $\xi^F$ and $\rho_0$ satisfy Assumption~\ref{def_noise}, let $\ve\in(0,1)$, let $g\in L^2(\O\times[0,T];L^2(\TT^d))^d$ be an $\F_t$-predictable process, and for every $\eta\in(0,\nicefrac{1}{4})$ let $\rho^\eta$ be the solution of \eqref{ex_2} constructed in Proposition~\ref{prop_weak_exist}.  Then the laws of $\{\rho^\eta\}_{\eta\in(0,\nicefrac{1}{4})}$ are tight on $L^2(\TT^d\times[0,T])$.\end{prop}

\begin{proof}  We will first show that, for every $k\in\N$, the laws of the $\{\Theta_{\nicefrac{1}{5k}}(\rho^\eta)\}_{\eta\in(0,\nicefrac{1}{4})}$ are tight on $L^2(\TT^d\times[0,T])$.  It is a consequence of Definition~\ref{def_Theta}, and in particular the boundedness and Lipschitz continuity of $\Theta_\d$, that for every $\k\in\N$ there exists $c\in(0,\infty)$ depending on $k$ such that
\[\sup_{\eta\in(0,\nicefrac{1}{4})}\E \norm{\Theta_{\nicefrac{1}{5k}}(\rho^\eta)}_{L^2([0,T];H^1(\TT^d))} \leq  c\Big(\E\Big[\int_{\TT^d}\rho_0^2+\int_0^T\int_{\TT^d}\abs{g}^2\Big]+\ve(F_1+\norm{F_3}T)\Big).\]
The tightness of the $\{\Theta_{\nicefrac{1}{5k}}(\rho^\eta)\}_{\eta\in(0,\nicefrac{1}{4})}$ is then a consequence of Proposition~\ref{prop_Theta_ap}, the compact embedding of $H^1(\TT^d)$ into $L^2(\TT^d)$, the continuous embedding of $L^2(\TT^d)$ into $H^{-s}(\TT^d)$ for every $s>\frac{d+2}{2}$, and the Aubin--Lions--Simon lemma \cite{Aubin,pLions,Simon}.

We will now deduce the tightness of the $\{\rho^\eta\}_{\eta\in(0,1)}$.  Let $n\in\N$ be arbitrary and, using the tightness of the renormalized functions, for every $k\in\N$ let $C_k$ be a compact subset of $L^2(\TT^d\times[0,T])$ with respect to the usual norm-induced metric satisfying for every $\eta\in(0,\nicefrac{1}{4})$ that $\P[\Theta_{\nicefrac{1}{5k}}(\rho^\eta)\notin C_k] \leq \frac{1}{2^kn}$.  For every $k\in\N$ let $F_k\colon L^2(\TT^d\times[0,T])\rightarrow L^2(\TT^d\times[0,T])$ be defined by $F_k(f) = \Theta_{\nicefrac{1}{5k}}(f)$.  Since it follows from the Lipschitz continuity of $\Theta_{\nicefrac{1}{5K}}$ that $F_k$ is continuous with respect to the usual norm-induced metric, let $D_k$ be the closed subset $D_k = F^{-1}_k(C_k)$ for every $k\in\N$ and let $D = \cap_{k=1}^\infty D_k$.  It follows from Proposition~\ref{def_metric} that $D$ is a compact subset of $L^2(\TT^d\times[0,T])$, and it follows that, for every $\eta\in(0,1)$,
\[\P[\rho^\eta\notin D]\leq \sum_{k=1}^\infty \P[\Theta_{\nicefrac{1}{5k}}(\rho^\eta)\notin C_k]\leq \sum_{k=1}^\infty \frac{1}{2^k n}\leq n^{-1}.\]
Since $n\in\N$ was arbitrary, this completes the proof.  \end{proof}

\begin{thm}\label{thm_rks_exist}  Let $T\in(0,\infty)$, let $\xi^F$ and $\rho_0$ satisfy Assumption~\ref{def_noise}, let $\ve\in(0,1)$, and let $g\in L^2(\O\times[0,T];L^2(\TT^d))^d$ be an $\F_t$-predictable process.  Then there exists a stochastic kinetic solution of \eqref{ex_3} in the sense of Definition~\ref{def_sol}.  Furthermore, the solution satisfies the estimates of Proposition~\ref{prop_ap} and Theorem~\ref{prop_est_infty}.  \end{thm}

\begin{proof}  The proof is virtually identical to \cite[Theorem~5.29]{FG21}, where in this case the metric defined in Proposition~\ref{def_metric} plays the role of the metric defined in \cite[Definition~5.23]{FG21} and the entropy estimate of Proposition~\ref{prop_ap} yields the optimal regularity of the measure.  \end{proof}

\section{The central limit theorem}\label{sec_clt}

In this section, we will study the fluctuations of the equation
\begin{equation}\label{clt_0}\partial_t \rho^\ve=\Delta \rho^\ve-\sqrt{\ve}\nabla\cdot(\sqrt{\rho^\ve(1-\rho^\ve)}\circ \xi^K)\;\;\textrm{in}\;\;\TT^d\times(0,T)\;\;\textrm{with}\;\;\rho(\cdot,0)=\rho_0,\end{equation}
about the hydrodynamic limit
\begin{equation}\label{clt_00}\partial_t\overline{\rho}=\Delta\overline{\rho}\;\;\textrm{in}\;\;\TT^d\times(0,T)\;\;\textrm{with}\;\;\overline{\rho}(\cdot,0)=\rho_0.\end{equation}
Precisely, we will identify an $\ve\rightarrow 0$, $K(\ve)\rightarrow\infty$ scaling regime such that the random variables
\begin{equation}\label{clt_000}v^\ve = \ve^{-\frac{1}{2}}(\rho^\ve-\overline{\rho}),\end{equation}
converge in probability in $L^2([0,T];H^{-s}(\TT^d))$, for every $s>\frac{d}{2}$, to the generalized Ornstein--Uhlenbeck process
\begin{equation}\label{clt_0000}\partial_t v = \Delta v -\nabla\cdot (\sqrt{\overline{\rho}(1-\overline{\rho})}\xi)\;\;\textrm{in}\;\;\TT^d\times(0,\infty)\;\;\textrm{with}\;\;v=0\;\;\textrm{on}\;\;\TT^d\times\{0\},\end{equation}
for $\xi$ a $d$-dimensional space-time white noise.  An essential difficulty in proving a central limit theorem for the solutions of \eqref{clt_0} is that the equation is only satisfied in the renormalized sense of Definition~\ref{def_sol}, which involves studying the equation satisfied by a nonlinear function of the solution.  The nonlinearity is incompatible with the convergence of the fluctuations in the space $H^{-s}$.

It is for this reason that we first establish a strong CLT for equation \eqref{clt_0} with the square root replaced by a smooth noise coefficient $\sigma$,
\begin{equation}\label{clt_1000}\partial_t \rho^\ve=\Delta \rho^\ve-\sqrt{\ve}\nabla\cdot(\sigma(\rho^\ve)\circ \xi^K)\;\;\textrm{in}\;\;\TT^d\times(0,T)\;\;\textrm{with}\;\;\rho(\cdot,0)=\rho_0,\end{equation}
satisfying $\sigma(0)=\sigma(1)=0$ and $\sigma\in\C^2_c((0,1)))\cap \C([0,1];[0,1])$.  We then extend this CLT to equation \eqref{clt_0} for initial data $\rho_0$ satisfying $\d\leq \rho_0\leq 1-\d$ for some $\delta\in(0,\nicefrac{1}{2})$ using the $L^\infty$-estimate of Theorem~\ref{prop_est_infty} below, after approximating the square root by a smooth $\sigma$ that agrees with the square root on $[\nicefrac{\d}{2},1-\nicefrac{\d}{2}]$ and proving using the pathwise uniqueness proof of Theorem~\ref{thm_rks_unique} that the solutions of \eqref{clt_0} and \eqref{clt_1000} agree for this choice of $\sigma$ on the event that both solutions remain outside the $\nicefrac{\d}{2}$-neighborhood of zero and one.

\subsection{A quantitative LLN for the approximating SPDE}  In this section, we will establish a qualitative law of large numbers for the solutions of \eqref{clt_0} in Theorem~\ref{thm_LLN_classical}, and a quantitative law of large numbers for the solutions of the approximating SPDE \eqref{clt_1000} in Proposition~\ref{prop_quant} that depends on the regularity of $\sigma$.  The well-posedness of \eqref{clt_1000} is explained in Definition~\ref{def_classical} and Proposition~\ref{prop_classical_wp}.  Lastly, in Assumption~\ref{def_noise_trig} we introduce a particular choice of noise $\{\xi^K\}_{K\in\N}$.  We emphasize that our methods do not rely on this choice, and would yield a quantitative rate of convergence for any sequence satisfying Assumption~\ref{def_noise} with $N_K$ replaced by $F_1$ and $M_K$ by $\norm{F}_3$.

\begin{assumption}\label{def_noise_trig}  Let $(\O,\F,\P)$ be a probability space, let $(\F_t)_{t\in[0,\infty)}$ be a filtration on $(\O,\F)$, let $(B^k,W^k)_{k\in\Z^d}$ be independent, $d$-dimensional, $\F_t$-adapted Brownian motions taking values in the space $\C^\infty([0,\infty);(\R^d)^\infty)$ equipped with the metric topology of coordinate-wise convergence, and for every $k\in\Z^d$ let $e_k=\sqrt{2}\sin(k\cdot x)$ and $e'_k = \sqrt{2}\cos(k\cdot x)$.  For every $K\in\N$ let $\xi^K$ be defined by
\[\xi^K(x,t) = \sum_{\abs{k}\leq K} \Big(e_k(x)B^k_t+e'_k(x)W^k_t\Big).\]
We observe that this noise is a special case of Assumption~\ref{def_noise} where we have that
\[ F_1 = N_K = \#\{k\in\Z^d\colon \abs{k}\leq K\}\;\; \textrm{and}\;\;F_3 = M_K = \sum_{\abs{k}\leq K} \abs{k}^2,\]
and that $N_K\leq cK^d$ and $M_K\leq cK^{d+2}$ for some $c\in(0,\infty)$ independent of $K\in\N$.  Finally, let $\rho_0\in L^2(\O;L^2(\TT^d;[0,1]))$ be $\F_0$-measurable.  \end{assumption}

\begin{thm}\label{thm_LLN_classical} Let $T\in(0,\infty)$, let $\{\xi^K\}_{K\in\N}$ satisfy Assumption~\ref{def_noise_trig}, let $\rho_0\in L^2(\TT^d;[0,1])$, and let $\{K(\ve)\}_{\ve\in(0,1)}$ be a sequence that satisfies $\ve K(\ve)^{d+2}\rightarrow 0$ as $\ve\rightarrow 0$.  Then, for the solutions $\rho^\ve$ of \eqref{clt_0} in the sense of Definition~\ref{def_sol}, $\P$-a.s.\ as $\ve\rightarrow 0$,
\[\rho^\ve\rightarrow\overline{\rho}\;\;\textrm{strongly in}\;\;L^2(\TT^d\times[0,T]),\]
for $\overline{\rho}$ the unique solution of \eqref{clt_00} with initial data $\rho_0$.  \end{thm}

\begin{proof}  The proof is a simplified version of Proposition~\ref{prop_entropy_1} with $g=0$ and $\rho_0$ deterministic.  Since we will not require this result in what follows, we postpone the details until the proof of Proposition~\ref{prop_entropy_1} below.  \end{proof}

\begin{definition}\label{def_classical} Let $T\in(0,\infty)$, let $\sigma\in \C^2_c((0,1))\cap \C([0,1];[0,1])$ with $\sigma(0)=\sigma(1)=0$, let $\{\xi^K\}_{K\in\N}$ satisfy Assumption~\ref{def_noise_trig}, let $\ve\in(0,1)$, let $K\in\N$, and let $\rho_0\in L^2(\TT^d;[0,1])$.  A \emph{weak solution} of \eqref{clt_1000} is a continuous $L^2(\TT^d;[0,1])$-valued, $\F_t$-adapted process $\rho\in L^2(\O\times[0,T];L^2(\TT^d))$ that satisfies the following two properties.
\begin{enumerate}
\item  \emph{Regularity}:  we have that
\[\rho\in L^2(\O\times[0,T];H^1(\TT^d)).\]
\item \emph{The equation}:  we have $\P$-a.s.\ that, for every $\psi\in\C^\infty(\TT^d)$ and $t\in[0,T]$,
\[\int_{\TT^d}\rho(x,t)\psi(x) = \int_{\TT^d}\rho_0(x)\psi(x)+\sqrt{\ve}\int_{\TT^d}\sigma(\rho)\nabla\psi\cdot\dd\xi^F-\frac{\ve N_K}{2}\int_{\TT^d}(\sigma'(\rho))^2\nabla\rho\cdot \nabla\psi.\]
\end{enumerate}
\end{definition}

\begin{prop}\label{prop_classical_wp}  Let $T\in(0,\infty)$, let $\sigma\in \C^2_c((0,1))\cap \C([0,1];[0,1])$ with $\sigma(0)=\sigma(1)=0$, let $\{\xi^K\}_{K\in\N}$ satisfy Assumption~\ref{def_noise_trig}, let $\ve\in(0,1)$, let $K\in\N$, and let $\rho_0\in L^2(\TT^d;[0,1])$.  Then there exists a unique solution of \eqref{clt_1000} in the sense of Definition~\ref{def_classical} that satisfies the estimates of Proposition~\ref{prop_ap}.  Furthermore, the solutions defined by Definition~\ref{def_sol} and Definition~\ref{def_classical} coincide.\end{prop}

\begin{proof}  The existence is a consequence of Proposition~\ref{prop_weak_exist}, and the uniqueness follows from a simplified version of Theorem~\ref{thm_rks_unique}.  \end{proof}

\begin{prop}\label{prop_quant}  Let $T\in(0,\infty)$, let $\sigma\in \C^2_c((0,1))\cap \C([0,1];[0,1])$ with $\sigma(0)=\sigma(1)=0$, let $\{\xi^K\}_{K\in\N}$ satisfy Assumption~\ref{def_noise_trig}, let $\ve\in(0,1)$, let $K\in\N$, and let $\rho_0\in L^2(\TT^d;[0,1])$.  Then, for the solution $\rho$ of \eqref{clt_1000} in the sense of Definition~\ref{def_classical} and the solution $\overline{\rho}$ of \eqref{clt_00},
\begin{align*}
& \E\Big[\sup\nolimits_{t\in[0,T]}\norm{\rho-\overline{\rho}}^2_{L^2(\TT^d)}+\int_0^T\int_{\TT^d}\abs{\nabla(\rho-\overline{\rho})}^2\Big]
\\ & \leq c\Big(\ve^2 N^2_K\norm{\sigma'}^2_{L^\infty((0,1))}\int_0^T\int_{\TT^d}\abs{\nabla\sigma(\rho)}^2+ \ve(N_K+ M_KT)\Big).
\end{align*}
\end{prop}

\begin{proof}  Let $w = \rho -\overline{\rho}$ and observe that $w$ is a solution of the equation
\[\partial_t w = \Delta w - \sqrt{\ve}\nabla\cdot(\sigma(\rho)\xi^K)-\frac{\ve N_K}{2}\nabla\cdot (\sigma'(\rho)\nabla\sigma(\rho))\;\;\textrm{in}\;\;\TT^d\times(0,T)\;\;\textrm{with}\;\;w=0\;\;\textrm{on}\;\;\TT^d\times\{0\}.\]
The claim then follows from a repetition of the arguments leading to the $L^2$-estimate of Proposition~\ref{prop_ap} using the boundedness of $\sigma$, H\"older's inequality, and Young's inequality.  \end{proof}

\subsection{The CLT for the approximating SPDE}  We will first establish a strong CLT for the solutions of \eqref{clt_1000} defined by a smooth and bounded $\sigma$.  We will prove that the fluctuations converge strongly in $L^2([0,T];H^{-s}(\TT^d))$, for every $s>\frac{d}{2}$, to the Ornstein--Uhlenbeck process
\begin{equation}\label{OU_approx}\partial_t v = \Delta v - \nabla\cdot(\sigma(\overline{\rho})\xi)\;\;\textrm{in}\;\;\TT^d\times(0,T)\;\;\textrm{with}\;\;v = 0 \;\;\textrm{on}\;\;\TT^d\times\{0\},\end{equation}
for $\xi$ a $d$-dimensional space-time white noise.  We explain the well-posedness of \eqref{OU_approx} in Definition~\ref{def_ou} and Proposition~\ref{ou_wp}.  Observe that Definition~\ref{def_ou} and Proposition~\ref{ou_wp} do not require any smoothness of $\sigma$, and therefore they apply to the square root $\sigma(\xi) = \sqrt{\xi(1-\xi)}$.  We then prove a quantitative CLT for the solutions of \eqref{clt_1000} in Theorem~\ref{thm_clt} that does depend on the smoothness of $\sigma$.

\begin{definition}\label{def_ou}  Let $T\in(0,\infty)$, let $\xi$ be an $\R^d$-valued space-time white noise, let $\sigma\in \C([0,1];[0,1])$ satisfy $\sigma(0)=\sigma(1)=0$, and let $\rho_0\in L^2(\TT^d)$.  Let $\overline{\rho}$ be the  weak solution of \eqref{clt_000} with initial data $\rho_0$.  A strong solution of \eqref{clt_0000} is an $\F_t$-adapted and almost surely continuous $H^{-s}(\TT^d)$-valued process $v\in L^2([0,T]\times\O;H^{-s}(\TT^d))$, for every $s>\frac{d}{2}$, that almost surely satisfies, for every $\psi\in\C^\infty(\TT^d)$ and $t\in[0,T]$,
\[\langle v(t),\psi\rangle_s=\int_0^t\langle v(r),\Delta\psi\rangle_s\dr+\int_0^t\int_{\TT^d}s(\overline{\rho})\nabla\psi\cdot\xi,\]
where $\langle\cdot,\cdot\rangle_s\colon H^{-s}(\TT^d)\times H^{s}(\TT^d)\rightarrow\R$ is the pairing between $H^{-s}(\TT^d)$ and $H^s(\TT^d)$.
\end{definition}

\begin{prop}\label{ou_wp}Let $T\in(0,\infty)$, let $\sigma\in \C([0,1];[0,1])$ satisfy $\sigma(0)=\sigma(1)=0$, let $\xi$ be an $\R^d$-valued space-time white noise, and let $\rho_0\in L^2(\TT^d;[0,1])$.  Then there exists a unique solution of \eqref{clt_0000} in the sense of Definition~\ref{def_ou}.  \end{prop}

\begin{proof}  Let the noise $\{\xi^K\}_{K\in\N}$ be defined by Assumption~\ref{def_noise_trig} and let $\overline{\rho}$ be the unique solution of \eqref{clt_00} with initial data $\rho_0$.  Simplified versions of Theorem~\ref{thm_rks_unique} and Proposition~\ref{prop_weak_exist} (or, Theorem~\ref{thm_rks_exist}) prove that, for every $K\in\N$, there exists a unique continuous $L^2(\TT^d)$-valued, $\F_t$-adapted process $v_K\in L^2([0,T]\times\O;L^2(\TT^d))$ that satisfies the SPDE with additive noise
\[\partial_t v_K = \Delta v_K - \nabla\cdot(\sigma(\overline{\rho})\xi^K)\;\;\textrm{in}\;\;\TT^d\times(0,T)\;\;\textrm{with}\;\;v_K(\cdot,0)=0.\]
Since it follows that $\int_{\TT^d} v_K(x,t) = 0$ for every $t\in[0,T]$, let $s>\frac{d+2}{2}$, let $z_K = (-\Delta)^{-\frac{s}{2}} v_K$, and observe using the methods of Proposition~\ref{prop_ap} that
\begin{align}\label{wpwn_1}
& \E\Big[\max_{t\in[0,T]}\norm{z_K(\cdot,t)}^2_{L^2(\TT^d)}+\int_0^T\int_{\TT^d}\abs{\nabla z_K}^2\Big] \leq \E\Big[\max_{t\in[0,T]}\abs{\int_0^t\int_{\TT^d}\sigma(\overline{\rho})((-\Delta)^{-\frac{s}{2}}\nabla z_K)\cdot \dd \xi^K}\Big]
\\ \nonumber & \quad  +\frac{1}{2}\E\Big[\sum_{\abs{k}\leq K}\sum_{i=1}^d\int_0^T\int_{\TT^d}\Big(\abs{(-\Delta)^{-\frac{s}{2}}\partial_i(\sigma(\overline{\rho})e_k)}^2+\abs{(-\Delta)^{-\frac{s}{2}}\partial_i(\sigma(\overline{\rho})e'_k)}^2\Big)\Big].
\end{align}
For the first term on the righthand side of \eqref{wpwn_1}, the Burkh\"older--Davis--Gundy inequality, H\"older's inequality, Young's inequality, $s> \frac{d+2}{2}\geq \frac{3}{2}$, the orthonormality of the $L^2(\TT^d)$-basis $\{e_k,e'_k\}_{k\in\Z^d}$, and the boundedness of $\sigma$ prove that, for some $c\in(0,\infty)$ independent of $K$,
\begin{align*}
& \E\Big[\max_{t\in[0,T]}\abs{\int_0^t\int_{\TT^d}\sigma(\overline{\rho})((-\Delta)^{-\frac{s}{2}}\nabla z_K)\cdot \dd \xi^K}\Big]
\\ & \leq c\E\Big[\Big(\int_0^T\sum_{\abs{k}\leq K}\Big(\int_{\TT^d}\sigma(\overline{\rho})((-\Delta)^{-\frac{s}{2}}\nabla z_K)e_k\dx\Big)^2+\Big(\int_{\TT^d}\sigma(\overline{\rho})((-\Delta)^{-\frac{s}{2}}\nabla z_K)e'_k\dx\Big)^2\dt\Big)^\frac{1}{2}\Big]
\\ & \leq c\E\norm{\sigma(\overline{\rho})((-\Delta)^{-\frac{s}{2}}\nabla z_K)}_{L^2(\TT^d\times[0,T])^d}\leq c\E\norm{z_K}_{L^2(\TT^d\times[0,T])}.
\end{align*}
For the final term on the righthand side of \eqref{wpwn_1}, it follows from $s>\frac{d+2}{2}\geq \frac{3}{2}$ and the orthonormality of the $L^2(\TT^d)$-basis $\{e_k,e'_k\}_{k\in\Z^d}$ that, for some $c\in(0,\infty)$ independent of $K$,
\[\sum_{\abs{k}\leq K}\sum_{i=1}^d\int_0^T\int_{\TT^d}\Big(\abs{(-\Delta)^{-\frac{s}{2}}\partial_i(\sigma(\overline{\rho})e_k)}^2+\abs{(-\Delta)^{-\frac{s}{2}}\partial_i(\sigma(\overline{\rho})e'_k)}^2\Big)\leq c\norm{\sigma(\overline{\rho})}_{L^2(\TT^d\times[0,T])}^2.\]
Returning to \eqref{wpwn_1}, it follows that, for some $c\in(0,\infty)$ independent of $K$,
\[\E\Big[\max_{t\in[0,T]}\norm{z_K(\cdot,t)}^2_{L^2(\TT^d)}+\int_0^T\int_{\TT^d}\abs{\nabla z_K}^2\Big] \leq c\E \Big[\norm{z_K}_{L^2(\TT^d\times[0,T])}+\norm{\sigma(\overline{\rho})}_{L^2(\TT^d\times[0,T])}^2\Big],\]
from which it follows from Young's inequality, Gr\"onwall's inequality, and the definition of $z_K$ that there exists $c\in(0,\infty)$ independent of $K$ but depending on $T$ such that
\begin{equation}\label{wpwn_2}  \E\Big[\norm{v_K}_{L^\infty([0,T];H^{-s}(\TT^d))}^2+\norm{v_K}^2_{L^2([0,T];H^{-s+1}(\TT^d))}\Big]\leq c(1+\norm{\sigma(\overline{\rho})}_{L^2(\TT^d\times[0,T])}^2).\end{equation}
Lastly, to estimate the time-regularity, we observe distributionally that
\begin{equation}\label{wpwn_3}v_K(\cdot,t) = \int_0^t\Delta v_K-\int_0^t\nabla\cdot(\sigma(\overline{\rho})\dd\xi^K)= I^{\textrm{f.v.}}_t(\cdot)+I^{\textrm{mart.}}_t(\cdot),\end{equation}
where it follows from \eqref{wpwn_2} that the finite-variation part satisfies, for some $c\in(0,\infty)$ independent of $K$,
\begin{equation}\label{wpwn_4}\norm{I^{\textrm{f.v.}}_{\cdot}(\cdot)}_{W^{1,2}([0,T];H^{-(s+1)}(\TT^d))}\leq c\norm{v_K}_{L^2([0,T];H^{-s}(\TT^d))},\end{equation}
and, following the methods of Proposition~\ref{prop_ap}, it follows from $0\leq \sigma\leq 1$ and the Burkh\"older--Davis--Gundy inequality that for every $\beta\in(0,\nicefrac{1}{2})$ there exists $c\in(0,\infty)$ depending on $\beta$ and $T$ but independent of $K$ such that
\begin{align}\label{wpwn_5}
& \E \norm{I^{\textrm{mart.}}_{\cdot}(\cdot)}_{W^{\beta,2}([0,T];H^{-(s+1)}(\TT^d))}^2
\\ \nonumber & = \E \int_0^T\int_0^T\abs{s-t}^{-(1+2\beta)} \lVert \sum_{\abs{k}\leq K}\int_s^t\int_{\TT^d}\Big(\sigma(\overline{\rho})e_k \dd B^k_t+\sigma(\overline{\rho})e'_k \dd W^k_t\Big)\rVert_{H^{-s}(\TT^d)}^2
\\ \nonumber & \leq c\E\int_0^T\int_0^T\abs{s-t}^{-(1+2\beta)}\sum_{\abs{k}\leq K} \int_s^t\norm{\sigma(\overline{\rho})e_k}^2_{H^{-s}(\TT^d)}+\norm{\sigma(\overline{\rho})e'_k}^2_{H^{-s}(\TT^d)}
\\ \nonumber & \leq c\int_0^T\int_0^T\abs{s-t}^{-2\beta} \leq c.
\end{align}
Returning to \eqref{wpwn_3}, the embedding of $W^{1,2}$ into $W^{\beta,2}$ for $\beta\in(0,1)$, \eqref{wpwn_4}, and \eqref{wpwn_5} prove that, for every $\beta\in(0,\nicefrac{1}{2})$ there exists $c\in(0,\infty)$ depending on $\beta$ and $T$ but independent of $K$ such that
\[\E\Big[\norm{v_K}_{L^2([0,T];H^{-(s+1)}(\TT^d))}\Big]\leq c(1+\norm{v_K}_{L^2([0,T];H^{-(s+1)}(\TT^d))}).\]
It now follows from the compact embedding of $H^{-s}$ into $H^{-s'}$ whenever $s<s'\in(0,\infty)$, the Aubin--Lions--Simon lemma \cite{Aubin,pLions,Simon}, and specifically \cite[Corollary~5]{Simon}, and estimates \eqref{wpwn_2} and \eqref{wpwn_4} that the laws of the $\{v_K\}_{K\in\N}$ are tight on $L^2([0,T];H^{-s}(\TT^d))$ for every $s>\frac{d}{2}$.  A simplified version of Theorem~\ref{thm_rks_exist} proves that, after passing to the limit $K\rightarrow 0$, the $\{v_K\}_{K\in\N}$ converge in law on $L^2([0,T];H^{-s}(\TT^d))$ to an element $v\in L^2([0,T]\times\O;H^{-s}(\TT^d))$, for every $s>\frac{d}{2}$, satisfying, for every $\psi\in\C^\infty(\TT^d)$ and $\d\in(0,1)$, for almost every $t\in[0,T]$,
\begin{equation}\label{wpwn_7}\langle v(t), \psi \rangle_s  = \int_0^t\langle v(r),\Delta\psi\rangle_s\dr +\int_0^t\int_{\TT^d}\sigma(\overline{\rho})\nabla\psi\cdot\dd\xi.\end{equation}
Since both terms on the righthand side of \eqref{wpwn_7} are continuous in time, this implies that, for every $k\in\Z^d$, the pairings $t\mapsto \langle v(t),e_k\rangle_s$ and $t\mapsto \langle v(t),e'_k\rangle_s$ admit continuous modifications in $L^2([0,T])$.   The $\P$-a.s.\ boundedness of $v$ in $L^2([0,T];H^{-s}(\TT^d))$, for every $s>\frac{d}{2}$, then implies that $v$ admits a continuous, $H^{-s}$-valued modification still denoted $v$ and that $v$ satisfies \eqref{wpwn_7} for every $\psi\in \C^\infty(\TT^d)$ and $t\in[0,T]$.  This completes the proof of existence.  In order to prove uniqueness, observe that if $v$ and $\tilde{v}$ are two solutions then $w = v-\tilde{v}$ is a distributional solution of the heat equation that $\P$-a.s.\ satisfies, for every $k\in \Z^d$, for every $t\in[0,T]$,
\[\langle w(t), e_k\rangle_s = -\abs{k}^2\int_0^t\langle w(r),e_k\rangle_s \dr\;\;\textrm{and}\;\;\langle w(t), e'_k\rangle_s = -\abs{k}^2\int_0^t\langle w(r),e'_k\rangle_s \dr.\]
Gr\"onwall's inequality proves that $\langle w(t), e_k\rangle_s=\langle w(t), e'_k\rangle_s=0$ for every $k\in\Z^d$ and $t\in[0,T]$.  Since $\{e_k,e'_k\}_{k\in\Z^d}$ is an $L^2(\TT^d)$-basis, this implies $\P$-a.s.\ that $w=0$ in $L^2([0,T];H^{-s}(\TT^d))$, for every $s>\frac{d}{2}$, and completes the proof of uniqueness.  \end{proof}

\begin{thm}\label{thm_clt}  Let $T\in(0,\infty)$, let $\{\xi^K\}_{K\in \N}$ satisfy Assumption~\ref{def_noise_trig}, let $\sigma\in \C([0,1];[0,1])\cap\C^2_c((0,1)))$ satisfy $\sigma(0)=\sigma(1)=0$, let $K\in\N$, let $\ve\in(0,1)$, and let $\rho_0\in L^2(\TT^d;[0,1])$.  Let $\rho^\ve$ be the solution of \eqref{clt_0} in the sense of Definition~\ref{def_classical}, let $\overline{\rho}$ be the solution of \eqref{clt_00}, let $v^\ve = \ve^{-\nicefrac{1}{2}}(\rho^\ve-\overline{\rho})$, and let $v$ be the solution of \eqref{clt_0000} in the sense of Definition~\ref{def_ou}.  Then, for every $s>\frac{d}{2}$ there exists $c\in(0,\infty)$ independent of $\ve$ and $K$ such that
\begin{align*}
& \E\Big[\norm{v^\ve-v}^2_{L^2([0,T];H^{-s}(\TT^d))}\Big]
\\ & \leq c\Big(\Big(K^{-1}+\ve N^2_{K}\norm{\sigma'}^2_{L^\infty((0,1))}\Big)\norm{\nabla\sigma(\overline{\rho})}^2_{L^2(\TT^d\times[0,T])}+K^{(-2s+(d+2))}\norm{\sigma(\overline{\rho})}^2_{L^2(\TT^d\times[0,T])}\Big)
\\ & \quad +c\norm{\sigma(\rho^\ve)-\sigma(\overline{\rho})}^2_{L^2(\TT^d\times[0,T])}.
\end{align*}
\end{thm}
\begin{proof}  Let $w^\ve = v^\ve -v$, let $s> \frac{d+2}{2}$, and let $z^\ve = (-\Delta)^{-\frac{s}{2}}w^\ve$.  We observe that since $s>\frac{d+2}{2}$ we have $\P$-a.s. that $z^\ve\in L^2([0,T];H^1(\TT^d))$ and that $z^\ve$ solves
\[\partial_t z^\ve = \Delta z^\ve - (-\Delta)^{-\frac{s}{2}}\nabla\cdot(\sigma(\rho^\ve)\xi^{K})+(-\Delta)^{-\frac{s}{2}}\nabla\cdot(\sigma(\overline{\rho})\xi)+\frac{\sqrt{\ve}N_{K}}{2}(-\Delta)^{-\frac{s}{2}}\nabla\cdot((\sigma'(\rho^\ve))^2\nabla\rho^\ve).\]
A repetition of the $L^2$-energy estimate of Proposition~\ref{prop_ap} and the definition of $\xi^{K}$ prove that, for $\Theta\colon[0,1]\rightarrow\R$ satisfying $\Theta(0)=0$ and $\Theta'(\xi) = (\sigma'(\xi))^2$,
\begin{align}\label{dist_000}
&  \E\Big[\int_0^T\int_{\TT^d}\abs{\nabla z^\ve}^2\Big]  \leq - \frac{\sqrt{\ve} N_{K}}{2}\E\Big[\int_0^T\int_{\TT^d}((-\Delta)^{-\frac{s}{2}}\nabla \Theta(\rho))\cdot\nabla z^\ve\Big]
\\ \nonumber & \quad +\sum_{\abs{k}\leq K}\int_0^T\int_{\TT^d}((-\Delta)^{-\frac{s}{2}}\nabla\cdot((\sigma(\rho^\ve)-\sigma(\overline{\rho}))e_k))^2+((-\Delta)^{-\frac{s}{2}}\nabla\cdot((\sigma(\rho^\ve)-\sigma(\overline{\rho}))e'_k))^2
\\ \nonumber & \quad +\sum_{\abs{k}>K}\int_0^T\int_{\TT^d}((-\Delta)^{-\frac{s}{2}}\nabla\cdot(\sigma(\overline{\rho})e_k))^2+((-\Delta)^{-\frac{s}{2}}\nabla\cdot(\sigma(\overline{\rho})e'_k))^2.
\end{align}
H\"older's inequality and Young's inequality prove that the first term on the righthand side of \eqref{dist_000} satisfies, for some $c\in(0,\infty)$,
\begin{align}\label{wpwn_10}
& \frac{\sqrt{\ve} N_{K}}{2}\E\Big[\int_0^T\int_{\TT^d}((-\Delta)^{-\frac{s}{2}}\nabla \Theta(\rho))\cdot\nabla z^\ve\Big]
\\ \nonumber &  \leq c\ve N^2_{K}\norm{\sigma'}^2_{L^\infty((0,1))}\E\Big[\int_0^T\int_{\TT^d}\abs{\nabla\sigma(\rho^\ve)}^2\Big]+\frac{1}{4}\E\Big[\int_0^T\int_{\TT^d}\abs{\nabla z^\ve}^2\Big].
\end{align}
The orthonormality of the $L^2$-basis $\{e_k,e'_,\}_{k\in\Z^d}$ and $s>\frac{d+2}{2}>\frac{3}{2}$ prove that the final two terms on the righthand side of \eqref{dist_000} satisfy, for some $c\in(0,\infty)$,
\begin{align}\label{wpwn_11}
& \sum_{\abs{k}\leq K}\int_0^T\int_{\TT^d}((-\Delta)^{-\frac{s}{2}}\nabla\cdot((\sigma(\rho^\ve)-\sigma(\overline{\rho}))e_k))^2+((-\Delta)^{-\frac{s}{2}}\nabla\cdot((\sigma(\rho^\ve)-\sigma(\overline{\rho}))e'_k))^2
\\ \nonumber & \leq c\norm{\sigma(\rho^\ve)-\sigma(\overline{\rho})}^2_{L^2(\TT^d\times[0,T])},
\end{align}
and, a calculation using the definition of the $H^{-s}$-norm proves that, for some $c\in(0,\infty)$,
\begin{align}\label{fgd_10}
& \sum_{\abs{k}>K}\int_0^T\int_{\TT^d}((-\Delta)^{-\frac{s}{2}}\nabla\cdot(\sigma(\overline{\rho})e_k))^2+((-\Delta)^{-\frac{s}{2}}\nabla\cdot(\sigma(\overline{\rho})e'_k))^2
\\ \nonumber & \leq c\sum_{\abs{k}>K}\Big(\norm{\sigma(\overline{\rho})e_k}^2_{H^{-s+1}(\TT^d\times[0,T])}+\norm{\sigma(\overline{\rho})e'_k}^2_{H^{-s+1}(\TT^d\times[0,T])}\Big)
\\ \nonumber & \leq c K^{-1}\norm{\nabla\sigma(\overline{\rho})}^2_{L^2(\TT^d\times[0,T])}+cK^{(-2s+(d+2))}\norm{\sigma(\overline{\rho})}^2_{L^2(\TT^d\times[0,T])}.
\end{align}
Returning to \eqref{dist_000}, the claim now follows from estimates \eqref{wpwn_10}, \eqref{wpwn_11}, and \eqref{fgd_10}.  \end{proof}

\subsection{The CLT for the SPDE with singular coefficients}  We will first establish an $L^\infty$-estimate for the solutions of \eqref{clt_0} with singular coefficients, and then use this estimate in Theorem~\ref{thm_clt_prob} to establish the quantitative CLT in probability.

\begin{thm}\label{prop_est_infty}  Let $T\in(0,\infty)$, let $\ve\in(0,1)$, let $\xi^F$ satisfy Assumption~\ref{def_noise}, let $\rho_0\in L^2(\TT^d;[0,1])$, let $M=\esssup_{x\in\TT^d}\rho_0(x)$, and let $M'=\essinf_{x\in\TT^d}\rho_0(x)$.  Then, for the solution $\rho^\ve$ of \eqref{clt_0} in the sense of Definition~\ref{def_sol}, there exist $c,\gamma\in(0,\infty)$ independent of $\ve$ but depending on $T$ such that
\[\E\Big[\norm{(\rho^\ve-M)_+}_{L^\infty(\TT^d\times[0,T])}\Big]+\E\Big[\norm{(\rho^\ve-M')_-}_{L^\infty(\TT^d\times[0,T])}\Big]\leq c\Big(\ve\norm{F_3}\Big)^\gamma. \]
\end{thm}

\begin{proof}    Let $M = \sup_{x\in\TT^d}\rho_0(x)$ and let $\psi = (\rho-M)_+$.  A repetition of the proof of the energy estimates in Proposition~\ref{prop_ap} proves that, for $\E_{\F_0}[\cdot ] = \E[\cdot|\F_0]$, for every bounded stopping time $\tau$ taking values in $[0,T]$, we have $\P$-a.s.\ that, for $c\in(0,\infty)$ independent of $\ve$ and $\a$,
\begin{equation}\label{fgd_3} \E_{\F_0}\Big[\max_{t\in[0,\tau]}\int_{\TT^d}\psi^{\a+1}+\int_0^\tau\int_{\TT^d}\abs{\nabla\psi^\frac{\a+1}{2}}^2\Big] \leq  c\a^2\ve\norm{F_3}\E_{\F_0}\Big[\int_0^\tau\int_{\TT^d}\psi^{\a-1}\Big]. \end{equation}
It then follows from Revuz and Yor \cite[Chapter~4, Proposition~4.7, Exercise~4.30]{RevYor1999}, \eqref{fgd_3}, and H\"older's inequality that, for every $\a\in[1,\infty)$, for $n_\a = (\a+1)^{-1}$,
\begin{align}\label{in_1}
& \E\Big[\Big(\max_{t\in[0,\tau]}\int_{\TT^d}\psi^{\a+1}+\int_0^\tau\int_{\TT^d}\abs{\nabla\psi^{\frac{\a+1}{2}}}^2\Big)^\frac{1}{\a+1}\Big]\leq \frac{n_{\a}^{-n_\a}}{1-n_\a}(c\a^2\ve\norm{F_3})^{n_\a}\E\Big[\norm{\psi}_{L^{\a-1}(\TT^d\times[0,T])}\Big]^\frac{\a-1}{\a+1}.
\end{align}
We are now prepared to conclude using a Moser iteration.  We first use interpolation and the Sobolev inequality to deduce that for
\[\theta=\frac{d}{2+d}\;\;\textrm{and}\;\;q = \frac{(2+d)(\a+1)}{d},\]
we have that, for $c\in(0,\infty)$ independent of $\ve$ and $\a$ but depending on $T$,
\begin{align*}
& \norm{\psi}_{L^q(\TT^d\times[0,T])} \leq \norm{\psi}^\theta_{L^\infty([0,T];L^{\a+1}(\TT^d))}\norm{\psi}^{1-\theta}_{L^{\a+1}([0,T];L^{\frac{2_*}{2}(\a+1)}(\TT^d))}
\\ & = \norm{\psi}^\theta_{L^\infty([0,T];L^{\a+1}(\TT^d))}\norm{\psi^{\frac{\a+1}{2}}}^{\frac{2(1-\theta)}{\a+1}}_{L^2([0,T];L^{2_*}(\TT^d))}
\\ & \leq\norm{\psi}^\theta_{L^\infty([0,T];L^{\a+1}(\TT^d))}\Big(c\Big(\norm{\psi}_{L^\infty([0,T];L^{\a+1}(\TT^d))}^\frac{\a+1}{2}+\norm{\nabla \psi^{\frac{\a+1}{2}}}_{L^2([0,T];L^{2}(\TT^d))}\Big)\Big)^{\frac{2(1-\theta)}{\a+1}}.
\end{align*}
H\"older's inequality, the inequality $(x+y)^2\leq 2(x^2+y^2)$ for all $x,y\in[0,\infty)$, $\theta\in(0,1)$, $\a\in[1,\infty)$, and \eqref{in_1} prove that, for $c\in(0,\infty)$ independent of $\ve$ and $\a$,
\begin{align}\label{in_2}
& \E\Big[ \norm{\psi}_{L^q(\TT^d\times[0,T])}\Big]
\\ \nonumber  \leq & \E\Big[\norm{\psi}_{L^\infty([0,T];L^{\a+1}(\TT^d))}\Big]^\theta\E\Big[\Big(c\norm{\psi}^{\a+1}_{L^\infty([0,T];L^{\a+1}(\TT^d))}+c\norm{\nabla \psi^{\frac{\a+1}{2}}}^2_{L^2([0,T];L^{2}(\TT^d))}\Big)^\frac{1}{\a+1}\Big]^{1-\theta}
\\ \nonumber  \leq & \E\Big[\Big(c\norm{\psi}^{\a+1}_{L^\infty([0,T];L^{\a+1}(\TT^d))}+c\norm{\nabla \psi^{\frac{\a+1}{2}}}^2_{L^2([0,T];L^{2}(\TT^d))}\Big)^\frac{1}{\a+1}\Big]
\\ \nonumber  \leq & \frac{n_{\a}^{-n_\a}}{1-n_\a}(c\a^2\ve\norm{F_3})^{n_\a}\E\Big[\norm{\psi}_{L^{\a-1}(\TT^d\times[0,T])}\Big]^\frac{\a-1}{\a+1}.
\end{align}
We proceed inductively by defining
\[\alpha_0=0\;\;\textrm{and}\;\;\alpha_k = \frac{2+d}{d}\Big(\a_{k-1}+2\Big)\;\;\textrm{for every}\;\;k\in\N,\]
which implies that $\alpha_k> (1+\frac{2}{d})^k$ for every $k\in\N$, and we let $\beta_k = \alpha_{k-1}+1$.  Observe using \eqref{in_2} and the definition of $q$ that, for every $k\in\N$, for $c\in(0,\infty)$ independent of $\ve$, $\a$, and $k$,
\begin{align}\label{in_3}
& \E\Big[ \norm{\psi}_{L^{\alpha_k}(\TT^d\times[0,T])}\Big] \leq \frac{n_{\beta_k}^{-n_{\beta_k}}}{1-n_{\beta_k}}(c\beta_k^2\ve\norm{F_3})^{n_{\beta_k}}\E\Big[\norm{\psi}_{L^{\a_{k-1}}(\TT^d\times[0,T])}\Big]^\frac{\beta_k-1}{\beta_k+1}
\\ \nonumber & \leq \prod_{r=1}^k \Big(\frac{n_{\beta_r}^{-n_{\beta_r}}}{1-n_{\beta_r}}(c\beta_r)^{2 n_{\beta_r}}\Big)^{\prod_{s=r+1}^k\frac{\beta_s-1}{\beta_s+1}}\times\Big(\ve\norm{F_3}\Big)^{\sum_{r=1}^kn_{\beta_r}\prod_{s=r+1}^k\frac{\beta_s-1}{\beta_s+1}}.
\end{align}
Since by definition $\beta_k > (1+\frac{2}{d})^{k-1}+1$ for every $k\in[2,\infty)$, we have that
\[\liminf_{N\rightarrow\infty}\log\Big(\prod_{s=2}^N\frac{\beta_s-1}{\beta_s+1}\Big) = \liminf_{N\rightarrow\infty}\sum_{s=2}^N \log\Big(1-\frac{2}{\beta_s+1}\Big)>-\infty.\]
Therefore, by the dominated convergence theorem, there exists $\gamma\in(0,1)$ such that
\[\lim_{k\rightarrow\infty}\Big(\sum_{r=1}^kn_{\beta_r}\prod_{s=r+1}^k\frac{\beta_s-1}{\beta_s+1}\Big)=\gamma.\]
We have similarly that, by definition of $n_{\beta_r}$,
\[\limsup_{k\rightarrow\infty} \prod_{r=1}^k \Big(\frac{n_{\beta_r}^{-n_{\beta_r}}}{1-n_{\beta_r}}(c\beta_r)^{4 n_{\beta_r}}\Big)^{\prod_{s=r+1}^k\frac{\beta_s-1}{\beta_s+1}}\leq \limsup_{k\rightarrow\infty} \prod_{r=1}^k \Big(\frac{(1+\beta_r)^{1+n_{\beta_r}}}{\beta_r}(c\beta_r)^{4 n_{\beta_r}}\Big),\]
and that, for every $k\in\N$,
\[\log\Big(\prod_{r=1}^k \Big(\frac{(1+\beta_r)^{1+n_{\beta_r}}}{\beta_r}(c\beta_r)^{4 n_{\beta_r}}\Big)\Big) = \sum_{r=1}^k\Big(\log(1+\beta_r^{-1})+n_{\beta_r}\log(1+\beta_r)+4n_{\beta_r}\log(c\beta_r)\Big).\]
It follows from the facts that $\beta_r\geq \alpha_{r-1}+1$, $\alpha_0=0$, and $\alpha_{r}>(1+\frac{2}{d})^{r}$ that
\[\limsup_{k\rightarrow\infty}\Big(\sum_{r=1}^k\Big(\log(1+\beta_r^{-1})+n_{\beta_r}\log(1+\beta_r)+4n_{\beta_r}\log(c\beta_r)\Big)\Big)<\infty,\]
which proves that there exists $c\in(0,\infty)$ independent of $\ve$ and $\eta$ such that
\[\limsup_{k\rightarrow\infty} \prod_{r=1}^k \Big(\frac{n_{\beta_r}^{-n_{\beta_r}}}{1-n_{\beta_r}}(c\beta_r)^{4 n_{\beta_r}}\Big)^{\prod_{s=r+1}^k\frac{\beta_s-1}{\beta_s+1}}\leq c,\]
and, therefore, after passing to the limit in $k\rightarrow\infty$ in \eqref{in_3}, we conclude that
\[\E\Big[ \norm{\psi}_{L^\infty(\TT^d\times[0,T])}\Big] \leq c\Big(\ve\norm{F_3}\Big)^\gamma,\]
which completes the proof of the upper bound.  The lower bound is obtained by letting $M' = \inf_{x\in\TT^d}\rho_0(x)$ and considering $\psi(x,t) = (\rho(x,t)-M')_- = -\max(0, (M'-\rho))$.  \end{proof}

\begin{thm}\label{thm_clt_prob}  Let $T\in(0,\infty)$, let $\{\xi^K\}_{K\in\N}$ be the noise defined in Definition~\ref{def_noise_trig}, let $K\in\N$, let $\ve\in(0,1)$, and let $\rho_0\in L^2(\TT^d)$ satisfy $\d\leq\rho_0\leq 1-\d$ for some $\d\in(0,\nicefrac{1}{2})$.  Let $\rho^\ve$ be the solution of \eqref{clt_0} in the sense of Definition~\ref{def_sol}, let $\overline{\rho}$ be the solution of \eqref{clt_00}, and let $v^\ve = \ve^{-\nicefrac{1}{2}}(\rho^\ve-\overline{\rho})$.  Then, for every $s>\frac{d}{2}$ there exist $c,\gamma\in(0,\infty)$ such that, for every $\eta\in(0,1)$,
\begin{align*}
& \P\Big[\norm{v^\ve - v}_{L^2([0,T];H^{-s}(\TT^d))}\geq\eta\Big]
\\ & \leq c\eta^{-2}\Big(K^{-1}+\ve N^2_{K}\delta^{-1}+K^{(-2s+(d+2))}+\delta^{-2}\ve^2N_K^2+\delta^{-1}\ve M_K\Big)+c\delta^{-1}\Big(\ve M_K\Big)^\gamma.
\end{align*}
\end{thm}

\begin{proof}  Let $\sigma\in \C^2_c((0,1)))\cap\C([0,1];[0,1])$ be an arbitrary function satisfying that
\begin{equation}\label{fin_0}\sigma(\xi)=\sqrt{\xi(1-\xi)}\;\textrm{for every}\;\xi\in[\nicefrac{\d}{2},1-\nicefrac{\d}{2}]\;\textrm{and}\;\abs{\sigma'(\xi)}\leq 2\d^{-\frac{1}{2}}\;\textrm{for every}\;\xi\in(0,1),\end{equation}
and satisfying that
\begin{equation}\label{fin_1}\abs{\sigma'(\xi)}\leq \frac{2}{\sqrt{\xi(1-\xi)}}\;\;\textrm{for every}\;\;\xi\in(0,1).\end{equation}
For every $\ve\in(0,1)$ let $\tilde{v}^\ve = \ve^{-\nicefrac{1}{2}}(\tilde{\rho}^\ve-\overline{\rho})$ for the $\tilde{\rho}^\ve$ the solution of \eqref{clt_1000} defined by this $\sigma$ and let $v$ denote the Ornstein--Uhlenbeck process
\[\partial_t v = \Delta v-\nabla\cdot(\sqrt{\overline{\rho}(1-\overline{\rho})}\xi) = \Delta v-\nabla\cdot(\sigma(\overline{\rho})\xi)\;\;\textrm{in}\;\;\TT^d\times(0,T)\;\;\textrm{with}\;\;v = 0\;\;\textrm{on}\;\;\TT^d\times\{0\},\]
where here we use the fact that, since the comparison principle proves that $\d\leq \overline{\rho}\leq 1-\d$ on $\TT^d\times[0,T]$, we have by \eqref{fin_0} that $\sigma(\overline{\rho})=\sqrt{\overline{\rho}(1-\overline{\rho})}$.  We then observe through an exact repetition of the pathwise uniqueness proof of Theorem~\ref{thm_rks_unique} that
\[\rho^\ve = \tilde{\rho^\ve}\;\;\textrm{in}\;\;L^2(\TT^d\times[0,T]),\]
on the event $(\mathcal{S}\cap\tilde{\mathcal{S}})\subseteq\Omega$ defined by
\[\mathcal{S}=\{\norm{\rho^\ve}_{L^\infty(\TT^d\times[0,T])}\leq 1-\nicefrac{\d}{2}\;\textrm{and}\;\norm{1-\rho^\ve}_{L^\infty(\TT^d\times[0,T])}\leq 1-\nicefrac{\d}{2}\},\]
and
\[\tilde{\mathcal{S}}=\{\norm{\tilde{\rho}^\ve}_{L^\infty(\TT^d\times[0,T])}\leq 1-\nicefrac{\d}{2}\;\textrm{and}\;\norm{1-\rho^\ve}_{L^\infty(\TT^d\times[0,T])}\leq 1-\nicefrac{\d}{2}\}.\]
Therefore, for every $\eta\in(0,1)$, we have that
\begin{equation}\label{cltp_0}  \P\Big[\norm{v^\ve - v}_{L^2([0,T];H^{-s}(\TT^d))}\geq\eta\Big] \leq \P\Big[\norm{\tilde{v}^\ve - v}_{L^2([0,T];H^{-s}(\TT^d))}>\eta\Big]+\P[\mathcal{S}^c]+\P[\tilde{\mathcal{S}}^c]. \end{equation}
Since it follows from \eqref{fin_1}, $\rho_0\in L^2(\TT^d;[0,1])$, and the entropy estimate of Proposition~\ref{prop_ap} that there exists $c\in(0,\infty)$ independent of $\ve\in(0,1)$ such that
\[\E\Big[\int_0^T\int_{\TT^d}\abs{\nabla\sigma(\rho^\ve)}^2\Big]\leq c,\]
Theorem~\ref{thm_clt}, Chebyshev's inequality, \eqref{fin_0}, and $0\leq \sigma\leq 1$ prove that, for some $c\in(0,\infty)$ independent of $\ve$, $\eta$, and $K$,
\begin{align*}
&  \P\Big[\norm{\tilde{v}^\ve - v}_{L^2([0,T];H^{-s}(\TT^d))}>\eta\Big]
\\ & \leq c\eta^{-2}\Big(K^{-1}+\ve N^2_{K}\norm{\sigma'}^2_{L^\infty((0,1))}+K^{(-2s+(d+2))}+\norm{\sigma(\rho^\ve)-\sigma(\overline{\rho})}^2_{L^2(\TT^d\times[0,T])}\Big)
\\ & \leq c\eta^{-2}\Big(K^{-1}+\ve N^2_{K}\delta^{-1}+K^{(-2s+(d+2))}+\delta^{-1}\norm{\rho^\ve-\overline{\rho}}^2_{L^2(\TT^d\times[0,T])}\Big).
\end{align*}
The quantitative law of large numbers of Proposition~\ref{prop_quant}, \eqref{fin_0}, and $N_K\leq M_K$ then prove that, for $c\in(0,\infty)$ independent of $\ve$ but depending on $T$,
\[\E\big[\norm{\rho^\ve-\overline{\rho}}^2_{L^2(\TT^d\times[0,T])}\big] \leq c\Big(\delta^{-1}\ve^2 N^2_K+ \ve M_K\Big).\]
Therefore,
\begin{align}\label{cltp_2}
& \P\Big[\norm{\tilde{v}^\ve - v}_{L^2([0,T];H^{-s}(\TT^d))}>\eta\Big]
\\ \nonumber & \leq c\eta^{-2}\Big(K^{-1}+\ve N^2_{K}\delta^{-1}+K^{(-2s+(d+2))}+\delta^{-2}\ve^2N_K^2+\delta^{-1}\ve M_K\Big).
\end{align}
For the final two terms on the righthand side of \eqref{cltp_0}, Theorem~\ref{prop_est_infty}, which applies equally to the approximating equation thanks to \eqref{fin_0} and \eqref{fin_1}, and Chebyshev's inequality prove that, for some $c,\gamma\in(0,\infty)$ independent of $\ve$ and $K$,
\begin{equation}\label{cltp_3}\P[\mathcal{S}^c]+\P[\tilde{\mathcal{S}}^c]\leq c\delta^{-1}\Big(\ve M_K\Big)^\gamma.\end{equation}
Returning to \eqref{cltp_0}, it follows from \eqref{cltp_2} and \eqref{cltp_3} that
\begin{align*}
& \P\Big[\norm{v^\ve - v}_{L^2([0,T];H^{-s}(\TT^d))}\geq\eta\Big]
\\ & \leq c\eta^{-2}\Big(K^{-1}+\ve N^2_{K}\delta^{-1}+K^{(-2s+(d+2))}+\delta^{-2}\ve^2N_K^2+\delta^{-1}\ve M_K\Big)+c\delta^{-1}\Big(\ve M_K\Big)^\gamma,
\end{align*}
which completes the proof. \end{proof}

\begin{cor}\label{cor_clt_prob}  Let $T\in(0,\infty)$, let $\{\xi^K\}_{K\in\N}$ be the noise defined in Definition~\ref{def_noise_trig}, let $\alpha_d=(\frac{1}{2(d+2)}\wedge\frac{1}{2d})$ and let $K(\ve) = \lfloor\ve^{-\a_d}\rfloor$ for every $\ve\in(0,1)$, and let $\rho_0\in L^2(\TT^d)$ satisfy $\d\leq\rho_0\leq 1-\d$ for some $\d\in(0,\nicefrac{1}{2})$.  Let $\rho^\ve$ be the solution of \eqref{clt_0} corresponding to $(\ve,K(\ve))$ in the sense of Definition~\ref{def_sol}, let $\overline{\rho}$ be the solution of \eqref{clt_00}, and let $v^\ve = \ve^{-\nicefrac{1}{2}}(\rho^\ve-\overline{\rho})$.  Then, for every $s>\frac{d}{2}$ there exist $c,\gamma\in(0,\infty)$ such that, for every $\eta\in(0,1)$,
\[\P\Big[\norm{v^\ve - v}_{L^2([0,T];H^{-s}(\TT^d))}\geq\eta\Big] \leq c\eta^{-2}\delta^{-2}\Big(\ve^{\alpha_d}+\ve^{\alpha_d(2s-(d+2))}\Big)+c\delta^{-1}\ve^\frac{\gamma}{2}.\]
\end{cor}

\begin{proof}  The proof is an immediate consequence of Theorem~\ref{thm_clt_prob} and the bounds $N_K\leq cK^d$ and $M_K\leq cK^{d+2}$ for some $c\in(0,\infty)$ independent of $K\in\N$.\end{proof}

\section{The large deviations principle}\label{sec_old}

In this section, we will identify a scaling regime for which the solutions of the SPDE 
\begin{equation}\label{new_1}\partial_t \rho = \Delta\rho - \sqrt{\ve}\nabla\cdot (\sqrt{\rho(1-\rho)}\circ\xi^{K(\ve)}),\end{equation}
for the spectral approximations $\{\xi^K\}_{K\in\N}$ defined in Assumption~\ref{def_noise_trig}, satisfy a large deviations principle with rate functions $I_{\rho_0}$ defined by
\begin{equation}\label{ldp_1}I_{\rho_0}(\mu)=\frac{1}{2}\big\{\norm{g}^2_{L^2}\colon\mu=\rho\,dx,\,\partial_{t}\rho=\Delta\rho-\nabla\cdot(\sqrt{\rho(1-\rho)}g)\;\;\textrm{with}\;\;\rho(\cdot,0)=\rho_0\big\}.\end{equation}
We emphasize that the these techniques do not rely on the specific choice of noise $\{\xi^K\}_{K\in\N}$ and would apply without change, for example, to spatial convolutions $\{\xi^\d=(\xi*\kappa^\ve)\}_{\d\in(0,1)}$ of white noise, or to a general sequence satisfying Assumption~\ref{def_noise} converging in distribution to a space-time white noise.  We make this choice in order to precisely quantify the scaling in $\ve$ and $K$.  In what follows, we first analyze the skeleton equation appearing in the definition of the rate function in Section~\ref{sec_skeleton} and prove the LDP in Section~\ref{sec_ldp_old}.

\subsection{The skeleton equation}\label{sec_skeleton} In this section, we will prove the well-posedness of the skeleton equation
\begin{equation}\label{skel_1}\partial_t \rho = \Delta\rho-\nabla\cdot(\sqrt{\rho(1-\rho)}g)\;\;\textrm{in}\;\;\TT^d\times(0,T)\;\;\textrm{with}\;\;\rho=\rho_0\;\;\textrm{on}\;\;\TT^d\times\{0\},\end{equation}
for $L^2$-integrable controls $g$ and initial data $\rho_0$ in the space $L^2(\TT^d;[0,1])$.  In Definition~\ref{skel_def}, we define a weak solution to \eqref{skel_1}.  In Proposition~\ref{prop_unique}, we prove that weak solutions are unique.  In Proposition~\ref{prop_exist}, we prove that weak solutions exist.  Finally, in Proposition~\ref{prop_weak_strong}, we prove the strong continuity of solution $\rho$ with respect to weak convergence of the control $g$.

\begin{definition}\label{skel_def}  Let $T\in(0,\infty)$, $\rho_0\in L^2(\TT^d;[0,1])$, and $g\in L^2(\TT^d\times[0,T])^d$.  A \emph{weak solution} of \eqref{skel_1} is a function $\rho\in L^2([0,T];H^1(\TT^d))\cap \C([0,T];L^2(\TT^d;[0,1]))$ that satisfies, for every $t\in[0,T]$ and $\psi\in \C^\infty(\TT^d)$,
\[\int_{\TT^d}\rho(x,t)\psi(x)\dx = \int_{\TT^d}\rho_0 \psi\dx-\int_0^t\int_{\TT^d}\nabla\rho\cdot\nabla\psi\dx\ds +\int_0^t\int_{\TT^d}\sqrt{\rho(1-\rho)}g\cdot\nabla\psi\dx\ds.\]
\end{definition}

\begin{prop}\label{prop_unique}  Let $T\in(0,\infty)$, let $\rho_{0,1},\rho_{0,2}\in L^2(\TT^d;[0,1])$, and let $g\in L^2(\TT^d\times[0,T])^d$.  Let $\rho_1,\rho_2\in L^2([0,T];H^1(\TT^d))$ be weak solutions of \eqref{skel_1} in the sense of Definition~\ref{skel_def} with initial data $\rho_{0,1}, \rho_{0,2}$ and control $g$.  Then,
\[\max_{t\in[0,T]}\norm{\rho_1(\cdot,t)-\rho_2(\cdot,t)}_{L^1(\TT^d)}\leq\norm{\rho_{0,1}-\rho_{0,2}}_{L^1(\TT^d)}.\]
\end{prop}

\begin{proof}  For every $\ve\in(0,1)$ let $\eta_d^\ve$ be a standard convolution kernel on $\TT^d$ of scale $\ve$ and let $\eta_1^\ve$ be a standard convolution kernel on $\R$ of scale $\ve$.  For every $i\in\{1,2\}$ let $\rho^\ve_i = \rho^i*\eta_d^\ve$.  Let $a\colon \R\rightarrow[0,\infty)$ denote the absolute value function $a(x)=\abs{x}$ and for every $\d\in(0,1)$ let $a^\d = a*\eta_1^\d$.

Definition~\ref{skel_def} implies that, for every $i\in\{1,2\}$, as functions in $L^2(\TT^d\times[0,T])$,
\begin{equation}\label{skel_2}\partial_t \rho^\ve_i(x,t)= -\int_{\TT^d}\nabla_y\eta^\ve_d(y-x)\cdot\nabla\rho_i(y,t)+\sqrt{\rho_i(y,t)(1-\rho_i(y,t))}g(y,t)\cdot\nabla_y\eta^\ve_d(y-x)\dy.\end{equation}
Therefore, for every $\ve,\d\in(0,1)$,
\[\partial_t\int_{\TT^d}a^\d(\rho^\ve_1(x,t)-\rho^\ve_2(x,t))\dx = \int_{\TT^d}\sgn^\d(\rho^\ve_1(x,t)-\rho^\ve_2(x,t))\partial_t(\rho^\ve_1(x,t)-\rho^\ve_2(x,t))\dx,\]
where $\sgn^\d=\sgn*\eta_1^\d$ for the left-continuous sign function $\sgn\colon\R\rightarrow\{-1,1\}$.  It follows from \eqref{skel_2} that, after integrating by parts in the $x$-variable, for every $\ve,\d\in(0,1)$,
\begin{align*}
& \partial_t\int_{\TT^d}a^\d(\rho^\ve_1(x,t)-\rho^\ve_2(x,t))\dx
\\ & = -2\int_{(\TT^d)^2}\eta_1^\d(\rho^\ve_1(x,t)-\rho^\ve_2(x,t))\nabla_x(\rho^\ve_1(x,t)-\rho^\ve_2(x,t))\cdot \nabla_y(\rho_1(y,t)-\rho_2(y,t))\eta_d^\ve(x-y)\dx\dy
\\  & + \int_{(\TT^d)^2} 2\eta_1^\d(\rho^\ve_1(x,t)-\rho^\ve_2(x,t))\sqrt{\rho_1(y,t)(1-\rho_1(y,t))}\nabla_x(\rho^\ve_1(x,t)-\rho^\ve_2(x,t))\cdot g(y,t)\eta_d^\ve(x-y)\dx\dy
\\ & - \int_{(\TT^d)^2} 2\eta_1^\d(\rho^\ve_1(x,t)-\rho^\ve_2(x,t))\sqrt{\rho_2(y,t)(1-\rho_2(y,t))}\nabla_x(\rho^\ve_1(x,t)-\rho^\ve_2(x,t))\cdot g(y,t)\eta_d^\ve(x-y)\dx\dy.
\end{align*}
After passing to the limit $\ve\rightarrow 0$, using the $H^1$-regularity and boundedness of weak solutions, for every $\d\in(0,1)$,
\begin{align*}
& \partial_t\int_{\TT^d}a^\d(\rho_1(x,t)-\rho_2(x,t))\dx
\\ & = -\int_{\TT^d}2\eta_1^\d(\rho_1-\rho_2)\abs{\nabla(\rho_1-\rho_2)}^2\dx
\\ & \quad + \int_{\TT^d} 2\eta_1^\d(\rho_1-\rho_2)\Big(\sqrt{\rho_1(1-\rho_1)}-\sqrt{\rho_2(1-\rho_2)}\Big)\nabla\Big(\rho_1-\rho_2\Big)\cdot g \dx.
\end{align*}
H\"older's inequality and Young's inequality prove that, for every $\d\in(0,1)$,
\[ \partial_t\int_{\TT^d}a^\d(\rho_1(x,t)-\rho_2(x,t))\dx \leq \int_{\TT^d}\eta_1^\d(\rho_1-\rho_2)\Big(\sqrt{\rho_1(1-\rho_1)}-\sqrt{\rho_2(1-\rho_2)}\Big)^2\abs{g}^2\dx.\]
Since there exists $c\in(0,\infty)$ independent of $\d\in(0,1)$ such that $\abs{\eta_1^\d}\leq \nicefrac{c}{\d}$, the definition of the convolution kernel and the H\"older regularity of the function $x\in[0,1]\mapsto\sqrt{x(1-x)}$ proves that there exists $c\in(0,\infty)$ such that
\[\abs{\eta_1^\d(\rho_1-\rho_2)\Big(\sqrt{\rho_1(1-\rho_1)}-\sqrt{\rho_2(1-\rho_2)}\Big)^2}\leq c\mathbf{1}_{\{0<\abs{\rho_1-\rho_2}<c\d\}}.\]
Therefore, for $c\in(0,\infty)$ independent of $\d\in(0,1)$,
\[ \partial_t\int_{\TT^d}a^\d(\rho_1(x,t)-\rho_2(x,t))\dx \leq \frac{c}{2}\int_{\{0<\abs{\rho_1-\rho_2}<c\d\}}\abs{g}^2\dx.\]
The $L^2$-integrability of $g$, the definition of $a$, and the dominated convergence theorem prove that, in the sense of distributions, after passing to the limit $\d\rightarrow 0$,
\[\partial_t \int_{\TT^d}\abs{\rho_1(x,t)-\rho_2(x,t)}\dx \leq 0,\]
which completes the proof.  \end{proof}

\begin{prop}\label{prop_exist}   Let $T\in(0,\infty)$.  For every $\rho_0\in L^2(\TT^d;[0,1])$ and $g\in L^2(\TT^d\times[0,T])^d$ there exists a weak solution of \eqref{skel_1} in the sense of Definition~\ref{skel_def}.  Furthermore, for some $c\in(0,\infty)$,
\[\norm{\rho}^2_{L^\infty([0,T];L^2(\TT^d))}+\norm{\nabla\rho}_{L^2(\TT^d\times[0,T])}^2 \leq c\Big(\norm{\rho_0}^2_{L^2(\TT^d)}+\norm{g}^2_{L^2(\TT^d\times[0,T])}\Big),\]
and
\[\norm{\partial_t\rho}_{L^2([0,T];H^{-1}(\TT^d))}\leq c(\norm{\rho_0}_{L^2(\TT^d)}+\norm{g}_{L^2(\TT^d\times[0,T])^d}).\]
\end{prop}

\begin{proof}  Let $s\colon\R\rightarrow[0,\nicefrac{1}{2}]$ be defined by
\[s(x)=\sqrt{x(1-x)}\;\;\textrm{if}\;\;x\in[0,1]\;\;\textrm{and}\;\;s(x)=0\;\;\textrm{if}\;\;x\notin[0,1].\]
Let $S\colon L^2(\TT^d\times[0,T])\rightarrow L^2([0,T];H^1(\TT^d))$ denote the solution map
\[\partial_t S(\rho) = \Delta S(\rho)-\nabla\cdot (s(\rho)g)\;\;\textrm{in}\;\;\TT^d\times(0,T)\;\;\textrm{with}\;\;S(v)=\rho_0\;\;\textrm{on}\;\;\TT^d\times\{0\}.\]
The boundedness of $s$, H\"older's inequality, and Young's inequality prove that, for some $c\in(0,\infty)$,
\begin{equation}\label{skel_8}\norm{S(\rho)}_{L^2([0,T];H^1(\TT^d))} \leq c(\norm{\rho_0}_{L^2(\TT^d)}+\norm{g}_{L^2(\TT^d\times [0,T];\R^d)}),\end{equation}
and it follows from \eqref{skel_8} and the equation that, for some $c\in(0,\infty)$,
\begin{equation}\label{skel_9}\norm{\partial_tS(\rho)}_{L^2([0,T];H^{-1}(\TT^d))}\leq c(\norm{\rho_0}_{L^2(\TT^d)}+\norm{g}_{L^2(\TT^d\times [0,T];\R^d)}).\end{equation}
The compact embedding $H^1(\TT^d)\hookrightarrow L^2(\TT^d)$, \eqref{skel_8}, \eqref{skel_9}, and the Aubin--Lions--Simons lemma \cite{Aubin,pLions,Simon} prove that the image of $S$ lies in a compact subset of $L^2(\TT^d\times[0,T])$.  The Schauder fixed point theorem proves that there exists a weak solution $\rho\in L^2([0,T];H^1(\TT^d))$ of the equation
\[\partial_t\rho = \Delta\rho-\nabla\cdot (s(\rho)g)\;\;\textrm{in}\;\;\TT^d\times(0,T)\;\;\textrm{with}\;\;\rho=\rho_0\;\;\textrm{on}\;\;\TT^d\times\{0\}.\]
A repetition of the regularization argument used in the proof of Proposition~\ref{prop_unique} proves using the definition of $s$ that, for the Dirac delta-distribution $\delta_0$,
\[ \partial_t\int_{\TT^d}(\rho-1)_+\dx =-\int_{\TT^d}\delta_0(\rho-1)\abs{\nabla\rho}^2+\int_{\TT^d}\delta_0(\rho-1)s(\rho)g\dx\leq 0, \]
which, since $\rho_0\leq 1$, proves that $\rho\leq 1$ almost surely.  Similarly,
\[ \partial_t\int_{\TT^d}(\rho)_-\dx =\int_{\TT^d}\delta_0(\rho)\abs{\nabla\rho}^2-\int_{\TT^d}\delta_0(\rho)s(\rho)g\dx\geq 0, \]
which, since $\rho_0\geq 0$, proves that $\rho\geq 0$ almost surely.  We conclude that $0\leq \rho\leq 1$ almost surely, and therefore using the definition of $s$ that $\rho$ is a weak solution of
\[\partial_t\rho = \Delta\rho-\nabla\cdot (\sqrt{\rho(1-\rho)}g)\;\;\textrm{in}\;\;\TT^d\times(0,T)\;\;\textrm{with}\;\;\rho=\rho_0\;\;\textrm{on}\;\;\TT^d\times\{0\},\]
in the sense of Definition~\ref{skel_def}.  The estimate is a consequence of \eqref{skel_8} and the weak-lower semi-continuity of the Sobolev norm.  The proof of $L^2(\TT^d)$-continuity follows analogously to the conclusion of Proposition~\ref{ou_wp}, which completes the proof.  \end{proof}

\begin{prop}\label{prop_weak_strong}  Let $T\in(0,\infty)$, let $\{\rho^n_0\}_{n\in\N},\rho_0\in L^2(\TT^d;[0,1])$, and let $\{g_n\}_{n\in\N},g\in L^2(\TT^d\times[0,T])^d$.  Assume that, as $n\rightarrow\infty$, $\rho^n_0\rightharpoonup \rho_0$ weakly in $L^2(\TT^d)$ and $g_n\rightharpoonup g$ weakly in $L^2(\TT^d\times[0,T])^d$.  Then, for the solutions $\{\rho_n\}_{n\in\N}, \rho$ of \eqref{skel_1} in the sense of Definition~\ref{skel_def} with controls $\{g_n\}_{n\in\N}$ and $g$ and initial data $\rho^n_0$ and $\rho_0$, as $n\rightarrow\infty$,
\[\rho_n\rightarrow \rho\;\;\textrm{strongly in}\;\;L^2(\TT^d\times[0,T]).\]
\end{prop}

\begin{proof}  The proof is an immediate consequence of Definition~\ref{skel_def}, Proposition~\ref{prop_exist}, the Aubin-Lions-Simon Lemma \cite{Aubin,pLions,Simon}, the compact embedding of $H^1(\TT^d)$ into $L^2(\TT^d)$, and the continuous embedding of $L^2(\TT^d)$ into $H^{-1}(\TT^d)$. \end{proof}

\subsection{The large deviations principle}\label{sec_ldp_old}  The LDP is based on the weak convergence approach to large deviations of \cite{BudDupMar2008}.  For this, it is essentially necessary to prove three conditions.  The first is to prove the existence of a measurable solution map taking initial data and noise living in the space $L^2(\TT^d;[0,1])\times\C([0,\infty);(\R^d)^\infty)$ to the solution space $L^2(\O;L^2(\TT^d;[0,1]))$.  The second of these is to prove the collapse of the controlled SPDE
\begin{equation}\label{new_control}\partial_t \rho = \Delta\rho - \sqrt{\ve}\nabla\cdot (\sqrt{\rho(1-\rho)}\circ\xi^K)-\nabla\cdot(\sqrt{\rho(1-\rho)}g),\end{equation}
for arbitrary controls $g\in L^2(\TT^d\times[0,T])^d$, to the skeleton equation
\begin{equation}\label{ldp_control_skel} \partial_t\rho = \Delta\rho-\nabla\cdot(\sqrt{\rho(1-\rho)}g).\end{equation}
The final condition is to prove the compactness of the level sets of the controlled skeleton equation, in the sense that it is necessary to prove that families of solutions of \eqref{ldp_control_skel} corresponding to $L^2$-bounded families of controls $g$ are relatively compact in $L^2(\TT^d\times[0,T])$.  These are established respectively in Proposition~\ref{prop_entropy}, Proposition~\ref{prop_entropy_1}, and Proposition~\ref{ldp_skel_compact}.  The proof of the large deviation principle in Theorem~\ref{thm_new} is then a consequence of \cite[Theorem~6]{BudDupMar2008}.

\begin{prop}\label{prop_entropy}  Let $T\in(0,\infty)$, let $\{\xi^K\}_{K\in\N}$ and $\rho_0$ satisfy Assumption~\ref{def_noise_trig}, and for every $\ve\in(0,1)$ and $K\in\N$ let $\rho^{\ve,K}(\rho_0)\in L^2(\Omega\times[0,T];L^2(\TT^d;[0,1]))$ be the unique solution of \eqref{new_1} with initial condition $\rho_0$ in the sense of Definition~\ref{def_sol}.  Then there exists a measurable map
\[S_{\ve,K}\colon L^2(\TT^d;[0,1])\times \textrm{C}([0,T];(\R^d)^\infty)\rightarrow L^2(\TT^d\times[0,T];[0,1])\]
such that, for every $\rho_0\in L^2(\TT^d;[0,1])$, $\P$-a.s.,
\[S_{\ve,K}(\rho_0,(B^k,W^k)_{k\in\Z^d})=\rho^{\ve,K}(\rho_0),\]
for the Brownian motions $(B^k,W^k)_{k\in\Z^d}$ defined in Assumption~\ref{def_noise_trig}.
\end{prop}

\begin{proof}  Let $\overline{B} = (B^k,W^k)_{k\in\Z^d}\in \C([0,\infty);(\R^d)^\infty)$.  Theorems~\ref{thm_rks_unique} and \ref{thm_rks_exist} prove that for every $\rho_0\in L^2(\TT^d)$ here exists a measurable function $S^{\ve,K}_{\rho_0}\colon\textrm{C}([0,\infty);(\R^d)^\infty)\rightarrow L^2(\TT^d\times[0,T])$ such that, $\P$-a.s.,
\begin{equation}\label{tec_0}S^{\ve,K}_{\rho_0}(\overline{B})=\rho^{\ve,K}(\rho_0).\end{equation}
We will now use the separability of $L^2(\TT^d)$ and the pathwise contraction property of Theorem~\ref{thm_rks_unique}.  For a countable dense subset $\{\rho_n\}_{n\in\N}$ of $L^2(\TT^d)$.  The pathwise $L^1$-contraction of Theorem~\ref{thm_rks_unique} and the countability of the set $\N\times\N$ prove that, on a measurable subset of full probability,
\begin{equation}\label{tec_1}\sup_{t\in[0,T]}\norm{S^{\ve,K}_{\rho_n}(\overline{B})-S^{\ve,K}_{\rho_m}(\overline{B})}_{L^1(\TT^d)}\leq \norm{\rho_n-\rho_m}_{L^1(\TT^d)}\;\;\textrm{for every}\;\;n,m\in\N.\end{equation}
Since the $L^\infty$-estimate proves that $\norm{\cdot}_{L^2(\TT^d\times[0,T])}\leq \norm{\cdot}^\frac{1}{2}_{L^1(\TT^d\times[0,T])}$, it follows from the density of the $\{\rho_n\}_{n\in\N}$ with respect to the strong $L^2(\TT^d)$-norm and \eqref{tec_1} that there almost surely exists a strongly continuous $\overline{S}^{\ve,K}_{\overline{B}}\colon L^2(\TT^d)\rightarrow L^2(\TT^d\times[0,T])$ such that, for every $n\in\N$, on a subset of full probability independent of $n$,
\begin{equation}\label{tec_02} \overline{S}^{\ve,K}_{\overline{B}}(\rho_n)=S^{\ve,K}_{\rho_n}(\overline{B}).\end{equation}
It then follows from \eqref{tec_1}, \eqref{tec_02}, and a simplified version of \cite[Theorem~5.29, Corollary~5.31]{FG21} that, for every $\rho\in L^2(\TT^d)$, on a subset of full probability depending on $\rho$,
\begin{equation}\label{tec_2} \overline{S}^{\ve,K}_{\overline{B}}(\rho)=S^{\ve,K}_{\rho}(\overline{B}).\end{equation}
We define $S_{\ve,K}\colon L^2(\TT^d)\times \textrm{C}([0,T];(\R^d)^\infty)\rightarrow L^2(\TT^d\times[0,T])$ by the rule
\[S_{\ve,K}(\rho,\overline{B})=\overline{S}^{\ve,K}_{\overline{B}}(\rho),\]
from which it follows from \eqref{tec_0} that for every $\rho\in L^2(\TT^d;[0,1])$ the restriction
\begin{equation}\label{tec_3} S_{\ve,K}(\rho,\cdot)\colon\textrm{C}([0,\infty);(\R^d)^\infty)\rightarrow L^2(\TT^d\times[0,T])\;\;\textrm{is measurable,} \end{equation}
and from \eqref{tec_2} that for almost every realization of $\overline{B}$ the restriction
\begin{equation}\label{tec_4} S_{\ve,K}(\cdot,\overline{B})\colon L^2(\TT^d)\rightarrow L^2(\TT^d\times[0,T])\;\;\textrm{is strongly continuous.}\end{equation}
In combination \eqref{tec_3}, \eqref{tec_4}, and the separability of $L^2(\TT^d)$ prove that $S_{\ve,K}$ is measurable, which with \eqref{tec_0} and \eqref{tec_2} completes the proof.  \end{proof}

\begin{prop}\label{prop_entropy_1}  Let $T\in(0,\infty)$, let $\{\rho_0\}_{\ve\in(0,1)},\rho_0\in L^2(\TT^d;[0,1])$ and let $\{g^\ve\}_{\ve\in(0,1)},g$ be $\F_t$-predictable, $\R^d$-valued processes such that, as $\ve\rightarrow 0$,
\[\rho^\ve_0\rightharpoonup \rho_0\;\;\textrm{weakly in}\;\;L^2(\TT^d)\;\;\textrm{and}\;\;g^\ve\rightharpoonup g\;\;\textrm{in law,}\]
and let $\{K(\ve)\}_{\ve\in(0,1)}$ be a sequence that satisfies, as $\ve\rightarrow 0$,
\[\ve K(\ve)^{d+2}\rightarrow 0\;\;\textrm{and}\;\;K(\ve)\rightarrow \infty.\]
Then, the solutions $\rho^\ve$ of \eqref{new_control} in the sense of Definition~\ref{def_sol} with parameters $(\ve,K(\ve))$, initial data $\rho^\ve_0$, and control $g^\ve$ satisfy, as $\ve\rightarrow 0$,
\[\rho^\ve\rightharpoonup \rho\;\;\textrm{in law on}\;\;L^2(\TT^d\times[0,T]),\]
for $\rho$ the unique solution of \eqref{skel_1} in the sense of Definition~\ref{skel_def} with initial data $\rho_0$ and control $g$. \end{prop}
 
\begin{proof}  The convergence in law of the $\{g^\ve\}_{\ve\in(0,1)}$ to $g$, the boundedness of the initial data, and the estimates of Proposition~\ref{prop_ap} prove using a repetition of Proposition~\ref{prop_eta_tight} that the laws of the solutions $\{\rho^\ve\}_{\ve\in(0,1)}$ are tight on $L^2(\TT^d\times[0,T])$ and that the laws of the gradients $\{\nabla\rho^\ve\}_{\ve\in(0,1)}$ are tight on $L^2(\TT^d\times[0,T])^d$.  A simplified version of \cite[Proposition~5.27, Theorem~5.29]{FG21} based on Prokhorov's theorem and the Skorokhod representation theorem then reduce Proposition~\ref{prop_entropy_1} to the following simpler question:  suppose that, $\P$-a.s.\ along a subsequence $\ve\rightarrow 0$,
\[\rho^\ve_0\rightharpoonup\rho^\ve\;\;\textrm{weakly in}\;\;L^2(\TT^d)\;\;\textrm{and}\;\;g^\ve\rightharpoonup g\;\;\textrm{weakly in}\;\;L^2(\TT^d\times[0,T])^d,\]
and that there exists $\rho\in L^2(\O\times[0,T];H^1(\TT^d))$ such that
\[\rho^\ve\rightarrow \rho\;\;\textrm{strongly in}\;\;L^2(\TT^d\times[0,T])\;\;\textrm{and weakly in}\;\;L^2([0,T];H^1(\TT^d)),\]
and that there $\P$-a.s.\ exists a finite, nonnegative measure $q$ on $\TT^d\times[0,T]$ such that, for the kinetic measures $\{q^\ve\}_{\ve\in(0,1)}$ of the $\{\rho^\ve\}_{\ve\in(0,1)}$,
\begin{equation}\label{pe1_00}(\xi(1-\xi))^{-1}q^\ve\rightharpoonup q\;\;\textrm{weakly in}\;\;\mathcal{M}_+(\TT^d\times[0,T]),\end{equation}
in the space $\mathcal{M}_+(\TT^d\times[0,T])$ of nonnegative Radon measures.  Then, we have that $\rho$ the unique solution of \eqref{skel_1} in the sense of Definition~\ref{skel_def} with initial data $\rho_0$ and control $g$.

The essential difficulty in proving this statement is in the necessity of considering test functions that are compactly supported in the velocity variable.  In order to remove this limit and recover the classical formulation of the skeleton equation, it is necessary to use the optimal regularity estimate for the kinetic measures implicit in \eqref{pe1_00}.  To make this precise, for every $\beta\in(0,\nicefrac{1}{2})$ let $\zeta_\beta\in\C^\infty_c((0,1);[0,1])$ satisfy that $\zeta_\beta(\xi) = 0$ if $\xi\leq \beta$ or if $\xi\geq 1-\beta$ and $\zeta_\beta(\xi) = 1$ if $\xi\in[2\beta, 1-2\beta]$, and that $|\zeta_\beta'(\xi)|\leq\nicefrac{c}{\beta(1-\beta)}$ for some $c\in(0,\infty)$ independent of $\beta$.  Then, for $\psi\in \C^\infty_c(\TT^d)$ and $\ve,\beta\in(0,1)$, we have $\P$-a.s.\ that
\begin{align}\label{pe1_0}
& \int_{\TT^d}\int_\R \chi^\ve(x,\xi,s)\psi(x)\zeta_\beta(\xi) = \int_{\TT^d}\int_\R \overline{\chi}(\rho^\ve_0)\psi\zeta_\beta-\int_0^t\int_{\TT^d} \nabla\rho^\ve \cdot \nabla\psi(x)\zeta_\beta(\rho^\ve)
\\ \nonumber  & -\sqrt{\ve}\int_0^t\int_{\TT^d} \nabla\cdot\Big(\sqrt{\rho^\ve(1-\rho^\ve)}\dd\xi^F\Big)\psi(x)\zeta_\beta(\rho^\ve)-\frac{\ve N_{K(\ve)}}{8}\int_0^t\int_{\TT^d} \frac{(1-2\rho^\ve)^2}{\rho^\ve(1-\rho^\ve)}\nabla\rho^\ve\cdot\nabla\psi(x)\zeta_\beta(\rho^\ve)
\\ \nonumber  & +\int_0^t\int_{\TT^d}\sqrt{\rho^\ve(1-\rho^\ve)}g^\ve\cdot \nabla\psi(x)\zeta_\beta(\rho^\ve)+\int_0^t\int_{\TT^d}\sqrt{\rho^\ve(1-\rho^\ve)}g^\ve\cdot \nabla\rho^\ve(x)\psi(x)\zeta'_\beta(\rho^\ve)
\\ \nonumber  & -\int_0^t\int_{\TT^d}\int_\R \psi(x)\zeta_\beta'(\xi) \dd q^\ve+\frac{\ve M_{K(\ve)}}{2}\int_0^t\int_{\TT^d}\psi(x)\zeta_\beta(\rho^\ve)\rho^\ve(1-\rho^\ve).
\end{align}
We pass first to the limit $\ve\rightarrow 0$, from which it follows from the compact support of $\zeta_\beta$ on $(0,1)$, $N_{K(\ve)}\leq cK(\ve)^d$, $M_{K(\ve)}\leq cK(\ve)^{d+2}$, $\lim\nolimits_{\ve\rightarrow 0}\ve K(\ve)^{d+2}=0$, and the estimates of Proposition~\ref{prop_ap} that, for every $t\in[0,T]$,
\begin{equation}\label{pe1_1}\limsup_{\ve\rightarrow 0}\abs{\frac{\ve N_{K(\ve)}}{8}\int_0^t\int_{\TT^d} \frac{(1-2\rho^\ve)^2}{\rho^\ve(1-\rho^\ve)}\nabla\rho^\ve\cdot\nabla\psi(x)\zeta_\beta(\rho^\ve)}+\abs{\frac{\ve M_{K(\ve)}}{2}\int_0^t\int_{\TT^d}\psi(x)\zeta_\beta(\rho^\ve)\rho^\ve(1-\rho^\ve)}=0.\end{equation}
For the martingale term, the Burkh\"older--Davis--Gundy inequality proves that, for some $c\in(0,\infty)$ independent of $\ve$ but depending on $\psi$ and $\beta$,
\begin{align*}
& \E\Big[\sup_{t\in[0,T]}\abs{\sqrt{\ve}\int_0^t\int_{\TT^d} \nabla\cdot\Big(\sqrt{\rho^\ve(1-\rho^\ve)}\dd\xi^F\Big)\psi(x)\zeta_\beta(\rho^\ve)}\Big]
\\ & \leq c\E\Big[\Big(\ve N_{K(\ve)}\int_0^T\int_{\TT^d}\abs{\nabla\rho^\ve}^2+\ve M_{K(\ve)} \Big)\Big]\leq c \sqrt{\ve M_{K(\ve)}}.
\end{align*}
We therefore have that, after passing to a further subsequence still denoted $\ve\rightarrow 0$, $\P$-a.s.,
\begin{equation}\label{pe1_2}\limsup_{\ve\rightarrow 0}\sup_{t\in[0,T]}\abs{\sqrt{\ve}\int_0^t\int_{\TT^d} \nabla\cdot\Big(\sqrt{\rho^\ve(1-\rho^\ve)}\dd\xi^F\Big)\psi(x)\zeta_\beta(\rho^\ve)}=0.\end{equation}
We then observe using the identity, for every $t\in[0,T]$ and $\ve\in(0,1)$,
\[\int_{\TT^d}\overline{\chi}(x,\xi,t)\zeta_\beta(\xi)\psi(x)= \int_{\TT^d}\rho^\ve(x,t)\psi(x)+\int_{\TT^d}\int_{\TT^d}\int_\R\chi^\ve(x,\xi,t)\psi(x)(\zeta_\beta(\xi)-1),\]
that, for $c\in(0,\infty)$ independent of $\ve$ and $t\in[0,T]$,
\begin{equation}\label{pe1_20}\abs{\int_{\TT^d}\overline{\chi}(x,\xi,t)\zeta_\beta(\xi)\psi(x)-\int_{\TT^d}\rho^\ve(x,t)\psi(x)}\leq c\beta.\end{equation}
Returning to \eqref{pe1_0}, it follows from \eqref{pe1_1}, \eqref{pe1_2}, \eqref{pe1_20}, the weak convergence of $g^\ve$ to $g$, the weak convergence of $\rho^\ve_0$ to $\rho_0$, the strong convergence of $\rho^\ve$ to $\rho$, and the weak convergence of $\nabla\rho^\ve$ to $\nabla\rho$ that, for the kinetic function $\chi$ of $\rho$, $\P$-a.s.\ for almost every $t\in[0,T]$, for some $c\in(0,\infty)$ independent of $\ve,\beta\in(0,1)$,
\begin{align}\label{pe1_3}
& \abs{\int_{\TT^d}\rho(x,t)\psi(x)-\int_{\TT^d}\rho_0(x)\psi(x)-\int_0^t\int_{\TT^d}\nabla\rho\cdot\nabla\psi\zeta_\beta(\rho)-\int_0^t\int_{\TT^d}\sqrt{\rho(1-\rho)}g\cdot\nabla\psi(x)\zeta_\beta(\rho)}
\\ \nonumber & \leq c\b+\limsup_{\ve\rightarrow 0}\Big(\int_0^T\int_{\TT^d}\int_\R\psi(x)\zeta'_\beta(\xi)\dd q^\ve +\int_0^T\int_{\TT^d}\sqrt{\rho^\ve(1-\rho^\ve)}g^\ve\cdot\nabla\rho^\ve(x)\psi(x)\zeta'_\beta(\rho^\ve) \Big).
\end{align}
To treat the righthand side of \eqref{pe1_3}, we first observe that
\begin{align}\label{pe1_4}
\limsup_{\ve\rightarrow 0}\int_0^T\int_{\TT^d}\psi(x)\zeta'_\beta(\xi)\dd q^\ve & \leq \limsup_{\ve\rightarrow 0}\norm{\psi}_{L^\infty(\TT^d)}\int_0^T\int_{\TT^d}\int_\R (\xi(1-\xi))\zeta'_\beta(\xi)\frac{\dd q^\ve}{\xi(1-\xi)}
\\ \nonumber & = \norm{\psi}_{L^\infty(\TT^d)}\int_0^T\int_{\TT^d}\int_\R(\xi(1-\xi))\zeta'_\beta(\xi)\dd q.
\end{align}
We then observe using the bound on $\zeta'_\beta$, the finiteness of $q$, and the dominated convergence theorem that
\begin{equation}\label{pe1_5}\lim_{\beta\rightarrow 0}\abs{\int_0^T\int_{\TT^d}\int_\R(\xi(1-\xi))\zeta'_\beta(\xi)\dd q} = 0.\end{equation}
For the second term on the righthand side of \eqref{pe1_3}, it follows from H\"older's inequality, the uniform $L^2$-boundedness of the $g^\ve$, and the regularity of the kinetic measure in Definition~\ref{def_sol} that, for some $c\in(0,\infty)$,
\begin{align}\label{pe1_6}
& \limsup_{\ve\rightarrow 0}\abs{\int_0^T\int_{\TT^d}\sqrt{\rho^\ve(1-\rho^\ve)}g^\ve\cdot\nabla\rho^\ve(x)\psi(x)\zeta'_\beta(\rho^\ve)}
\\ \nonumber & \leq \limsup_{\ve\rightarrow 0}\norm{\psi}_{L^\infty(\TT^d)}\norm{g^\ve}_{L^2(\TT^d\times[0,T])^d}\Big(\int_0^T\int_{\TT^d}(\rho^\ve(1-\rho^\ve))\abs{\zeta'_\beta(\rho^\ve)}^2\abs{\nabla\rho^\ve}^2\Big)^\frac{1}{2}
\\ \nonumber & \leq \limsup_{\ve\rightarrow 0}\norm{\psi}_{L^\infty(\TT^d)}\norm{g^\ve}_{L^2(\TT^d\times[0,T])^d}\Big(\int_0^T\int_{\TT^d}\int_\R(\xi(1-\xi))^2\abs{\zeta'_\beta(\xi)}^2\frac{\dd q^\ve}{\xi(1-\xi)}\Big)^\frac{1}{2}
\\ \nonumber & \leq c\norm{\psi}_{L^\infty(\TT^d)}\Big(\int_0^T\int_{\TT^d}\int_\R(\xi(1-\xi))^2\abs{\zeta'_\beta(\xi)}^2\dd q\Big)^\frac{1}{2}.
\end{align}
The bounds of $\zeta'_\beta$, the finiteness of $q$, and the dominated convergence theorem then prove that
\begin{equation}\label{pe1_7} \lim_{\beta\rightarrow 0}\Big(\int_0^T\int_{\TT^d}\int_\R(\xi(1-\xi))^2\abs{\zeta'_\beta(\xi)}^2\dd q\Big)^\frac{1}{2}=0.\end{equation}
Returning to \eqref{pe1_3}, after passing to the limit $\beta\rightarrow 0$ using \eqref{pe1_4}, \eqref{pe1_5}, \eqref{pe1_6}, \eqref{pe1_7}, and the equality, for every $t\in[0,T]$,
\[\lim_{\beta\rightarrow 0}\int_{\TT^d}\int_\R\chi(x,\xi,s)\zeta_\beta(\rho)\psi(x)\dx\dxi = \int_{\TT^d}\rho(x,s)\psi(x)\dx,\]
we have $\P$-a.s.\ that, for every $t\in[0,T]$,
\[\int_{\TT^d}\rho(x,t)\psi(x) = \int_{\TT^d}\rho_0(x)\psi(x)+\int_0^t\int_{\TT^d}\sqrt{\rho(1-\rho)}g\cdot\nabla\psi(x).\]
The claim then follows from the separability of $\C^\infty_c(\TT^d)$ in the strong $H^1(\TT^d)$-metric, which completes the proof. \end{proof}

\begin{prop}\label{ldp_skel_compact}  Let $T,M\in(0,\infty)$ and let $\mathcal{S}_M\subseteq L^2(\TT^d\times[0,T])$ be defined by
\[\mathcal{S}_M=\{\rho\colon \exists\; g\in B_M(L^2(\TT^d\times[0,T])^d), \rho_0\in L^2(\TT^d)\;\textrm{such that $\rho$ solves \eqref{skel_1}}\},\]
for $B_M(L^2(\TT^d\times[0,T])^d)$ the ball of radius $M$ in $L^2(\TT^d\times[0,T])^d$.  Then $\mathcal{S}_M$ is compact in the strong topology of $L^2(\TT^d\times[0,T])$.

\begin{proof}  The proof is a consequence of the estimates of Proposition~\ref{prop_exist}, the uniform $L^2$-boundedness of the controls, and the Aubin--Lions--Simons lemma \cite{Aubin,pLions,Simon}. \end{proof}

\end{prop}

\begin{thm}\label{thm_new} Let $T\in(0,\infty)$ and assume that $\{K(\ve)\}_{\ve\in(0,1)}$ satisfy, as $\ve\rightarrow 0$,
\[\ve K(\ve)^{d+2}\rightarrow 0,\;\;\;\textrm{and}\;\;K(\ve)\rightarrow \infty.\]
Then the rate functions $\{I_{\rho_0}\}_{\rho_0\in L^2(\TT^d)}$  defined in \eqref{ldp_1} are good rate functions with compact level sets on compact sets.  For every $\rho_0\in L^2(\TT^d;[0,1])$ the solutions $\{\rho^\ve(\rho_0)\}_{\ve\in(0,1)}$ of \eqref{new_1} satisfy a large deviations principle with rate function $I_{\rho_0}$ on $L^2(\TT^d\times[0,T])$.  Furthermore, the solutions satisfy a uniform large deviations principle with respect to weakly $L^2(\TT^d;[0,1])$-compact subsets of $L^2(\TT^d;[0,1])$ on $L^2(\TT^d\times[0,T])$.
\end{thm}

\begin{proof}  The statement is now a consequence of \cite[Theorem~6]{BudDupMar2008}, Proposition~\ref{prop_entropy}, Proposition~\ref{prop_entropy_1}, Proposition~\ref{ldp_skel_compact}, and the equivalence of uniform Laplace and large deviations principles with respect to weakly compact subsets of the initial data \cite[Theorem~4.3]{BudDupSal}.  \end{proof}

\section*{Acknowledgements}  The first author acknowledges financial support from the EPSRC through the EPSRC Early Career Fellowship EP/V027824/1.  The third author acknowledges financial support by the DFG through the CRC 1283 ``Taming uncertainty and profiting from randomness and low regularity in analysis, stochastics and their applications.''

\bibliography{SSEPLDP}
\bibliographystyle{plain}

\end{document}